\documentclass[11.5pt]{article}

\usepackage[centertags]{amsmath}
\usepackage{amsfonts}
\usepackage{amssymb,amsmath,amsfonts}
\usepackage{booktabs}
\usepackage{array, color}
\usepackage{placeins}
\usepackage{float}
\usepackage{relsize}
\usepackage{lipsum}
\usepackage{graphicx}
\usepackage{caption}
\usepackage{subcaption}
\usepackage{float}
\usepackage{cite}
\usepackage{derivative}
\usepackage{multirow}
\usepackage{amssymb,amsmath,amsfonts,fancyhdr}
\usepackage{amsmath}
\usepackage{bigints}
\numberwithin{equation}{section}
\usepackage{amsthm}
\usepackage{mathtools}

\usepackage[utf8]{inputenc}
\usepackage[english]{babel}
\newtheorem{theorem}{Theorem}[section]
\newtheorem{corollary}{Corollary}[theorem]

\newtheorem{lemma}{Lemma}[section]
\usepackage{bm}
\usepackage{amssymb}

\usepackage{booktabs}
\usepackage{color}
\usepackage{subfloat}
\usepackage{hyperref}
\usepackage{multirow}

\usepackage{lscape}
\usepackage{natbib}
\makeatletter
\newcommand{\thickhline}{%
    \noalign {\ifnum 0=`}\fi \hrule height 1pt
    \futurelet \reserved@a \@xhline
}
\newcolumntype{"}{@{\hskip\tabcolsep\vrule width 1pt\hskip\tabcolsep}}
\makeatother
\makeatletter \oddsidemargin  -.1in \evensidemargin -.1in
\newcommand{\doublespacing}{\let\CS=\@currsize
 \renewcommand{\baselinestretch}{1.05}\tiny\CS}
\begin{document}
\newcommand{\bea}{\begin{eqnarray}}
\newcommand{\eea}{\end{eqnarray}}
\newcommand{\nn}{\nonumber}
\newcommand{\bee}{\begin{eqnarray*}}
\newcommand{\eee}{\end{eqnarray*}}
\newcommand{\lb}{\label}
\newcommand{\nii}{\noindent}
\newcommand{\ii}{\indent}
\newtheorem{thm}{Theorem}[section]
\newtheorem{lem}{Lemma}[section]
\newtheorem{rem}{Remark}[section]
\theoremstyle{remark}
\renewcommand{\theequation}{\thesection.\arabic{equation}}
\vspace{5cm}
\title{\bf Shrinkage Estimators Dominating Some Naive Estimators of the Selected Entropy}
\author{ Masihuddin$^{a}$
~and Neeraj Misra$^{b}$}
\date{}
\maketitle \noindent {\it $^{a,b}$ Department of Mathematics \& Statistics, Indian
Institute of Technology Kanpur, Kanpur-208016, Uttar Pradesh, India} 

\newcommand{\oddhead}{Estimation of the Selected Entropy}
\renewcommand{\@oddhead}
{\hspace*{-3pt}\raisebox{-3pt}[\headheight][0pt]
{\vbox{\hbox to \textwidth
{\hfill\oddhead}\vskip8pt}}}
\vspace*{0.05in}
\noindent {\bf Abstract}:     
  Consider two populations characterized by independent random variables $X_1$ and $X_2$ such that $X_i, i=1,2,$ follows a gamma distribution with an unknown scale parameter $\theta_i>0$, and known shape parameter $\alpha >0$ (the same shape parameter for both the populations). Here $(X_1,X_2)$ may be an appropriate minimal sufficient statistic based on independent random samples from the two populations. The population associated with the larger (smaller) Shannon entropy is referred to as the ``worse" (``better") population. For the goal of selecting the worse (better) population, a natural selection rule is the one that selects the population corresponding to $\max\{X_1,X_2\} ~(\min\{X_1,X_2\})$ as the worse (better) population. This natural selection rule is known to possess several optimum properties. We consider the problem of estimating the Shannon entropy of the population selected using the natural selection rule (to be referred to as the selected entropy) under the mean squared error criterion.  In order to improve upon various naive estimators of the selected entropy, we derive a class of shrinkage estimators that shrink various naive estimators towards the central entropy. For this purpose, we first consider a class of naive estimators comprising linear, scale and permutation equivariant estimators and identify optimum estimators within this class. The class of naive estimators considered by  us contains three natural plug-in estimators. To further improve upon the optimum naive estimators, we consider a general class of equivariant estimators and obtain dominating shrinkage estimators. We also present a simulation study on the performances of various competing estimators.  A real data analysis is also reported to illustrate the applicability of proposed estimators.  \vspace{2mm}\\
\noindent {\it AMS 2010 SUBJECT CLASSIFICATIONS:} 62F07 · 62F10 · 62C20 ·94A17\\ 

\newcommand\blfootnote[1]{%
	\begingroup
	\renewcommand\thefootnote{}\footnote{#1}%
	\addtocounter{footnote}{-1}%
	\endgroup
}\blfootnote{Corresponding author(s): masih.iitk@gmail.com}
\section{Introduction}
In 1948, Claude Shannon in his seminal paper (\cite{shannon1948l}) introduced the notion of entropy which paved the way for a separate field of research, called Information Theory. The notion of entropy has been used in various fields of science, engineering, economics and medicine to describe uncertainty. In the field of Statistics, entropy measures the uncertainty associated with a random variable (rv). Let $X_1$ and $X_2$ be two rvs with the Lebesgue probability density functions (pdfs) $f_{\theta_1}(\cdot)$ and $f_{\theta_2}(\cdot)$, respectively, where $\theta_i \in \Omega \subset \mathbb{R}$, $i=1,2;$ here $\mathbb{R}$ denotes the real line and $\Omega$ denotes the common parameter space of $\theta_i$s. The Shannon entropy of $X_i$  (or of the associated probability distribution) is defined by: 
\begin{equation*}
H(\theta_i)=\mathbb{E}_{\theta_i}\left(-\ln f_{\theta_i}(X_i) \right),~ \theta_i \in \Omega,~i=1,2.
\end{equation*}
For two populations $\Pi_1$ and $\Pi_2$, characterized by random variables $X_1$ and $X_2$, the one corresponding to the larger (smaller) entropy $\max\{H(\theta_1), H(\theta_2)\}$ $\left(\min\{H(\theta_1),H(\theta_2)\}\right)$ is considered to be more (less) volatile (or chaotic) than the other. For example, if $\Pi_1$ and $\Pi_2$ are populations of daily stock prices of two stocks, then the stock corresponding to $\max\{H(\theta_1),H(\theta_2)\}$ $\left(\min\{H(\theta_1),H(\theta_2)\}\right)$ is considered to be more riskier (safer) than the other. We call the population corresponding to riskier (safer) stock the ``worse" (``better") population.
In such a situation, it may be of interest to select the worse (better) population and also to have an estimate of the entropy (volatility index) of the selected population. Since $H(\theta_i)s$ are unknown, they may be estimated by appropriate estimators $\widehat{H_1}(\underline{X})$ and $\widehat{H_2}(\underline{X})$  based on independent random samples from the two populations, and the population corresponding to  $\max\{\widehat{H_1}(\underline{X}) ,\widehat{H_2}(\underline{X})\}$ $\left(\min\{\widehat{H_1}(\underline{X}) ,\widehat{H_2}(\underline{X})\}\right)$ may be identified as the worse (better) population. We call such a selection rule the ``natural selection rule". Under the assumption that the underlying pdfs $\{f_{\theta}:\theta \in \Omega\}$ have the monotone likelihood ratio property, the natural selection rule is known to posses several optimum properties. After the target population (riskier or safer stock) has been selected, it may be desired to have an estimate of the volatility index (Shannon entropy) of the selected population. In the statistical literature, such estimation problems are known as ``estimation after selection problems". Because of their applications in various fields of science, engineering, economics and medicine these estimation problems have attracted the attention of many eminent researchers and there exists a vast literature on these problems. Some of the key contributions are due to \cite{sarkadi1967estimation}, \cite{165}, \cite{hsieh1981estimating}, \cite{cohen1982estimating}, \cite{sackrowitz1986evaluating},  \cite{vellannals1992inadmissibility}, \cite{hwang1993empirical}, \cite{vellaisamy1993umvu}, \cite{misra1994estimation}, \cite{qomi2015admissibility} and \cite{arshad2015estimation}. \par 
Under the set up of $X_1$ and $X_2$ following gamma distributions with unknown scale parameters and having a common known shape parameter, we study this problem and obtain shrinkage estimators that shrink some naive estimators of the selected entropy towards the central entropy.
Let $\Pi_1$ and $\Pi_2$ be two populations such that the observations from the population $\Pi_i, i=1,2$, follow a gamma distribution with an unknown scale parameter $\theta_i ~(>0)$ and known shape parameter $\beta ~(>0)$. Here the gamma distributions corresponding to both the populations have the same shape parameter $\beta$. Let $X_{i1},X_{i2}, \cdots, X_{in}$ be a random sample of size $n$ drawn from the population $\Pi_i$  and $X_i=\sum_{j=1}^{n}X_{ij}$, $i=1,2$. Assume that the two random samples are mutually independent. Then $\underline{X} = (X_1,X_2)$ is a complete-sufficient (and, hence minimal sufficient) statistic for $(\theta_1,\theta_2) \in \Theta=(0,\infty) \times (0, \infty)$, $X_1$ and $X_2$  are independently distributed, and $X_i,~i=1,2,$ follows a gamma distribution with an unknown scale parameter $\theta_i$ and known shape parameter $\alpha=n\beta$. The pdf of $X_i$ is given by
\begin{equation*}
f(x|\theta_i,\alpha)=\frac{1}{\Gamma(\alpha) {\theta_i}^{\alpha}}x^{\alpha -1}e^{-\frac{x}{\theta_i}}~;~ x>0,~\alpha>0,~\theta_i>0,~i=1,2,
\end{equation*}
where, $\Gamma(\alpha)$ denotes the usual gamma function.
The Shannon entropy of the population $\Pi_i ~(i=1,2)$ is given by 
\begin{align*}
H(\theta_i)&=\mathbb{E}_{\theta_i}\left(-\ln f(X_i|\theta_i,\alpha) \right)\\
&=\ln \theta_i+\alpha+ \ln(\Gamma(\alpha))+(1-\alpha)\psi(\alpha),~i=1,2,
\end{align*}
where, $\psi(\alpha)=\frac{\Gamma'(\alpha)}{\Gamma(\alpha)}$, $\alpha >0$, denotes the digamma function; here $\Gamma'(\alpha)$ denotes the derivative of the gamma function $\Gamma(\alpha)$. Call the population corresponding to $\max\{H(\theta_1),H(\theta_2)\}$ $\left(\min\{H(\theta_1),H(\theta_2)\}\right)$ the ``worse" (``better") population. Equivalently the population corresponding to $\max\{\theta_1,\theta_2\}$ $\left(\min\{\theta_1,\theta_2\}\right)$  is called the ``worse" (``better") population. Consider the natural selection rule that selects the population corresponding to  $Z_2=\max\{X_1,X_2\}~( Z_1=\min\{X_1,X_2\})$ as the worse (better) population. The natural selection rule is known to possess several desirable optimum properties (see \cite{bahadur1952impartial}, \cite{eaton1967some}, \cite{misra1994non} and \cite{misra2014selecting}). We first discuss estimation of the selected worse entropy (entropy of the population identified as the worse population). Let $S \equiv S(\underline{X})$ denote the index of the selected population, $i.e., S = i, \text{if},  X_i = Z_2, i = 1, 2;$ here $\underline{X}=(X_1,X_2)$.
Our goal is to estimate the Shannon entropy of the selected population, which is equivalent to the estimation of  \\
\begin{equation}
\label{S1.E1}
H_S(\underline{\theta})=\begin{cases}
\ln \theta_1, & \text{if ${X}_1  \geq {X}_2$ }\\
\ln \theta_2, & \text{if ${X}_1 < {X}_2$}
\end{cases}
=\ln \theta_1 I(X_1 \geq X_2)+\ln \theta_2 I(X_1 < X_2),
\end{equation}
where, $\underline{\theta}=(\theta_1,\theta_2)$ and, for any event $A$, $I(A)$ denotes its indicator function. We consider the squared error loss function, given by
\begin{equation}\label{S1.E2}
L(\underline{\theta}, a) = \left(a-H_S(\underline{\theta})\right)^2,~
\underline{\theta} \in \Theta, ~ a \in \mathcal{A},
\end{equation}
where, $\Theta=(0,\infty)\times(0,\infty)$ is the parameter space and  $\mathcal{A}=\mathbb{R}$ is the action space. \par 
A natural estimator of $H_S(\underline{\theta})$ can be obtained by plugging in suitable estimators of $\theta_1$ and $\theta_2$ (or $\ln \theta_1$ and $\ln \theta_2$) in the expression of $H_S(\underline{\theta})$, for $\theta_1$ and $\theta_2$ (or $\ln \theta_1$ and $\ln \theta_2$). Note that, $\frac{X_i}{\alpha}$ and $\frac{X_i}{\alpha+1}$, respectively, are the maximum likelihood estimator (mle) and the best scale equivariant estimators of $\theta_i$, i=1,2. Thus two natural estimators of $H_S(\underline{\theta})$ are $\widehat{H}^{(1)}_{N_1}(\underline{X})=H_S\left(\frac{X_1}{\alpha},\frac{X_2}{\alpha}\right) =\ln Z_2-\ln(\alpha)$ and $\widehat{H}^{(1)}_{N_2}(\underline{X})=H_S\left(\frac{X_1}{\alpha+1},\frac{X_2}{\alpha+1}\right) =\ln Z_2-\ln(\alpha+1)$. Also, note that, $\ln(X_i)-\psi(\alpha)$ is the best scale equivariant estimator (also the uniformly minimum variance unbiased estimator) of $\ln \theta_i,~i=1,2$, under the squared error loss function. Thus, another natural estimator of $H_S(\underline{\theta})$ is $\widehat{H}^{(1)}_{N_3}(\underline{X})=\ln Z_2-\psi(\alpha)$. Guided by forms of natural estimators $\widehat{H}^{(1)}_{N_i}(\underline{X}),~i=1,2,3$, we consider a class $\mathcal{K}_1=\left\{\delta_{c}(\cdot):c \in \mathbb{R}\right\}$ of estimators, where $\delta_{c}(\underline{X})=\ln Z_2-c, c \in \mathbb{R}$. We call the class $\mathcal{K}_1$, the class of naive estimators of  $H_S(\underline{\theta})$. Interestingly, the class of naive estimators $\mathcal{K}_1$ also contains a class of generalized Bayes estimators of $H_S(\underline{\theta})$ under the squared error loss function \eqref{S1.E2}, with respect to a class of improper prior distributions (that also contains the Jeffreys non-informative prior distribution). It will be meaningful to identify optimum estimators within class $\mathcal{K}_1$ of naive estimators. To deal with this issue, we consider characterizing estimators that are admissible (or inadmissible) within class $\mathcal{K}_1$ of naive estimators. We will also find naive estimators dominating over the inadmissible naive estimators in class $\mathcal{K}_1$. \par 
The given estimation problem is invariant under the scale group of transformations $\mathcal{G}_a=\{g_a:g_a(\underline{x})=a\underline{x}, ~a>0\}$ and also under the group of permutations $\mathcal{G}_p=\{h_1, h_2\}$, where $\underline{x}=(x_1,x_2)$, $a\underline{x}=(ax_1,ax_2)$, $h_{1}(x_{1}, x_{2})=(x_1,x_2)$ and $ h_{2}(x_{1}, x_{2})=(x_2,x_1); (x_{1}, x_{2}) \in (0,\infty) \times (0,\infty)$.
Under the group of transformations $\mathcal{G}_a$, $\underline{X}$ $\rightarrow a\underline{X}$,~ $\underline{\theta}$ $\rightarrow a\underline{\theta}$,~~so that $\ln \theta_i \rightarrow \ln \theta_i +\ln a$, $i=1,2$, and
$H_S(\underline{\theta}) \rightarrow H_S(\underline{\theta}) +\ln a$. Under the transformation $(x_1,x_2) \to (x_2,x_1)$, $(\theta_1,\theta_2) \rightarrow (\theta_2,\theta_1)$ and $H_S(\underline{\theta}) \rightarrow H_S(\underline{\theta})$.
Therefore, it is natural to consider  estimators $\delta$ satisfying $\delta(aX_1,aX_2)= \delta(X_1, X_2) + \ln a$,  $ \forall a > 0 $, and $\delta(X_1, X_2)=\delta(X_2, X_1)$. Any such estimator will have the following form
\begin{align} \label{eq:1.3}
\delta_{\Phi}(\underline{X}) &= \ln Z_2 -\Phi(T), 
\end{align}
for some function $\Phi(\cdot)$ defined on $\left(0,1\right]$, $T=\frac{Z_1}{Z_2}$. An estimator of the type $(1.3)$ will be called a scale and permutation equivariant estimator of $H_S(\underline{\theta})$. Let $\mathcal{K}_2$ denote the class of all scale and permutation equivariant estimators of the type $(1.3)$. Clearly, $\mathcal{K}_1 \subseteq \mathcal{K}_2$. We observe that, for an equivariant estimator $\delta_{\Phi}$ $\in$ $\mathcal{K}_2$, the risk function (mean squared error) $R(\underline{\theta}, \delta_{\Phi})=\mathbb{E}_{\underline{\theta}}\left(\delta_{\Phi}(\underline{X})-H_S(\underline{\theta})\right)^2$ depends on $\underline{\theta}$ through $\mu=\frac{\mu_2}{\mu_1}$,  where, $\mu_1=\min\{\theta_1,\theta_2\}$, $\mu_2=\max\{\theta_1,\theta_2\}$ and $\mu \geq 1$. Therefore, for notational simplicity we denote $R(\underline{\theta}, \delta)$ by $R_{\mu}(\delta)$, $\mu\geq 1$. \par 
In the literature on decision theoretic estimation problems, notion of shrinkage has been widely used to find improvements over naive estimators under different settings. A shrinkage estimate of a naive estimate is an estimate obtained by shrinking a raw extreme estimate towards a central value. In many situations developing shrinkage estimators of naive estimators result in better estimators. Following seminal works of \cite{stein1}, \cite{stein2}, \cite{stein3}, \cite{jamesstein1} and \cite{brewster1974improving}, several researchers have obtained shrinkage estimators dominating naive estimators in different settings. Since $\mathcal{K}_1 \subseteq \mathcal{K}_2$, in order to find shrinkage estimators dominating naive estimators in class $\mathcal{K}_1,$ we explore the class $\mathcal{K}_2$ of scale and permutation equivariant estimators. 
Note that a naive estimator $\delta_{{c}} \in \mathcal{K}_1$ is a plug-in estimator with plug-ins for $\ln \theta_1$ and $\ln \theta_2$ as $\ln X_1 - c$ and $\ln X_2 - c$, respectively. If the data pretends that $\ln \theta_1$ and $\ln \theta_2$, or equivalently $\theta_1$ and $\theta_2$, are close (evidenced through large observed value of $\frac{Z_1}{Z_2}$), then, under the notion of shrinkage, it may be appropriate to shrink the naive estimator $\delta_{{c}}(X_1, X_2)$ towards a plug-in estimate corresponding to common scale parameter  $\theta$ ($\theta_1$=$\theta_2$=$\theta$). This leads to considering shrinkage estimators of the type 
 \begin{equation*}
 \label{S1.E1}
\delta_{{c},d}(\underline{X})=\begin{cases}
 \ln Z_2 -c , & \text{if $\frac{Z_1}{Z_2}  < d $ }\\
 \ln (X_1+X_2)-c - \ln (1+d), & \text{if $\frac{Z_1}{Z_2}  \geq  d$}
 \end{cases}, ~c\in \mathbb{R}, ~d >0.
 \end{equation*}
Here the choice $c - \ln (1+d)$ attached to $\ln (X_1+X_2)$ ensures continuity (smoothness) of $\delta_{{c},d}$. Following the ideas of \cite{stein2} and \cite{brewster1974improving}, we derive shrinkage estimators of $H_S(\underline{\theta})$ dominating over some optimum naive estimators (including some generalized Bayes estimators) belonging to class $\mathcal{K}_1$. \par 
The rest of the paper is  organized as follows: Section 2 deals with estimation of the Shannon entropy of the worse selected population, under the mean squared error criterion. Here we derive a class of generalized Bayes estimator and observe that they belong to the class $\mathcal{K}_1$ of naive estimators. Subsequently, we find optimum estimators within the class $\mathcal{K}_1$ of naive estimators. Following the ideas of \cite{stein2} and \cite{brewster1974improving}, we prove a general result for obtaining improvements over arbitrary scale and permutation equivariant estimators belonging to class $\mathcal{K}_2$. As a consequence of this general result, we obtain shrinkage estimators improving upon several optimum naive estimators belonging to class $\mathcal{K}_1$. In Section 3, analogous results are obtained for the problem of estimating the selected better entropy. In Section 4, we report a numerical study to assess performances of various competing estimators. Finally, to illustrate applicability of various optimum estimators, an analysis of real data set is provided in Section 5 of the paper.


\section{Estimation of Entropy of the Worse Selected Population} 
\par The following lemma will be useful in obtaining various findings of the paper.
\begin{lemma}
	Let $X \sim $ Gamma($\alpha$, $\beta$) (gamma distribution with scale parameter $\beta >0$ and shape parameter $\alpha>0$), so that the pdf of X is, 
	\begin{equation*}
	f(x|\alpha,\beta)=\frac{1}{\beta^{\alpha} \Gamma(\alpha)}e^{-\frac{x}{\beta}}x^{\alpha-1}~; ~~x>0,\alpha>0,\beta>0.
	\end{equation*}
Let $c_1(\alpha)=\psi(\alpha)$ and $c_2(\alpha)=2\displaystyle{\int_{0}^{\infty}}\ln(z)G_{\alpha}\left( z\right)g_{\alpha}(z)dz,~ \alpha >0$, where, $G_{\alpha}(\cdot)$ and $g_{\alpha}(\cdot)$ are the distribution function(df) and the pdf, respectively, of $Gamma(\alpha,1)$. Then, 
	\begin{itemize}
		\item[(i)] 	$\mathbb{E}(\ln X)= \ln \beta + \psi(\alpha)$,
		\item [(ii)] $\psi(\alpha) < \ln \alpha $.
		\item [(iii)] 
		$ c_1(\alpha) < \psi(2\alpha)-\ln 2 < c_2(\alpha) < \psi(2\alpha).$
	\end{itemize}
\end{lemma}
\begin{proof}
	The proof of the first assertion above is straightforward and the second assertion follows using Jensen's inequality on $\psi(\alpha)=\mathbb{E}\left(\ln \frac{X}{\beta}\right)$.\\
	(iii) For proving the third assertion, let $Z_{1,\alpha}$ and $Z_{2,\alpha}$ be independent and identically distributed (iid) as $Gamma(\alpha,1)$. Further let  $Z_{2\alpha} \sim Gamma(2\alpha,1),$ so that $Z_{1,\alpha}+Z_{2,\alpha} \stackrel{d}{=} Z_{2\alpha}$ ( where $\stackrel{d}{=}$ means equality in distribution ), and 
	\begin{align*}
	c_2(\alpha)=2\displaystyle{\int_{0}^{\infty}}\ln(z)G_{\alpha}\left( z\right)g_{\alpha}(z)dz&= \mathbb{E}\left(\ln (\max\{Z_{1,\alpha}, Z_{2,\alpha}\})\right)\\&<  \mathbb{E}\left(\ln (Z_{1,\alpha}+ Z_{2,\alpha})\right)\\
	&=\mathbb{E}\left(\ln (Z_{2\alpha})\right)\\
	&=\psi(2\alpha).
	\end{align*}
Also 
\begin{align*}
c_2(\alpha)-\psi(2\alpha)&=\mathbb{E}\left[ \ln \left( \max\{Z_{1,\alpha},Z_{2,\alpha}\}\right)-\ln \left(Z_{1,\alpha}+Z_{2,\alpha}\right)\right]\\
&=\mathbb{E}\left[ \ln \left(\frac{\max\{Z_{1,\alpha},Z_{2,\alpha}\}}{Z_{1,\alpha}+Z_{2,\alpha}}\right)\right]\\
&>\ln \left(\frac{1}{2}\right)~~ \left(\text{as}~ \max\{Z_{1,\alpha},Z_{2,\alpha}\} > \frac{Z_{1,\alpha}+Z_{2,\alpha}}{2}, \text{w.p.} 1\right)\\
\implies c_2(\alpha)&> \psi(2\alpha)-\ln(2).
\end{align*}
Finally
\begin{align*}
\psi(2\alpha)-c_1(\alpha)&=\mathbb{E} \left[ \ln \left(Z_{1,\alpha}+Z_{2,\alpha}\right)-\ln \left(Z_{1,\alpha}\right)\right]\\
&=\mathbb{E}\left[-\ln \frac{Z_{1,\alpha}}{Z_{1,\alpha}+Z_{2,\alpha}}\right]\\
&>-\ln \left[\mathbb{E} \left(\frac{Z_{1,\alpha}}{Z_{1,\alpha}+Z_{2,\alpha}}\right)\right]\\
&=-\ln \mathbb{E}\left(B_{\alpha,\alpha}\right)=-\ln \left(\frac{1}{2}\right)=\ln 2,
\end{align*}
where, $B_{\alpha,\alpha}=\frac{Z_{1,\alpha}}{Z_{1,\alpha}+Z_{2,\alpha}}$	has the $Beta(\alpha,\alpha)$ distribution.
\end{proof}
\noindent We first obtain a class of generalized Bayes estimators of $H_S(\underline{\theta})$, under the squared error loss function \eqref{S1.E2}. For this, we consider following class of improper prior densities for $\underline{\theta}=\left(\theta_1,\theta_2\right)$:
\begin{align}
\Pi_{\beta}({\underline{\theta}})&=\begin{cases}
\frac{1}{\left(\theta_1 \theta_2\right)^{\beta+1}}, & \text{if $\underline{\theta} \in \Theta$ }\\
0, & \text{otherwise}.
\end{cases}	~,~~\beta >-\alpha.
\end{align}
Note that $\Pi_{0}({\underline{\theta}})=\frac{1}{\theta_1 \theta_2},\underline{\theta} \in \Theta$, is the Jeffreys non-informative prior density.
The posterior density function of $\underline{\theta}=\left(\theta_1,\theta_2\right)$ given $\underline{X}=\left(x_1,x_2\right) \in (0,\infty) \times (0,\infty)$ is obtained as
\begin{align*}
\Pi_{\underline{x},\beta}\left(\underline{\theta}\right)&=\begin{cases}\prod_{i=1}^{2}\left\{\frac{x_i^{\alpha}}{\Gamma(\alpha) \theta_i^{\alpha+\beta+1}} e^{-\frac{x_i}{\theta_i}}\right\}, & \text{if $\underline{\theta} \in \Theta$ }\\
0, & \text{o.w.}
\end{cases}	,~~~\beta > -\alpha,
\end{align*}
i.e., the posterior distribution of $\underline{\theta}$, given $\underline{X}=\underline{x} ~(\underline{x} \in (0,\infty)\times (0,\infty))$, is such that $\theta_1^{-1}$ and $\theta_2^{-1}$ are independent and $\theta_i^{-1} \sim Gamma(\alpha+\beta,x_i^{-1}),~i=1,2.$ Also, the posterior risk of an estimator $\delta$ is given by
\begin{equation}
r_{\beta}(\delta,\underline{x})=\mathbb{E}_{\Pi_{\underline{x}}}\left(\delta-H_S(\underline{\theta})\right)^2,~\underline{x} \in (0,\infty)\times (0,\infty),~ \beta >-\alpha.
\end{equation}
The generalized Bayes estimator, $\delta^{GB}_{\beta}(\underline{X})$, which minimizes the posterior risk $(2.2)$, is obtained as
\begin{align*}
\delta^{GB}_{\beta}(\underline{X})&=\mathbb{E}_{\Pi_{\underline{x},\beta}}\left[H_S(\underline{\theta})\right]\\
&= \begin{cases}
\displaystyle \mathbb{E}_{\Pi_{\underline{x},\beta}}\left(\ln \theta_1\right),
 & \text{if $X_1 \geq X_2$ }\nonumber\\\\ 
\displaystyle \mathbb{E}_{\Pi_{\underline{x},\beta}}\left(\ln \theta_2\right),
 & \text{if $X_1 <  X_2$}
\end{cases} \nonumber\\
&= \begin{cases}
\displaystyle \ln X_1-\psi(\alpha+\beta),
& \text{if $X_1 \geq X_2$ }\\\\ 
\displaystyle  \ln X_2-\psi(\alpha+\beta),
& \text{if $X_1 <  X_2$}
\end{cases} \nonumber\\
&=\ln Z_2-\psi(\alpha+\beta) \nonumber \\&= \delta_{\psi(\alpha+\beta)}(\underline{X}).
\end{align*}
Clearly the class $\mathcal{K}_{GB}=\{\delta^{GB}_{\beta}: \beta>-\alpha\}$ of generalized Bayes estimators is contained in the class $\mathcal{K}_1$ of naive estimators. In particular, the natural estimator $\widehat{H}^{(1)}_{N_3}(\underline{X})=\delta^{GB}_0(\underline{X})=\ln Z_2-\psi(\alpha)$ (analogue of the best scale equivariant estimators of $\ln (\theta_1)$ and $\ln(\theta_2)$) is the generalized Bayes estimator with respect to Jeffrey's non-informative prior density $\Pi_{0}({\underline{\theta}})=\frac{1}{\theta_1 \theta_2},\underline{\theta} \in \Theta$. \par 
Now we will attempt to find optimum estimators within the class $\mathcal{K}_1$ of naive estimators under the squared error loss function \eqref{S1.E2}. The following lemma will be useful in obtaining admissible estimators within the subclass $\mathcal{K}_1$.
\begin{lemma}
Let $U=\ln Z_2-H_S(\underline{\theta})$. Then, for any $\underline{\theta} \in \Theta$,
\begin{equation}
\mathbb{E}_{\underline{\theta}}\left(U \right)=\displaystyle{\int_{0}^{\infty}}\ln(z)G_{\alpha}\left(\frac{z}{\mu}\right)g_{\alpha}(z)dz+\displaystyle{\int_{0}^{\infty}}\ln(z)G_{\alpha}\left(\mu z\right)g_{\alpha}(z)dz,~ \mu \geq 1,
\end{equation}
where $G_{\alpha}\left(\cdot\right) $ and $g_{\alpha}\left(\cdot\right)$,  respectively, denote the df and the pdf of $Gamma(\alpha,1)$ distribution.
\end{lemma}
\begin{proof}
	Let $Y_i=\frac{X_i}{\theta_i},~i=1,2$. Then $Y_1 $ and $Y_2$ are iid $Gamma(\alpha,1)$ random variables. Note that, the pdf of $U$ is a permutation symmetric function of $(\theta_1, \theta_2)$. Thus, without loss of generality, we may take $\theta_i=\mu_i,~i=1,2.$ Then, for any $\underline{\theta} \in \Theta,$
	\begin{align*}
	\mathbb{E}_{\underline{\theta}}\left(U \right)&=\mathbb{E}_{\underline{\theta}}\left[\left(\ln Z_2-H_S(\underline{\theta}) \right)\right]\\
	&= \mathbb{E}_{\underline{\theta}}\left[\ln\left(\frac{X_1}{\theta_1}\right) I(X_1>X_2)\right] + \mathbb{E}_{\underline{\theta}}\left[\ln\left(\frac{X_2}{\theta_2}\right) I(X_2 \geq X_1)\right]\\
	&=\displaystyle{\int_{0}^{\infty}}\ln(z)G_{\alpha}\left(\frac{z}{\mu}\right)g_{\alpha}(z)dz+\displaystyle{\int_{0}^{\infty}}\ln(z)G_{\alpha}\left(\mu z\right)g_{\alpha}(z)dz.
	\end{align*}
\end{proof}
\noindent The risk (mean squared error) function of an estimator $\delta_{c} \in \mathcal{K}_1$, is given by
\begin{equation}
R_{\mu}(\delta_{c})=\mathbb{E}\left[\left(\ln Z_2-c-H_S(\underline{\theta})\right)^2\right],~~ \mu \geq 1, c \in \mathbb{R}.
\end{equation}
For any fixed $\mu \geq 1$, the risk function in $(2.4)$ is minimized at 
\begin{align*}
c \equiv c^*(\mu)&=\mathbb{E}_{\underline{\theta}}\left[\left(\ln Z_2-H_S(\underline{\theta}) \right)\right]=\mathbb{E}_{\underline{\theta}}\left(U\right)\\
&=\displaystyle{\int_{0}^{\infty}}\ln(z)G_{\alpha}\left(\frac{z}{\mu}\right)g_{\alpha}(z)dz+\displaystyle{\int_{0}^{\infty}}\ln(z)G_{\alpha}\left(\mu z\right)g_{\alpha}(z)dz,
\end{align*}
using $(2.3)$. Clearly,
\begin{align*}
\frac{d}{d \mu} c^*(\mu)&=-\frac{1}{\mu^2} \displaystyle{\int_{0}^{\infty}}z \ln(z)g_{\alpha}\left(\frac{z}{\mu}\right)g_{\alpha}(z)dz+\displaystyle{\int_{0}^{\infty}}z\ln(z)g_{\alpha}\left(\mu z\right)g_{\alpha}(z)dz\\
&=-\frac{\Gamma(2 \alpha) \mu ^{\alpha-1}}{(\Gamma(\alpha))^2 (1+\mu)^{2\alpha}} \ln \mu \leq 0,~ \forall \mu \geq 1.
\end{align*}
Consequently,
\begin{equation}
\inf\limits_{\mu \geq  1} c^*(\mu)=\psi\left(\alpha\right)=c_1(\alpha),~\text{say},~~ \text{and} 
~~~
\sup\limits_{\mu \geq  1} c^*(\mu)=2\displaystyle{\int_{0}^{\infty}}\ln(z)G_{\alpha}\left( z\right)g_{\alpha}(z)dz=c_2(\alpha),~ \text{say}.
\end{equation}
In the following theorem, we will characterize estimators that are admissible/inadmissible within the class $\mathcal{K}_1$ of naive estimators.
\begin{theorem}
	Let $c_1(\alpha)=\psi\left(\alpha\right)$ and $c_2(\alpha)=2\displaystyle{\int_{0}^{\infty}}\ln(z)G_{\alpha}\left( z\right)g_{\alpha}(z)dz$, $\alpha>0$. Then, for estimating $H_S(\underline{\theta})$, under the mean squared error criterion, the estimators in the class $\mathcal{K}_{1,M}=\left\{\delta_{c} \in \mathcal{K}_1: c \in [c_1(\alpha),c_2(\alpha)]\right\}$ are admissible within the class of estimators $\mathcal{K}_1$ of naive estimators. The estimators in the class\\ $\mathcal{K}_{1,I}=\left\{\delta_{c} \in \mathcal{K}_2 : c \in (-\infty,c_1(\alpha)) \cup (c_2(\alpha),\infty) \right\}$ are inadmissible.
	Furthermore,  for any $-\infty < d < c \leq  c_1(\alpha)$ or $c_2(\alpha) \leq c <d <\infty$,
	$$R_{\mu}\left(\delta_{{c}}\right)< R_{\mu}\left(\delta_{{d}}\right),~~~ \forall~\mu \geq 1.$$
\end{theorem}
\begin{proof}
For any fixed $\mu \geq 1$, the risk function $R_{\mu}(\delta_c)$, defined by $(2.4)$, is a strictly increasing function of $c$ on $[c^*(\mu), \infty)$, it is a decreasing function of $c$ on $(-\infty, c^*(\mu)]$, and it achieves its minimum at $c=c^*(\mu)$. Since $c^*(\mu)$  is a continuous function of $\mu$, using $(2.5)$, it follows that $c^*(\mu)$ takes all values in the interval $[c_1(\alpha), c_2(\alpha))$. Thus, we conclude that each $c \in [c_1(\alpha), c_2(\alpha))$ minimizes the risk function $R_{\mu}(\delta_c)$ at some $\mu \in [1,\infty)$. This establishes that the estimators $\delta_{c}(\cdot)$, for $c \in [c_1(\alpha), c_2(\alpha))$ are admissible within the subclass $\mathcal{K}_1$. Further, continuity of the risk function ensures the admissibility of estimator $\delta_{c_2(\alpha)}(\cdot)$, within the subclass $\mathcal{K}_1$. This proves the first assertion. Also, since $c_1(\alpha)=\psi(\alpha) \leq c^*(\mu) < c_2(\alpha)$, $\forall$ $\mu \geq 1$, it follows for any $\mu \geq 1$ the risk function $R_{\mu}(\delta_c)$ is a strictly decreasing function of $c$ on $(-\infty, c_1(\alpha)]$ and it is a strictly increasing function of $c$ on $[c_2(\alpha), \infty)$. This proves the second assertion. Hence the result follows.	
\end{proof}
In the sequel we discuss optimality of natural estimators $\widehat{H}^{(1)}_{N_i}(\underline{X}),~i=1,2,3.$ It directly follows from Theorem 2.1 that within the class $\mathcal{K}_1$ of naive estimators of $H_S(\underline{\theta})$, the natural estimator 	$\widehat{H}^{(1)}_{N_3}(\underline{X})=\delta_{\psi(\alpha)}(\underline{X})= \ln Z_2 -\psi(\alpha)$ is an admissible estimator of $H_S(\underline{\theta})$. Using lemma 2.1, we have $\ln (\alpha+1) > \ln \alpha > \psi(\alpha)=c_1(\alpha),~\forall \alpha >0 $. It can be numerically verified (see Table 2.1 and Figure 2.1 ) that $c_2(\alpha) > (<) \ln \alpha $ for $\alpha > (<) ~0.63$ and that $c_2(\alpha)> (<) \ln (\alpha+1)$, for $\alpha > (<) ~6.05$. The above discussion along with Theorem 2.1, yields the following corollary.
\begin{corollary}
	\begin{itemize}
		\item[(a)] Within the class $\mathcal{K}_1$ of naive estimators of $H_S(\underline{\theta})$, the natural estimator $\widehat{H}^{(1)}_{N_3}(\underline{X})=\delta_{\psi(\alpha)}(\underline{X})= \ln Z_2 -\psi(\alpha)$ is an admissible estimator for estimating $H_S(\underline{\theta})$.
		
		\item[(b)] For $\alpha \in (0,0.63)$, the natural estimator $\widehat{H}^{(1)}_{N_1}(\underline{X})=\delta_{\ln \alpha}(\underline{X})= \ln Z_2 -\ln \alpha$ is an inadmissible estimator for estimating $H_S(\underline{\theta})$ and it is dominated by the naive estimator $\delta_{{c_2(\alpha)}}(\underline{X})=\ln Z_2 -c_2(\alpha)$. However, for $\alpha \in (0.63,\infty)$, the natural estimator $\widehat{H}^{(1)}_{N_1}(\underline{X})$ is an admissible estimator within the class $\mathcal{K}_1$ of naive estimators.
		
		\item[(c)]  For $\alpha \in (0,6.05)$, the natural estimator $\widehat{H}^{(1)}_{N_2}(\underline{X})=\delta_{\ln \alpha}(\underline{X})= \ln Z_2 -\ln (\alpha+1)$ is an inadmissible estimator for estimating $H_S(\underline{\theta})$ and it is dominated by the naive estimator $\delta_{{c_2(\alpha)}}(\underline{X})=\ln Z_2 -c_2(\alpha)$. However, for $\alpha \in (6.05,\infty)$, the natural estimator $\widehat{H}^{(1)}_{N_2}(\underline{X})$ is an admissible estimator within the class $\mathcal{K}_1$ of naive estimators.
		
		\item[(d)] For $\alpha \in (0,0.63)$, the natural estimator $\widehat{H}^{(1)}_{N_1}(\underline{X})=\delta_{\ln \alpha}(\underline{X})= \ln Z_2 -\ln \alpha$ dominates the natural estimator $\widehat{H}^{(1)}_{N_2}(\underline{X})=\delta_{\ln (\alpha+1)}(\underline{X})= \ln Z_2 -\ln (\alpha+1)$, and both of them are dominated by the naive estimator $\delta_{{c_2(\alpha)}}(\underline{X})=\ln Z_2- c_2(\alpha)$.
	\end{itemize}
\end{corollary}

Now we discuss optimality of generalized Bayes estimators, belonging to the class $\mathcal{K}_{GB}=\left\{ \delta^{GB}_{\beta}: \beta >-\alpha\right\}= \left\{ \delta_{\psi(\alpha+\beta)}: \beta >-\alpha\right\} \subseteq \mathcal{K}_1$, where  $\delta^{GB}_{\beta}(\underline{X})=\delta_{\psi(\alpha+\beta)}(\underline{X})=\ln Z_2 -\psi(\alpha+\beta), \beta > -\alpha$. Note that $\psi(t)$ is an increasing function of $t \in (0,\infty)$ (see \cite{abramowitz1964handbook}). For any fixed $\alpha>0$, define $\beta_0(\alpha)=\psi^{-1}(c_2(\alpha))-\alpha$. Since $c_2(\alpha) > c_1(\alpha)=\psi(\alpha)$ (i.e., $\beta_0(\alpha)>0$), $\forall$ $\alpha >0$ and $\psi(2\alpha)>c_2(\alpha)$ (Lemma 2.1), (i.e., $\beta_0(\alpha)<\alpha$) $\forall \alpha >0$, we have $\beta_0(\alpha) \in (0,\alpha),~\forall \alpha >0.$ Moreover, $\psi(\beta+\alpha)< c_1(\alpha)=\psi(\alpha)~ \forall \beta <0$, $\psi(\beta+\alpha)> c_2(\alpha)~ \forall \beta >\beta_0(\alpha)$ and $c_1(\alpha) \leq \psi(\beta+\alpha) \leq c_2(\alpha),~ \forall \beta \in [0,\beta_0(\alpha)]$. Now, we have the following corollary to Theorem 2.1.

\begin{corollary}
	For any $\alpha>0$, define $\beta_0(\alpha)=\psi^{-1}(c_2(\alpha))-\alpha$. Then
	\begin{itemize}
		\item[(a)] the generalized Bayes estimators $\{\delta_{\psi(\alpha+\beta)}: 0 \leq \beta \leq \beta_0(\alpha)\}$ are admissible within the class $\mathcal{K}_1$ of naive estimators. 
		\item[(b)] the generalized Bayes estimators $\{\delta_{\psi(\alpha+\beta)}: \beta \in (-\alpha,0) \cup ( \beta_0(\alpha),\infty)\}$ are inadmissible for estimating  $H_S(\underline{\theta})$. For any $\beta \in (-\alpha,0)$, the generalized Bayes estimator $\delta_{\psi(\alpha+\beta)}(\underline{X})=\ln Z_2 - \psi(\alpha+\beta)$ is dominated by the natural estimator $\widehat{H}^{(1)}_{N_3}(\underline{X})=\delta_{c_1(\alpha)}(\underline{X})=\ln Z_2-\psi(\alpha)$, and for any $\beta \in \left(\beta_0(\alpha), \infty\right)$, the generalized Bayes estimator $\delta_{\psi(\alpha+\beta)}$ is dominated by the naive estimator $\delta_{{c_2(\alpha)}}(\underline{X})=\ln Z_2-c_2(\alpha)$.
	\end{itemize}
\end{corollary}
For various values of $\alpha$'s ($\alpha>0$), $c_1(\alpha)$, $c_2(\alpha)$, $\ln \alpha $, $\ln (\alpha+1)$ and $\beta_0(\alpha)$ are tabulated in Table 2.1.
\begin{table}[]
	\centering
	\caption*{\pmb{Table 2.1: Values of $c_1(\alpha)$, $c_2(\alpha)$, $\beta_0(\alpha)$ and $\psi(2\alpha)-\ln 2$ for various values of $\alpha$}}
	\scalebox{1.4}{
	\begin{tabular}{|c|c|c|c|c|c|c|}
		\hline
		$\pmb\alpha$ &
		$\pmb{ c_1(\alpha)}$ &
		$\pmb{ c_2(\alpha)}$ &
		$\pmb{ \ln \alpha}$ &
		$\pmb{ \ln (\alpha+1)}$ &
		$\pmb{ \beta_0(\alpha)}$ &
		$\pmb{ \psi(2\alpha)-\ln 2}$ \\ 
		 &
		 &
		 &
	     &
		 &
		$\pmb{ =\psi^{-1}(c_2(\alpha))-\alpha}$ 
		& \\ \hline
		0.2   & -5.289 & -2.682 & -1.609 & 0.182 & 0.183 & -3.254  \\ \hline
		0.4   & -2.561 & -1.158 & -0.916 & 0.336 & 0.322 & -1.658 \\ \hline
		0.6   & -1.54  & -0.531 & -0.51  & 0.47  & 0.428 & -0.982 \\ \hline
		\textbf{0.63}  & -1.425 & \textbf{-0.457} & \textbf{-0.457} & 0.49  & 0.443 & -0.901  \\ \hline
		0.8   & -0.965 & -0.152 & -0.223 & 0.587 & 0.514 & -0.567  \\ \hline
		1     & -0.577 & 0.115  & 0      & 0.693 & 0.588 & -0.270  \\ \hline
		1.5   & 0.036  & 0.566  & 0.405  & 0.916 & 0.738 & 0.229  \\ \hline
		2     & 0.422  & 0.865  & 0.693  & 1.098 & 0.86  & 0.562  \\ \hline
		2.5   & 0.703  & 1.091  & 0.916  & 1.252 & 0.964 & 0.812  \\ \hline
		3     & 0.922  & 1.272  & 1.098  & 1.386 & 1.057 & 1.012  \\ \hline
		3.5   & 1.103  & 1.423  & 1.252  & 1.504 & 1.141 & 1.179  \\ \hline
		4     & 1.256  & 1.553  & 1.386  & 1.609 & 1.218 & 1.322 \\ \hline
		4.5   & 1.388  & 1.667  & 1.504  & 1.704 & 1.291 & 1.447 \\ \hline
		5     & 1.506  & 1.769  & 1.609  & 1.791 & 1.359 & 1.558  \\ \hline
		5.5   & 1.611  & 1.861  & 1.704  & 1.871 & 1.424 & 1.658  \\ \hline
		6     & 1.706  & 1.944  & 1.791  & 1.945 & 1.486 & 1.749  \\ \hline
		\textbf{6.05} & 1.716  & \textbf{1.953}  & 1.8    & \textbf{1.953} & 1.494 & 1.759  \\ \hline
		6.5   & 1.792  & 2.021  & 1.871  & 2.014 & 1.55  & 1.832  \\ \hline
		7     & 1.872  & 2.092  & 1.945  & 2.07  & 1.661 & 1.909  \\ \hline
		8     & 2.015  & 2.22   & 2.079  & 2.197 & 2.458 & 2.047  \\ \hline
		9     & 2.14   & 2.333  & 2.197  & 2.302 & 3.042 & 2.169  \\ \hline
		10    & 2.251  & 2.433  & 2.302  & 2.397 & 3.459 & 2.277  \\ \hline
		12    & 2.442  & 2.608  & 2.484  & 2.564 & 3.91  & 2.463  \\ \hline
		15    & 2.674  & 2.822  & 2.708  & 2.772 & 3.913 & 2.691  \\ \hline
		16    & 2.741 &  2.883  &  2.772 & 2.833 &  3.782&  2.756  \\ \hline
		18    & 2.862  & 2.996  & 2.890  & 2.944 & 3.369 & 2.876  \\ \hline
		20    & 2.970  & 3.098  & 2.995  & 3.044 & 2.790 & 2.983  \\ \hline
	\end{tabular}}
\end{table}

\FloatBarrier
\vskip -0.6in
\begin{figure}[!h]
	\centering
	\includegraphics[height=3.6in,width=4.3in]{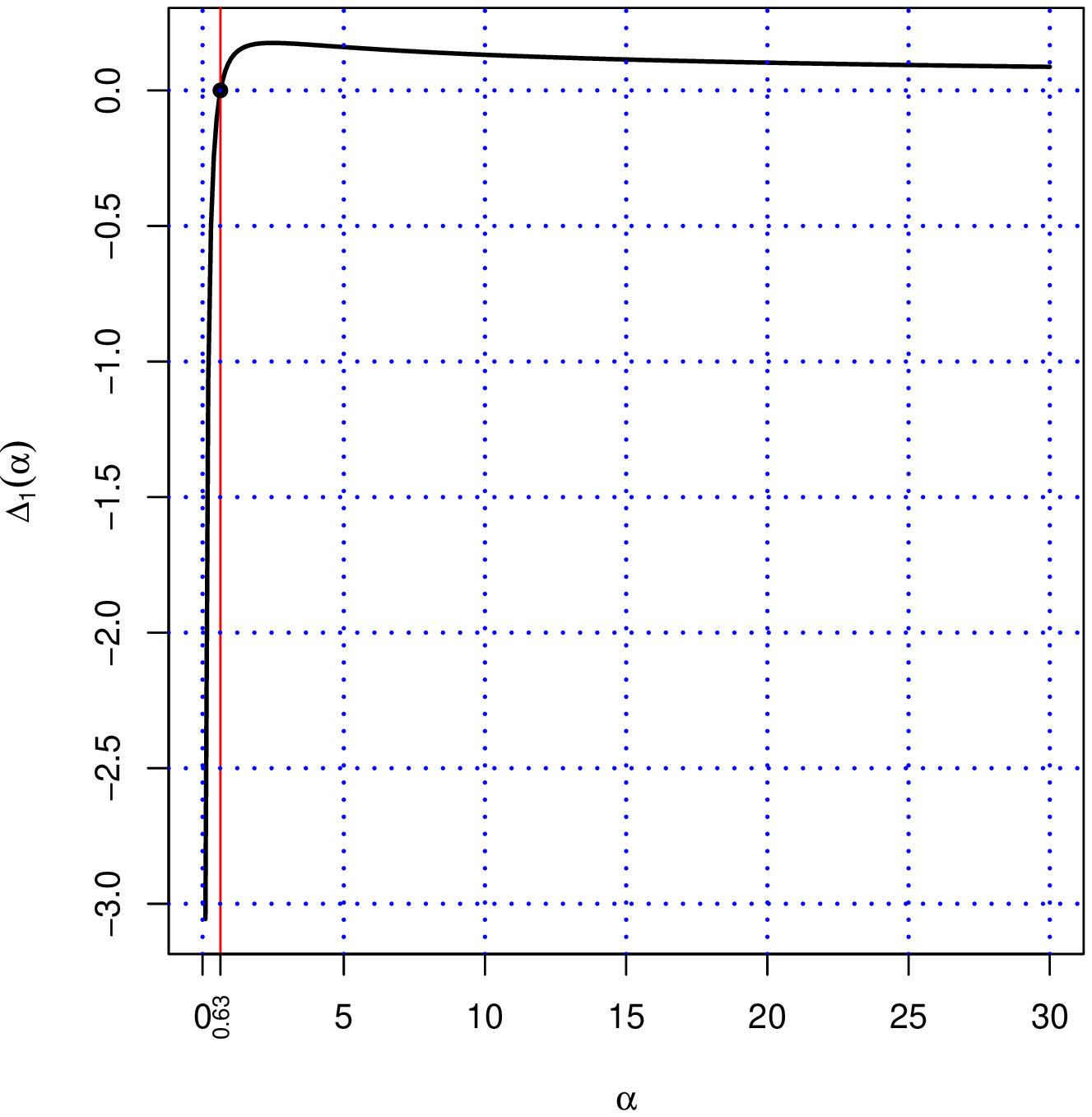}
\end{figure}
\vskip -0.3in
\begin{figure}[!h]
	\centering
	\includegraphics[height=3.6in,width=4.25in]{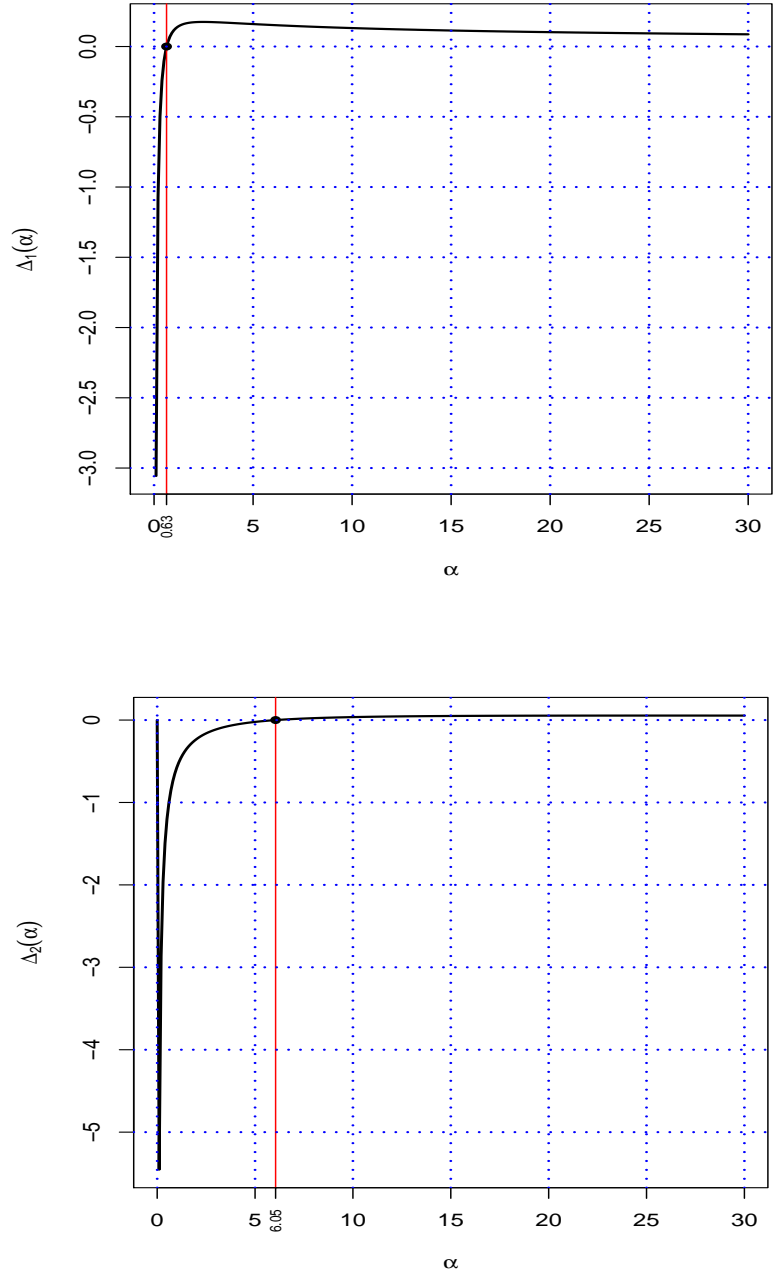}
	\caption{\textbf{Plots of $\Delta_1(\alpha)=2\displaystyle{\int_{0}^{\infty}}\ln(z)G_{\alpha}\left( z\right)g_{\alpha}(z)dz -\ln \alpha $ and $\Delta_{2}(\alpha)=2\displaystyle{\int_{0}^{\infty}}\ln(z)G_{\alpha}\left( z\right)g_{\alpha}(z)dz -\ln (\alpha+1) $.}}
\end{figure}
\FloatBarrier
In the following subsection we establish a general result that, under certain conditions, provides shrinkage type improvements over an an arbitrary scale and permutation equivariant estimator. As a consequence of this general result, we obtain shrinkage estimators dominating various naive estimators of $H_S(\underline{\theta})$ belonging to the class $\mathcal{K}_1$. 
\subsection{Shrinkage type improvements over scale and permutation equivariant estimators}
In this section we will attempt to derive conditions under which shrinkage type improvements over an arbitrary scale and permutation equivariant estimator can be found. For this purpose, we will consider orbit by orbit improvement of the risk function, as proposed by \cite{brewster1974improving}. Recall that a typical estimator in the class $\mathcal{K}_2$ of scale and permutation equivariant estimator of $ H_S(\underline{\theta})$ is of the form $\delta_{\Phi}(\underline{X}) = \ln Z_2 -\Phi(T) $, for some function function $\Phi(\cdot):\left(0,1\right] \to \mathbb{R}$, where $T=\frac{Z_1}{Z_2}$. 
We first provide two supporting lemmas that will be useful in proving the main result of this section.  

\begin{lemma} \label{S4,L2}
	 Let $U=\ln Z_2-H_S(\underline{\theta})$. For any fixed $ t \in (0,1]$ and $\mu \geq 1$, the conditional pdf of $U$, given $T=t$, is given by
	\begin{align*}
	f_{1,\underline{\theta}}(u|t)&=\frac{\frac{1}{\mu^{\alpha}}e^{2\alpha u}e^{-\left(1+\frac{t}{\mu}\right)e^{u}}+\mu^{\alpha} e^{2\alpha u}e^{-\left(1+t\mu\right)e^{u}}}{\Gamma(2\alpha)\left[\frac{1}{\mu^{\alpha}\left(1+\frac{t}{\mu}\right)^{2\alpha}}+\frac{\mu^{\alpha}}{\left(1+t \mu\right)^{2\alpha}}\right]},~ -\infty< u< \infty.
	\end{align*}
\end{lemma}
\begin{proof} Let $t \in (0,1]$ and $\mu \in [1,\infty)$ be fixed. Since the pdf of $U$, given $T=t$, is a permutation symmetric function of $(\theta_1,\theta_2)$, without loss of generality, we may assume that $\theta_i=\mu_i,~i=1,2.$
	Let $h_{\underline{\theta}}(\cdot)$ denote the pdf of $T$. Then, for any fixed $t \in (0,1]$, the df of $U$, given $T=t$, is,
	\begin{align}
	F_{1,\underline{\theta}}(u|t)&=\mathbb{P}_{\underline{\theta}}\left(U \leq u|T=t\right) \nonumber \\
	&=\frac{1}{h_{\underline{\theta}}(t)}\lim\limits_{h \downarrow 0} \frac{N_1(h|u,t,\underline{\theta})}{h},
	\end{align}
where, for $-\infty < u < \infty$ and $h>0$ (sufficiently small),
\begin{align*}
N_1(h|u,t,\underline{\theta})&=	\mathbb{P}_{\underline{\theta}}\left(U \leq u, t-h < T \leq t\right)\\
&=\mathbb{P}_{\underline{\theta}}\left(\ln Z_2-H_S(\underline{\theta}) \leq u, t-h < \frac{Z_1}{Z_2} \leq t\right)\\
&=\mathbb{P}_{\underline{\theta}}\left(X_1 \geq X_2, \ln(X_1)-\ln(\theta_1) \leq u, t-h < \frac{X_2}{X_1} \leq t\right)\\
&~~\hspace*{2cm}+ \mathbb{P}_{\underline{\theta}}\left(X_2 > X_1, \ln(X_2)-\ln \theta_2 \leq u, t-h < \frac{X_1}{X_2} \leq t\right)\\
&=\mathbb{P}_{\underline{\theta}}\left(\ln \left(Y_1\right) \leq u, \frac{(t-h)}{\mu} < \frac{Y_2}{Y_1} \leq \frac{t}{\mu}\right)+ \mathbb{P}_{\underline{\mu}}\left(\ln\left(Y_2\right) \leq u, (t-h)\mu < \frac{Y_1}{Y_2} \leq t \mu\right)\\
&=\mathbb{P}_{\underline{\theta}}\left( Y_1 \leq e^u, \frac{(t-h)Y_1}{\mu} < Y_2 \leq \frac{t Y_1}{\mu} \right)+ \mathbb{P}_{\underline{\mu}}\left( Y_2 \leq e^u, (t-h)\mu Y_2 < Y_1 \leq t \mu Y_2\right),
\end{align*}
where, $Y_i=\frac{X_i}{\theta_i},~i=1,2$, so that $Y_1 $ and $Y_2$ are iid $Gamma(\alpha,1)$. Let $G_{\alpha}$ and $g_{\alpha}$ denote the df and pdf, respectively, of $Y_1$. Then, 
\begin{align}
N_1(h|u,t,\underline{\theta})&=\displaystyle{\int_{0}^{e^u}} \left[G_{\alpha}\left(\frac{ty}{\mu}\right)-G_{\alpha}\left(\frac{(t-h)y}{\mu}  \right)\right] f_{\alpha}(y) dy \nonumber\\
&~~\hspace*{2cm}+ \displaystyle{\int_{0}^{e^u}} \left[G_{\alpha}\left(t \mu y\right)-G_{\alpha}\left((t-h)\mu y \right)\right] g_{\alpha}(y) dy, ~ -\infty < u < \infty, h>0 \nonumber\\
\Rightarrow \lim\limits_{h \downarrow 0} \frac{N_1(h|u,t,\underline{\theta})}{h}&= \displaystyle{\int_{0}^{e^u}} \frac{y}{\mu} g_{\alpha}\left(\frac{ty}{\mu}\right) g_{\alpha}\left(y\right) dy +  \displaystyle{\int_{0}^{e^u}} y\theta g_{\alpha}\left(ty\mu\right) g_{\alpha}\left(y\right) dy,~ -\infty < u < \infty.
\end{align}
Using $(2.6)$ and $(2.7)$, for any fixed $t \in (0,1]$, the conditional pdf of $U$, given $T=t$, is given by
\begin{align*}
f_{1,\underline{\theta}}(u|t)&=\frac{\frac{e^{2u}}{\mu} g_{\alpha}\left(\frac{te^{u}}{\mu}\right) g_{\alpha}\left(e^{u}\right)  +  \mu e^{2u} g_{\alpha}\left(\mu t e^{u}\right) g_{\alpha}\left(e^{u}\right)}{h_{\underline{\theta}}(t)}\\
&=\frac{\frac{1}{\mu^{\alpha}}e^{2\alpha u}e^{-\left(1+\frac{t}{\mu}\right)e^{u}}+\mu^{\alpha} e^{2\alpha u}e^{-\left(1+t\mu\right)e^{u}}}{\Gamma(2\alpha)\left[\frac{1}{\mu^{\alpha}\left(1+\frac{t}{\mu}\right)^{2\alpha}}+\frac{\mu^{\alpha}}{\left(1+t \mu\right)^{2\alpha}}\right]},~ -\infty < u < \infty, t \in (0,1], \mu \geq 1.
\end{align*}
\end{proof}
\begin{lemma}
	For any fixed $\alpha >0 $, define
	\begin{align}
	k_{t}(\mu)&=\frac{(1+t \mu)^{2 \alpha} \ln\left(1+\frac{t}{\mu}\right)+(\mu+t)^{2\alpha}\ln\left(1+t \mu\right)}{(1+t \mu)^{2\alpha}+(\mu+t)^{2\alpha}}~;~ 0 < t \leq 1,~ \mu \geq 1.
	\end{align}
	Then, $ \inf\limits_{\mu\geq  1} k_{t}(\mu)=\ln(1+t)$.
	
\end{lemma}
\begin{proof}
	Note that, for any fixed $t \in (0,1]$ and $\alpha >0$, $k_t(1)=\ln (1+t)$ and $\lim\limits_{\mu \to \infty} k_t(\mu)=\infty.$ Thus, it suffices to show that
	$k_t(\mu) \geq \ln (1+t), \forall t \in (0,1], \mu \geq 1 ~\text{and}~ \alpha >0 $,
	or equivalently that 
\begin{align}
~~~~~~~~\hspace{4mm} \frac{\ln (1+t\mu)-\ln (1+t)}{\ln (1+t)-\ln (1+\frac{t}{\mu})} \geq \left(\frac{1+t \mu}{\mu +t}\right)^{2\alpha}, \forall t \in (0,1], \mu \geq 1 ~ \text{and}~ \alpha > 0.
\end{align}
Since, for any $t \in (0,1]$ and $\mu \geq 1$, $\frac{1+t \mu}{\mu + t} <1$, to establish $(2.9)$, it is enough to show that 
\begin{align*}
\frac{\ln (1+t\mu)-\ln (1+t)}{\ln (1+t)-\ln (1+\frac{t}{\mu})} &\geq 1, \forall t \in (0,1], \mu \geq 1 \\
\iff  \ln (1+t\mu) + \ln (\mu+t)-\ln \mu &\geq 2 \ln (1+t), \forall t \in (0,1],~ \mu \geq 1.
\end{align*}
It is easy to verify that the expression on left hand side above is an increasing function of $\mu$ on $[1, \infty)$, and consequently, the above inequality holds.
\end{proof}
For estimating $H_S(\underline{\theta})$, the risk function of any scale and permutation equivariant estimator $\delta_{\Phi}(\underline{X})=\ln Z_2-\Phi(T)$, is given by,
\begin{equation*}
R_{\mu}(\delta_{\Phi})=\mathbb{E}_{\underline{\theta}}\left[R_1(\mu,\Phi(t))\right],~ \mu \geq 1,
\end{equation*}
where, for any fixed $t \in (0,1]$, 
\begin{equation}
R_1(\mu,\Phi(t))=\mathbb{E}_{\underline{\theta}}\left[\left(\ln Z_2-\Phi(T)-H_S(\underline{\theta})\right)^2\Big|T=t\right],~ \mu \geq 1
\end{equation}
is the conditional risk of $\delta_{\Phi}(\cdot)$ given $T=t$. For fixed $\mu \geq 1$ and $t \in (0,1]$, the conditional risk in $(2.10)$ is minimized for the choice $\Phi(t)=\Phi_{\mu}(t)$, where
\begin{align}
\Phi_{\mu}(t)&=\mathbb{E}_{\underline{\theta}}\left[\left(\ln Z_2-H_S(\underline{\theta})\right)\Big|T=t\right]\\
&=\mathbb{E}_{\underline{\theta}}\left[U|T=t\right] \nonumber \\
&= \displaystyle{\int_{-\infty}^{\infty}} u f_{1, \underline{\theta}}(u|t) du \nonumber \\
&=\frac{N}{\Gamma(2\alpha)\left[\frac{1}{\mu^{\alpha}\left(1+\frac{t}{\mu}\right)^{2\alpha}}+\frac{\mu^{\alpha}}{\left(1+t \mu\right)^{2\alpha}}\right]},
\end{align}
using Lemma $2.3$, where
\begin{align*}
N &=\frac{1}{\mu^{\alpha}}\displaystyle{\int_{-\infty}^{\infty}} u e^{2\alpha u} e^{-\left(1+\frac{t}{\mu}\right)e^{u}} du+\mu^{\alpha}\displaystyle{\int_{-\infty}^{\infty}} u e^{2\alpha u} e^{-\left(1+t\mu\right)e^{u}} du\\
&=\frac{1}{\mu^{\alpha}} \displaystyle{\int_{0}^{\infty}} \ln(u) u^{2\alpha-1}e^{-\left(1+\frac{t}{\mu}\right)u}du+\mu^{\alpha}\displaystyle{\int_{0}^{\infty}} \ln(u) u^{2\alpha-1}e^{-\left(1+t\mu\right)u}du\\
&= \frac{\Gamma(2\alpha)}{\mu^{\alpha}\left(1+\frac{t}{\mu}\right)^{2\alpha}} \left[\psi(2\alpha)-\ln\left(1+\frac{t}{\mu}\right)\right]+\frac{\mu^{\alpha}\Gamma(2\alpha)}{\left(1+t\mu\right)^{2\alpha}} \left[\psi(2\alpha)-\ln\left(1+t\mu\right)\right]. 
\end{align*}
Therefore, for any fixed $t \in (0,1]$,
\begin{align*}
\Phi_{\mu}(t)&=\psi(2\alpha)-k_t(\mu),~ \mu \geq 1,
\end{align*}
where, $k_{t}(\mu)$ is defined by (2.8). Further, using Lemma $2.4$, we get
\begin{align}
\Phi_{*}(t)&=\sup\limits_{\mu \geq  1} \Phi_{t}(\mu)=\psi(2 \alpha)-\ln(1+t),~0<t \leq 1.
\end{align}
We now use the idea of \cite{brewster1974improving} to obtain a general result that, under certain conditions, is helpful in finding estimators dominating a scale and permutation equivariant estimator.
\begin{theorem}
For a given function $\Phi:(0,1] \to \mathbb{R}$, let  $\delta_{\Phi}\left(\underline{X}\right)=\ln Z_2-\Phi\left(T\right) $ be a scale and permutation equivariant estimator of $H_S(\underline{\theta})$. Suppose that
$\mathbb{P}_{\underline{\theta}}\left[\Bigg\{T:\Phi(T) > \Phi_{*}(T) \Bigg\} \right] >0,~ \text{for some}~ \underline{\theta} \in \Theta,$ where $\Phi_{*}(\cdot)$ is given by $(2.13).$
Then, the estimator $\delta_{\Phi}(\cdot)$ is dominated by $\delta_{\Phi_I}\left(\underline{X}\right)=\ln Z_2-\Phi_I(T)$, where,
\begin{equation}
\Phi_I\left(\underline{t}\right)=\begin{cases}
\psi(2\alpha)-\ln(1+t), & \text{ if }~ \Phi(t) > \psi(2\alpha)- \ln(1+t),\\
\Phi(t), & \text{otherwise}.
\end{cases}
\end{equation}
\end{theorem} 
\begin{proof}
Consider the risk difference,
\begin{align*}
R\left(\underline{\theta}, \delta_{\Phi}\right)-R\left(\underline{\theta}, \delta_{\Phi_I}\right)&=\mathbb{E}_{\underline{\theta}}\left(D_{\underline{\theta}}(T)\right),
\end{align*}
where, for $t \in (0,1]$,
\begin{align*}
D_{\underline{\theta}}(T)
&=\left(\Phi(t)-\Phi_{I}(t)\right)\left[\Phi(t)+\Phi_{I}(t)-2\Phi_{\underline{\mu}}(t)\right], \mu \geq 1;
\end{align*}
here, $\Phi_{\mu}(t)$ is defined by $(2.11)$. Let us fix $t \in (0,1]$. Clearly, if $\Phi(t) \leq  \Phi_{*}(t)$ (so that $\Phi_{I}(t)=\Phi(t)$), then $D_{\underline{\theta}}(T)=0.$ Also, if $\Phi(t) > \Phi_{*}(t)$ (so that $\Phi_{I}(t)=\Phi_{*}(t)$), then
\begin{align*}
D_{\underline{\theta}}(t)
&=\left(\Phi(t)-\Phi_{*}(t)\right)\left[\Phi(t)+\Phi_{*}(t)-2\Phi_{\mu}(t)\right]\\
& > \left(\Phi(t)-\Phi_{*}(t)\right)^2\\
&\geq0,~ \forall \mu \geq 1.
\end{align*}
Since, $\mathbb{P}_{\underline{\theta}}\left[\Bigg\{T:\Phi(T) > \Phi_{*}(T) \Bigg\} \right] > 0,~ \text{for some}~ \underline{\theta} \in \Theta~ ,$  we conclude that $R\left(\underline{\theta}, \delta_{\Phi}\right)-R\left(\underline{\theta}, \delta_{\Phi_I}\right)\geq 0,~ \forall \underline{\theta} \in \Theta,$ with strictly inequality for some $\underline{\theta} \in \Theta$. Hence, the result follows.
\end{proof} 

Now we will discuss applications of the above theorem in finding shrinkage type improvements over various naive estimators belonging to class $\mathcal{K}_1$. Let $c > \psi(2\alpha)-\ln 2$. Then
$$\mathbb{P}_{\underline{\theta}}\left(c > \Phi_{*}(T)\right)=\mathbb{P}_{\underline{\theta}}\left(\ln (1+T) > \psi(2\alpha)-c \right) >0,~\forall~ \underline{\theta} \in \Theta.$$
Consequently, any naive estimator $\delta_{c}(\underline{X})= \ln Z_2-c$, with $c>\psi(2\alpha)-\ln 2$, is inadmissible for estimating $H_S(\underline{\theta})$ and is dominated by the shrinkage estimator

 \begin{align}
\delta^{(S)}_{c}(\underline{X})
&= \begin{cases}
\ln Z_2 -c,  &\text{if, $\frac{Z_1}{Z_2} < e^{\psi(2 \alpha)-c}-1$ }\\
\ln(X_1 +X_2)-\psi(2\alpha), ~~ &  \text{if, $\frac{Z_1}{Z_2} \geq  e^{\psi(2 \alpha)-c}-1$}
\end{cases}
.
\end{align}
From Lemma 2.1 (ii), we have, for any $\alpha >0$, $\psi(2\alpha)< \ln (2\alpha)$, i.e., $\ln (\alpha+1)> \ln \alpha >\psi(2\alpha)-\ln 2$. Now, from the above discussion, it follows that natural estimators $\widehat{H}^{(1)}_{N_1}(\underline{X})=\delta_{\ln \alpha}(\underline{X})= \ln Z_2 -\ln \alpha$
and $\widehat{H}^{(1)}_{N_2}(\underline{X})=\delta_{\ln (\alpha+1)}(\underline{X})= \ln Z_2 -\ln (\alpha+1)$ are inadmissible for estimating $H_S(\underline{\theta})$ and are dominated by shrinkage estimators

 \begin{align}
\delta^{(S)}_{\ln\alpha}(\underline{X})
&= \begin{cases}
\ln Z_2 -\ln \alpha,  &\text{if, $\frac{Z_1}{Z_2} < \frac{e^{\psi(2 \alpha)}}{\alpha}-1$ }\\
\ln(X_1 +X_2)-\psi(2\alpha), ~~ &  \text{if, $\frac{Z_1}{Z_2} \geq   \frac{e^{\psi(2 \alpha)}}{\alpha}-1$}
\end{cases}
,
\end{align}
and

\begin{align}
\delta^{(S)}_{\ln(\alpha+1)}(\underline{X})
&= \begin{cases}
\ln Z_2 -\ln (\alpha+1),  &\text{if, $\frac{Z_1}{Z_2} < \frac{e^{\psi(2 \alpha)}}{\alpha+1}-1$ }\\
\ln(X_1 +X_2)-\psi(2\alpha), ~~ &  \text{if, $\frac{Z_1}{Z_2} \geq   \frac{e^{\psi(2 \alpha)}}{\alpha+1}-1$}
\end{cases}
,
\end{align}
respectively. A summary of the above discussion, in conjunction with Theorem 2.1, Lemma 2.1 (iii) and Theorem 2.2, is provided in the form of following theorem.
\begin{theorem}
	\begin{itemize}
		\item[(a)] The naive estimators $\{\delta_{{c}}: c \in (-\infty, c_1(\alpha)) \cup (\psi(2\alpha)-\ln 2,\infty)\}$ are inadmissible for estimating $H_S(\underline{\theta})$ under the squared error loss function \eqref{S1.E2}. For any $c \in \left(-\infty, c_1(\alpha)\right)$ the naive estimator $\delta_{{c}}(\underline{X})=\ln Z_2-c$ is dominated by the natural estimator  $\delta_{{c_1(\alpha)}}(\underline{X})=\ln Z_2-c_1(\alpha)$; for any $c \in \left(\psi(2\alpha)-\ln 2,c_2(\alpha)\right)$, the naive estimator $\delta_{{c}}(\underline{X})$ is dominated by the shrinkage estimator, defined by $(2.15)$ and for any $c \in [c_2(\alpha), \infty)$ the naive estimator $\delta_{{c}}$ is dominated by the shrinkage estimator $\delta^{(S)}_{{c_2(\alpha)}}(\underline{X})$, as defined in $(2.15)$.
		
		\item[(b)] The natural estimator $\widehat{H}^{(1)}_{N_1}(\underline{X})=\delta_{\ln \alpha}(\underline{X})= \ln Z_2 -\ln \alpha$ is inadmissible for estimating $H_S(\underline{\theta})$ and is dominated by the shrinkage estimator $\delta^{(S)}_{{\ln \alpha}}(\underline{X})$, defined by $(2.16)$.
		
		\item[(c)] The natural estimator $\widehat{H}^{(1)}_{N_2}(\underline{X})=\delta_{\ln (\alpha+1)}(\underline{X})= \ln Z_2 -\ln (\alpha+1)$ is inadmissible for estimating $H_S(\underline{\theta})$ and is dominated by the shrinkage estimator $\delta^{(S)}_{{\ln (\alpha+1)}}(\underline{X})$, defined by $(2.17)$.
	\end{itemize}
\end{theorem}

\begin{rem}
	\begin{itemize}
		\item[(i)] The global admissibility of naive estimators $\{\delta_{{c}}:c \in [c_1(\alpha),\psi(2\alpha)-\ln 2]\}$ is unresolved. We believe that these estimators are globally admissible but we have not been able to prove this. This seems to be an interesting problem for future research.
		\item[(ii)] A natural question that arises is whether an unbiased estimator of $H_S(\underline{\theta})$ exists $?$ We have tried to address this question and could not succeed in resolving it. Through our experience with analysis carried out to resolve the question, we conjecture that an unbiased estimator for the selected entropy $H_S(\underline{\theta})$ does not exists. However, we have not been able to settle this question and it remains an open problem, that may also be considered in our future research.
	\end{itemize}

\end{rem}


\section{Estimation of Entropy of the Better Selected Population}
We call the population associated with $\min\{H(\theta_1), H(\theta_2)\}$, as the ``better" population. A natural selection rule for selecting the better population is to choose the population corresponding to $Z_1=\min\{X_1,X_2\}$. Let $M \equiv M(\underline{X})$ denotes the index of the better selected population, $i.e., M = i, \text{if},  X_i = Z_1, i = 1, 2.$ Following selection of the better population, our goal is to estimate the Shannon entropy of the selected better population, which is equivalent to estimation of
\begin{equation}
\label{S5.E1}
H_M(\underline{\theta})=\begin{cases}
\ln \theta_1, & \text{if ${X}_1  \leq {X}_2$ }\\
\ln \theta_2, & \text{if ${X}_1 > {X}_2$}
\end{cases}
=\ln \theta_1I(X_1 \leq X_2)+\ln \theta_2I(X_1 > X_2).
\end{equation}
In this section, we consider estimation of $H_M(\underline{\theta})$ under the  squared error loss function,
\begin{equation}
\label{S3.E2}
L(\underline{\theta},a) = \left(a-H_M(\underline{\theta})\right)^2,~
\underline{\theta} \in \Theta= (0,\infty) \times (0,\infty), ~ a \in \mathcal{A} = \mathbb{R}.
\end{equation}
For the goal of estimating  $H_M(\underline{\theta})$, any scale and permutation equivariant estimator is of the form,
\begin{equation}
d_{\phi}\left(\underline{X}\right)=\ln Z_1-\phi\left(V\right),
\end{equation}
with $V=\frac{Z_2}{Z_1}$ and, for some real valued function $\phi(\cdot)$ defined on $[1,\infty)$. We define $\mathcal{M}_2$  the class of all scale and permutation equivariant estimators of $H_M(\underline{\theta})$, given by (3.3). We observe that, for any estimator $\delta_{\phi}$ $\in$ $\mathcal{M}_2$, the risk function $R(\underline{\theta}, \delta)=\mathbb{E}_{\underline{\theta}}\left( \left(\delta-H_M(\underline{\theta})\right)^2\right)$  depends on $\underline{\theta}$ through $\theta=\frac{\mu_1}{\mu_2}$, where $\mu_1=\min\{\mu_1,\mu_2\}$, $\mu_2=\max\{\mu_1,\mu_2\}$. Clearly, $\theta \in (0,1]$. Therefore, for notational simplicity, we denote $R(\underline{\theta}, \delta)$ by $R_{\theta}(\delta)$.\\
As in Section 1, three naive estimators of  $H_M(\underline{\theta})$ based on MLEs and best equivariant estimators of $\theta_1$ and $\theta_2$ (or $\ln \theta_1$ and $\ln \theta_2$) are $d_{N_1}(\underline{X})=\ln Z_1-\ln(\alpha)$, $d_{N_2}(\underline{X})=\ln Z_1-\ln(\alpha+1)$ and $d_{N_3}(\underline{X})=\ln Z_1-\psi(\alpha)$.
Motivated by form of the estimators $d_{N_1}(\underline{X})$, $d_{N_2}(\underline{X})$ and $d_{N_3}(\underline{X})$, we consider a subclass $\mathcal{M}_1=\left\{d_{c}(\cdot):c \in \mathbb{R}\right\}$ of estimators,  where, $d_{c}(\underline{X})=\ln Z_1-c,c \in \mathbb{R}$. We call class $\mathcal{M}_1$, the class of linear, scale and permutation equivariant estimators. Clearly $\mathcal{M}_1 \subseteq \mathcal{M}_2.$ \par 
We will first provide a class of generalized Bayes estimators of $H_M(\underline{\theta})$ that is contained in class $\mathcal{M}_1$. We obtain a result characterizing admissible and inadmissible estimators within the subclass $\mathcal{M}_1=\left\{d_{c}(\cdot):d_{c}(\underline{X})=\ln Z_1-c,c \in \mathbb{R}\right\}$ under the mean squared error criterion. Further, we also derive a sufficient condition for inadmissibility of an arbitrary scale and permutation equivariant estimator of $H_M(\underline{\theta})$ under the criterion of mean squared error. As a consequence of this general result, we obtain shrinkage estimators improving upon various naive estimators belonging to the class $\mathcal{M}_1$. \par 

Utilizing the arguments used in obtaining the class of generalized Bayes estimators $\mathcal{K}_{GB}$ of $H_S(\underline{\theta})$, we obtain the class $\mathcal{M}_{GB}=\left\{d_{\beta}^{GB}:\beta >- \alpha\right\}$, with $d_{\beta}^{GB}(\underline{X})=\ln Z_1-\psi(\alpha+\beta)$ as the class of generalized Bayes estimators of $H_M(\underline{\theta})$ under the squared error loss function \eqref{S3.E2}, and improper priors $(2.1)$. Remarkably, the naive estimator $d_{\psi(\alpha)}(\underline{X})=d_{0}^{GB}(\underline{X})=\ln Z_1-\psi(\alpha)$ (analogue of the best scale equivariant estimators of $\ln \theta_1$ and $\ln \theta_2$) is the generalized Bayes estimator with respect to Jeffrey's non-informative prior density $\Pi_{0}({\underline{\theta}})=\frac{1}{\theta_1 \theta_2}, ~\underline{\theta} \in \Theta$. \par 

\noindent The following lemma is useful in proving Theorem 3.2, reported in the sequel.
\begin{lemma}
Let $U_1=\ln Z_1-H_M(\underline{\theta})$. Then, for any $\underline{\theta} \in \Theta$,
\begin{align}
\mathbb{E}_{\underline{\theta}}\left(U_1 \right)&= \displaystyle{\int_{0}^{\infty}}\ln(z) \left[1-G_{\alpha}\left(\theta z\right)\right]g_{\alpha}(z)dz+ \displaystyle{\int_{0}^{\infty}}\ln(z)\left[1-G_{\alpha}\left(\frac{z}{\theta}\right)\right]g_{\alpha}(z)dz,~ 0< \theta \leq 1.
\end{align}
where $G_{\alpha}\left(\cdot\right) $ and $g_{\alpha}\left(\cdot\right) $ denote the df and pdf, respectively, of $Gamma(\alpha,1)$ distribution.
\end{lemma}
\begin{proof} Similar to the proof of Lemma 2.2.
\end{proof}
\noindent The risk function of an estimator $\delta_{c} \in \mathcal{M}_1$, under the squared error loss function \eqref{S3.E2}, is given by
\begin{equation}
R_{\theta}(\delta_{c})=\mathbb{E}\left[\left(\ln Z_1-c-H_M(\underline{\theta})\right)^2\right],~ 0 < \theta \leq 1, c \in \mathbb{R}.
\end{equation}
For any fixed $0 < \theta \leq 1$, the risk function in $(3.5)$ is minimized at 
\begin{align*}
c \equiv  c^*(\theta)&=\mathbb{E}_{\underline{\theta}}\left[\left(\ln Z_1-H_M(\underline{\theta}) \right)\right]=\mathbb{E}_{\underline{\theta}}\left(U_1\right)\\
&=\displaystyle{\int_{0}^{\infty}}\ln(z) \left[1-G_{\alpha}\left(\theta z\right)\right]g_{\alpha}(z)dz+ \displaystyle{\int_{0}^{\infty}}\ln(z)\left[1-G_{\alpha}\left(\frac{z}{\theta}\right)\right]g_{\alpha}(z)dz.
\end{align*}
Since $c^*(\theta)$ is a decreasing function of $\theta$ $\left(\frac{d}{d \theta} c^*(\theta)= \frac{\Gamma(2\alpha) \theta^{\alpha-1}}{(\Gamma(\alpha))^2(1+\theta)^{2\alpha}} \ln \theta \leq 0,  \forall \theta \in (0, 1]\right)$, we have,
\begin{equation}
\inf\limits_{0 <\theta \leq  1} c^*(\theta)=2\displaystyle{\int_{0}^{\infty}}\ln(z)\left[1-G_{\alpha}\left( z\right)\right]g_{\alpha}(z)dz=c_3(\alpha),~~ \text{say},~~~~\text{and}~~~ \sup\limits_{0 <\theta \leq  1} c^*(\theta)=\psi\left(\alpha\right)=c_1(\alpha), \text{say}
\end{equation}
 On using arguments similar to the ones used in proving Theorem $2.1$, we obtain the following result. 
 \begin{theorem}
 	Let $c_3(\alpha)=2\displaystyle{\int_{0}^{\infty}}\ln(z)\left[1-G_{\alpha}\left( z\right)\right]g_{\alpha}(z)dz$ and recall that $c_1(\alpha)=\psi\left(\alpha\right)$. Then, for estimating $H_M(\underline{\theta})$, under the mean squared error criterion, the estimators in the class $\mathcal{M}_{1,M}=\left\{d_{c} \in \mathcal{M}_1: c \in [c_3(\alpha),c_1(\alpha)]\right\}$ are admissible within the class $\mathcal{M}_1$, whereas the estimators in the class $\mathcal{M}_{1,I}=\left\{d_{c} \in \mathcal{M}_1 : c \in (-\infty,c_3(\alpha)) \cup (c_1(\alpha),\infty) \right\}$ are inadmissible for estimating $H_M(\underline{\theta})$. Moreover, for any $-\infty < b < c \leq c_3(\alpha)$ or $c_1(\alpha) \leq c <b <\infty$,
 	$$R_{\theta}\left(d_{{c}}\right)< R_{\theta}\left(d_{{b}}\right),~~~ \forall~ 0 < \theta \leq 1.$$	
 \end{theorem}
\noindent As a consequence of Theorem $3.1$, we have the following corollary, addressing admissibility and inadmissibility of some naive estimators, including generalized Bayes estimators of $H_M(\underline{\theta})$, belonging to the class $\mathcal{M}_{GB}$.
\begin{corollary}
	\begin{itemize}
		\item[(i)] For estimating $H_M(\underline{\theta})$ under the mean squared error criterion, the naive estimator $d_{\psi(\alpha)}(\underline{X})=\ln Z_1-\psi(\alpha)$ is admissible  within class $\mathcal{M}_1$, whereas the estimators $d_{\ln \alpha}(\underline{X})=\ln Z_1-\ln \alpha$ and $d_{\ln (\alpha+1)}(\underline{X})=\ln Z_1-\ln (\alpha+1)$ are inadmissible and are dominated by $d_{\psi(\alpha)}(\underline{X})=\ln Z_1-\psi(\alpha)$.
		
		\item[(ii)] For any $\alpha >0$, define $\beta_1(\alpha)=\psi^{-1}(c_3(\alpha))-\alpha$. Then
		\begin{itemize}
			\item[(a)] generalized Bayes estimators $\{\delta_{\psi(\alpha+\beta)}: \beta_1(\alpha) \leq \beta \leq 0\}$ are admissible within the class $\mathcal{M}_1$ of naive estimators. 
			\item[(b)] generalized Bayes estimators $\{\delta_{\psi(\alpha+\beta)}: \beta \in (-\alpha,\beta_1(\alpha)) \cup ( 0,\infty)\}$ are inadmissible for estimating  $H_M(\underline{\theta})$. For any $\beta \in (-\alpha,\beta_1(\alpha))$, the generalized Bayes estimator $d_{\psi(\alpha+\beta)}(\underline{X})=\ln Z_1 - \psi(\alpha+\beta)$ is dominated by the naive estimator $d_{c_3(\alpha)}(\underline{X})=\ln Z_1-c_3(\alpha)$, and for any $\beta \in \left(0, \infty\right)$, the generalized Bayes estimator $\delta_{\psi(\alpha+\beta)}$ is dominated by the naive estimator $d_{{\psi(\alpha)}}(\underline{X})=\ln Z_1-\psi(\alpha)$.
		\end{itemize}
\end{itemize}
\end{corollary}

\noindent The following Lemma will be useful in deriving the result stated in Theorem 3.2.
\begin{lemma} \label{S4,L2}
	\begin{itemize}

	\item [(i)] For fixed $ v \in [1,\infty)$, the conditional pdf of $U_1= \ln Z_1-H_M(\underline{\theta})$, given $V=\frac{Z_2}{Z_1}=v$, is given by
	\begin{align*}
	f_{2,\underline{\theta}}(u|v)&=\frac{\frac{1}{\theta^{\alpha}}e^{2\alpha u}e^{-\left(1+\frac{v}{\theta}\right)e^{u}}+\theta^{\alpha} e^{2\alpha u}e^{-\left(1+v\theta\right)e^{u}}}{\Gamma(2\alpha)\left[\frac{1}{\theta^{\alpha}\left(1+\frac{v}{\theta}\right)^{2\alpha}}+\frac{\theta^{\alpha}}{\left(1+v \theta\right)^{2\alpha}}\right]}, ~ -\infty <u < \infty,~0 < \theta \leq 1.
	\end{align*}
\item [(ii)]	For any fixed $\alpha >0 $, define
\begin{align}
k_{v}(\theta)&=\frac{(1+v \theta)^{2 \alpha} \ln\left(1+\frac{v}{\theta}\right)+(\theta+v)^{2\alpha}\ln\left(1+v \theta\right)}{(1+v \theta)^{2\alpha}+(\theta+v)^{2\alpha}}~,~ v \in [1,\infty),~ 0 < \theta \leq 1.
\end{align}
Then, for $1 \leq v \leq \min\left\{1+\frac{1}{2\alpha},1+\sqrt{3}\right\} $,~$ \inf\limits_{0< \theta \leq  1} k_{v}(\theta)=\ln(1+v)$. 
\end{itemize}

\end{lemma}
\begin{proof} 
(i) Similar to the proof of Lemma $2.3.$\\	
(ii) Since $\lim\limits_{\theta \to 0}k_v(\theta)=\infty$, we have, $\sup\limits_{0< \theta \leq  1} k_{v}(\theta)=\infty$. Note that $k_v(1)=\ln (1+v), v \geq 1.$ Also, for $0 < \theta \leq 1$ and $v \geq 1$, we have $k_v(\theta) \geq k_v(1)$, if, and only if,
\begin{align}
 (1+v \theta)^{2 \alpha} \ln\left(\frac{1+v/\theta}{1+v}\right)& \geq ( \theta+v )^{2 \alpha} \ln\left(\frac{1+v}{1+v \theta}\right) \nonumber \\
\iff \left(\frac{1+v \theta}{\theta +v}\right)^{2 \alpha} & \geq \frac{\ln\left(\frac{1+v}{1+v \theta}\right) }{\ln\left(\frac{1+v/\theta}{1+v}\right)} .\nonumber 
\end{align}
Since, for $2 \alpha \leq \frac{1}{v-1}$, $\left(\frac{1+v \theta}{\theta +v}\right)^{2 \alpha} \geq \left(\frac{1+v \theta}{\theta +v}\right)^{\frac{1}{v-1}}$, to show that $\inf\limits_{0< \theta \leq  1} k_{v}(\theta)=\ln(1+v)$, for $1 \leq v \leq \min\left\{1+\frac{1}{2\alpha},1+\sqrt{3}\right\} $, it suffices to show that,
\begin{align}
\left(\frac{1+v \theta}{\theta +v}\right)^{\frac{1}{v-1}} &\geq \frac{\ln\left(\frac{1+v}{1+v \theta}\right) }{\ln\left(\frac{1+v/\theta}{1+v}\right)},
\end{align}
provided $1 \leq v \leq \min\left\{1+\frac{1}{2\alpha},1+\sqrt{3}\right\} $. Let $\frac{1+v \theta}{\theta +v}=x$ and $\frac{1}{v-1}=\beta$, so that $\beta >0$ and $ x \in \left(\frac{\beta}{1+\beta},1\right)$. Proving inequality $(3.8)$ is equivalent to showing
\begin{align}
\Psi(x)=-x ^{\beta}\left[\ln (\beta+1)+\ln \left(x-\frac{\beta}{\beta+1}\right)\right] - \ln \beta -\ln \left(\frac{\beta+1}{\beta}-x\right)+\ln x & \geq 0,
\end{align}
for all $x \in \left(\frac{\beta}{\beta+1}-1\right)$ and $\beta \geq \max\{2\alpha,\frac{1}{\sqrt{3}}\}$.\par 

Note that, $\lim\limits_{x \to \frac{\beta}{\beta+1} } \Psi(x)= \infty$ and $\lim\limits_{x \to 1} \Psi(x)=0$. Thus, to prove $(3.9)$, it is sufficient to show that $\Psi'(x) \leq 0$, $\forall x \in \left(\frac{\beta}{\beta+1} ,1 \right)$, $\beta \geq \max\{2\alpha,\frac{1}{\sqrt{3}}\}$. We have 
$$ \Psi'(x)= x^{\beta-1}k_1(x), ~~ x \in \left(\frac{\beta}{\beta+1},1\right), ~~ \beta \geq \max\left\{2\alpha,\frac{1}{\sqrt{3}}\right\},$$ 
where
\begin{align*}
k_1(x)&= -\frac{x}{x-\frac{\beta}{\beta+1}}-\beta\left\{\ln(\beta+1)+\ln\left(x-\frac{\beta}{\beta+1}\right)\right\}-\frac{\beta+1}{\beta x^{\beta}\left(\frac{\beta+1}{\beta}-x\right)},~x \in \left(\frac{\beta}{\beta+1},1\right), ~~ \beta \geq \max\left\{2\alpha,\frac{1}{\sqrt{3}}\right\}.
\end{align*}
We have $\lim\limits_{x \to \frac{\beta}{\beta+1}} k_1(x)=-\infty$ and $\lim\limits_{x \to 1} k_1(x)=0$. Thus, to show that $\Psi'(x) \leq 0$, $\forall x \in \left( \frac{\beta}{\beta+1},1\right),~\beta \geq \max\left\{2\alpha,\frac{1}{\sqrt{3}}\right\} $, it suffices to show that $k_1'(x) \geq 0$, $\forall x \in \left( \frac{\beta}{\beta+1},1\right),~\beta \geq \max\left\{2\alpha,\frac{1}{\sqrt{3}}\right\} $. We have
\begin{align*}
k_1'(x)= (1-x) \left[\frac{\beta}{\left(x-\frac{\beta}{\beta+1}\right)^2} - \frac{(\beta+1)^2}{\beta x^{\beta+1}\left(\frac{\beta+1}{\beta}-x\right)^2}\right],~~x \in \left( \frac{\beta}{\beta+1},1\right),~\beta \geq \max\left\{2\alpha,\frac{1}{\sqrt{3}}\right\}.
\end{align*}
To show that $ k_1'(x) \geq 0$, $\forall x \in \left( \frac{\beta}{\beta+1},1\right),~\beta \geq \max\left\{2\alpha,\frac{1}{\sqrt{3}}\right\} $, we will show that
\begin{align*}
k_2(x)&=\beta x^{\frac{\beta+1}{2}}\left(\frac{\beta+1}{\beta}-x\right)-(\beta+1)\left(x-\frac{\beta}{\beta+1}\right)\\
&=(\beta+1)x^{\frac{\beta+1}{2}} -\beta x^{\frac{\beta+3}{2}}-(\beta+1)x+\beta \geq 0,~\forall x \in \left( \frac{\beta}{\beta+1},1\right),~\beta \geq \max\left\{2\alpha,\frac{1}{\sqrt{3}}\right\}.
\end{align*}
We have, for $ x \in \left( \frac{\beta}{\beta+1},1\right),~\beta \geq \max\left\{2\alpha,\frac{1}{\sqrt{3}}\right\} $,
\begin{align*}
k_2'(x)&=\frac{(\beta+1)^2}{2}x^{\frac{\beta-1}{2}}-\frac{\beta(\beta+3)}{2}x^{\frac{(\beta+1)}{2}}-(\beta+1)\\
k_2''(x)&=\frac{(\beta+1)x^{\frac{(\beta-3)}{2}}}{4}\left[\beta^2-1-\beta(\beta+3)x\right]\\
&\leq \frac{(\beta+1)x^{\frac{(\beta-3)}{2}}}{4}\left[\beta^2-1-\frac{\beta^2(\beta+3)}{\beta+1}\right]=-\frac{(2\beta^2+\beta+1)}{4}x^{\frac{\beta-3}{2}}<0
\end{align*}
\begin{align*}
\implies k_2'(x) \leq \lim\limits_{x \to \frac{\beta}{\beta+1}}k_2'(x)&=\frac{\beta^{\frac{\beta-1}{2}}}{2(\beta+1)^{\frac{\beta+1}{2}}}\left[3\beta+1-2\beta^2\left(1+\frac{1}{\beta}\right)^{\frac{\beta+3}{2}}\right]\\
&\geq \frac{\beta^{\frac{\beta-1}{2}}}{2(\beta+1)^{\frac{\beta+1}{2}}}\left[3\beta+1-2\beta^2\left(1+\frac{\beta+3}{2\beta}\right)\right]\\
&=\frac{\beta^{\frac{\beta-1}{2}}}{2(\beta+1)^{\frac{\beta+1}{2}}}\left[1-3\beta^2\right] \leq 0,\\
\implies k_2(x)&\geq \lim\limits_{x \to 1}k_2(x)=0.
\end{align*}
Hence, the assertion follows.
\end{proof}	
\noindent  Let $d_{\phi}(V)=\ln Z_1-\phi(V)$ be a scale and permutation equivariant estimator of $H_M(\underline{\theta})$. For any fixed $v \in [1, \infty)$, the conditional risk of $d_{\phi}(V)$, given $V=v$, is obtained as
\begin{equation}
R_1(\theta,\phi(v))=\mathbb{E}_{\underline{\theta}}\left[\left(\ln Z_1-\phi(V)-H_M(\underline{\theta})\right)^2\Big|V=v\right],~ 0 <\theta \leq 1.
\end{equation}
 For any fixed $v \in [1,\infty)$ and $\theta \in (0, 1]$, the choice of $\phi(\cdot)$ that minimizes the conditional risk  $(3.10)$ is obtained as
\begin{align*}
\phi_{\theta}(v)&=\mathbb{E}_{\underline{\theta}}\left[\left(\ln Z_1-H_M(\underline{\theta})\right)\Big|V=v\right]\\
&=\mathbb{E}_{\underline{\theta}}\left[U_1|V=v\right].
\end{align*}
Using Lemma $3.2.(i)$, we obtain
$\phi_{\theta}(v)= \psi(2\alpha)-k_v(\theta), ~ 0 < \theta \leq 1, v \in [1,\infty),$
where, $ k_v(\theta)$ is given by $(3.7)$. Further, using Lemma $3.2 ~(ii)$, we have,
for $1 \leq v \leq \min\left\{1+1/(2\alpha),1+\sqrt{3}\right\} $,
\begin{equation}
\sup\limits_{0 < \theta \leq  1} \phi_{v}(\theta)=\psi(2\alpha)-\ln(1+v)=\phi_{*}(v),~ \text{(say)}. 
\end{equation}
The following theorem is an analogue of Theorem 2.2 and provides a sufficient condition for inadmissibility of an arbitrary scale and permutation equivariant estimator of $H_M(\underline{\theta})$, under the mean squared error criterion. Since the proof of the following theorem is similar to that of Theorem 2.2, it is omitted here.
\begin{theorem}
	For estimating $H_M(\underline{\theta})$ under the mean squared error criterion, consider
	a scale and permutation equivariant estimator, $d_{\phi}\left(\underline{X}\right)=\ln Z_1-\phi\left(V\right) $ where, $V=Z_2/Z_1$ and $ \phi(\cdot)$
	is a real valued function defined on $[1,\infty)$. Let $$\mathbb{P}_{\underline{\theta}} \left[\left\{V : V \leq \min\left\{1+\frac{1}{2\alpha},1+\sqrt{3}\right\} \text{and}~ \phi(V) > \phi_*(V) \right\} \right]>0,~\text{for some}~ \underline{\theta} \in \Theta,$$ where, $\phi_*(V) $ is defined by $(3.11)$. Then, the estimator $\delta_{\phi}(\cdot)$ is dominated by $d_{\phi}^I\left(\underline{X}\right)=\ln Z_1-\phi^I\left(V\right),$ where,
\begin{equation}
	\phi^I\left(v\right)=\begin{cases}
	\phi_*(v), & \text{ if }~ \phi(v) > \phi_*(v)~ \text{and}~ v \leq \min\left\{1+\frac{1}{2\alpha},1+\sqrt{3}\right\} ,\\
	\phi(v), & \text{otherwise}.
	\end{cases}
\end{equation}
\end{theorem}
\noindent For estimating $H_M(\underline{\theta})$ under the mean squared error criterion, we obtain the following result on inadmissibility of any naive estimator $d_{c}(\underline{X})$. As a consequence of Theorem 3.2, we also provide the shrinkage estimators dominating on $d_{c}(\underline{X})$.
\begin{corollary}
	\begin{itemize}
	\item [(i)]	 For estimating $H_M(\underline{\theta})$ under the mean squared error criterion, any natural estimator $d_{c}(\underline{X})=\ln Z_1-c$,  with $c> \psi(2\alpha)-\ln (1+\lambda) ;~ \lambda= \min\left\{1+\frac{1}{2\alpha},1+\sqrt{3}\right\}$, is inadmissible and is dominated by the shrinkage estimator
		\begin{equation}
		d_{c}^{(S)}(\underline{X})=\begin{cases*}
		\ln(Z_1+Z_2)-\psi(2\alpha), & \parbox[t]{6cm}{$\text{if}~ 1\leq \frac{Z_2}{Z_1} \leq \min\left\{1+\frac{1}{2\alpha},1+\sqrt{3}\right\}  $\\ \hspace*{2cm} and \\ $ \frac{Z_2}{Z_1} > e^{\psi(2\alpha)-c}-1$}\\
		\ln Z_1 -c, & \parbox[t]{5.5cm}{$\text{otherwise}.$} \\
		\end{cases*}
		\end{equation}

\item[(ii)]	For estimating $H_M(\underline{\theta})$ under the mean squared error criterion, the naive estimators $d_{\ln \alpha}(\underline{X})=\ln Z_1-\ln \alpha$ and $d_{\ln (\alpha+1)}(\underline{X})=\ln Z_1-\ln (\alpha+1)$ are inadmissible  and are dominated by the shrinkage estimators
	\begin{equation}
	d_{\ln \alpha}^{(S)}(\underline{X})=\begin{cases*}
	\ln(Z_1+Z_2)-\psi(2\alpha), & \parbox[t]{6cm}{$\text{if},~ 1\leq \frac{Z_2}{Z_1} \leq \min\left\{1+\frac{1}{2\alpha},1+\sqrt{3}\right\} $}\\
	\ln Z_1 -\ln \alpha, & \parbox[t]{5.5cm}{$\text{otherwise}$} \\
	\end{cases*}
	\end{equation}
and 	
	\begin{equation}
	d_{\ln (\alpha+1)}^{(S)}(\underline{X})=\begin{cases*}
	\ln(Z_1+Z_2)-\psi(2\alpha), & \parbox[t]{6cm}{$\text{if},~ 1\leq \frac{Z_2}{Z_1} \leq \min\left\{1+\frac{1}{2\alpha},1+\sqrt{3}\right\}$}\\
	\ln Z_1 -\ln(\alpha+1), & \parbox[t]{5.5cm}{$\text{otherwise}$} \\
	\end{cases*}
	\end{equation}
	respectively.
	\end{itemize}
\end{corollary}
\section{Simulation}
In this section, we present results of a simulation study carried out to numerically assess the performances of various estimators of $H_S(\underline{\theta})$ and $H_M(\underline{\theta})$  in terms of the mean squared error (mse) and the absolute bias. For the numerical study, we have taken random samples of size $n=3,5,10,15,20$ etc. from relevant gamma distributions for different configurations of $\alpha$. The risk (mse) and the absolute bias values of various proposed estimators are simulated based on $60,000$ samples of size $n$. In Figure $4.2$ (Figure $4.4$), we have plotted simulated mses of the naive estimators $\delta_{\ln \alpha}$, $\delta_{\ln (\alpha+1)}$ and $\delta_{\psi(\alpha)}$ ( $d_{\ln \alpha}$, $d_{\ln (\alpha+1)}$ and $d_{\psi(\alpha)}$) of $H_S(\underline{\theta})$ ( $H_M(\underline{\theta})$). In Figures $4.5$ and $4.6$, we have plotted simulated absolute bias values of the three naive estimators of $H_S(\underline{\theta})$ and $H_M(\underline{\theta})$, respectively. In Figure $4.3$, we have plotted the mse values of the estimators $\delta_{\ln \alpha}$ and $\delta_{\ln (\alpha+1)}$ and their improved versions $\delta^{(S)}_{\ln \alpha}$ and $\delta^{(S)}_{\ln (\alpha+1)}$, respectively. In Figure $4.7$, we have plotted mse values of estimators $d_{\ln \alpha}$ and $d_{\ln (\alpha+1)}$ and their improved versions $d^{(S)}_{\ln \alpha}$ and  $d^{(S)}_{\ln (\alpha+1)}$, respectively.
Based on these plotted graphs, the following conclusions are obvious:
\begin{itemize}
	\item [(i)] For estimating $H_S(\underline{\theta})$, under the mean squared error criterion, among the three naive estimators ($\delta_{\ln (\alpha)}$, $\delta_{\ln (\alpha+1)}$ and $\delta_{\psi(\alpha)}$), the estimator $\delta_{\psi(\alpha)}$ performs better except for smaller values of $\mu$ (see Figure 4.2.).
	
	\item[(ii)] For estimating $H_S(\underline{\theta})$, mean squared errors(mses) of the estimators $\delta^{(S)}_{\ln \alpha}$ and $\delta^{(S)}_{\ln (\alpha+1)}$ are nowhere larger than mses of $\delta_{\ln \alpha}$ and $\delta_{\ln (\alpha+1)}$ respectively (see Figure 4.3.).
	
	\item[(iii)] For estimating $H_M(\underline{\theta})$ under the mean squared error criterion, the estimator $d_{\psi(\alpha)}$ uniformly performs better among all the three naive estimators ($d_{\ln \alpha}$, $d_{\ln (\alpha+1)}$ and $d_{\psi(\alpha)}$). (see Figure 4.4.)
	
	\item[(iv)] For estimating $H_S(\underline{\theta})$, the estimator $\delta_{\ln (\alpha+1)}$ has smaller absolute bias than the other two naive estimators ($\delta_{\ln (\alpha)}$ and $\delta_{\psi(\alpha)}$) for smaller values of $\mu$. For large values of $\mu$, $\delta_{\psi(\alpha)}$ seems to be a good choice, in terms of the absolute bias (see Figure 4.5.).
	\item[(v)] For estimating $H_M(\underline{\theta})$ in terms of absolute bias, the naive estimators $d_{\ln \alpha}$ and $d_{\ln (\alpha+1)}$ are uniformly dominated by the naive estimator $d_{\psi(\alpha)}$. Moreover, $d_{\ln \alpha}$ dominates $d_{\ln (\alpha+1)}$ in terms of the absolute bias (see Figure 4.6.).	
	\item[(vi)] For estimating $H_M(\underline{\theta})$, the mean squared error of the estimators $d^{(S)}_{\ln \alpha}$ and $d^{(S)}_{\ln (\alpha+1)}$ are nowhere larger than mses of $d_{\ln \alpha}$ and $d_{\ln (\alpha+1)}$ respectively (see Figure 4.7.).
	\item[(vii)] Based on the simulation study, we observe that the estimator $\delta_{\psi}$ ($d_{\psi}$) performs better (in terms of the mse and the absolute bias), for estimating the selected Shannon entropy $H_S(\underline{\theta})$($H_M(\underline{\theta})$).
\end{itemize}
\FloatBarrier
\begin{figure}[ht!]
	\begin{center}
		\begin{subfigure}[b]{0.5\textwidth}
			\centering
			\includegraphics[height=2.3in,width=6.5cm]{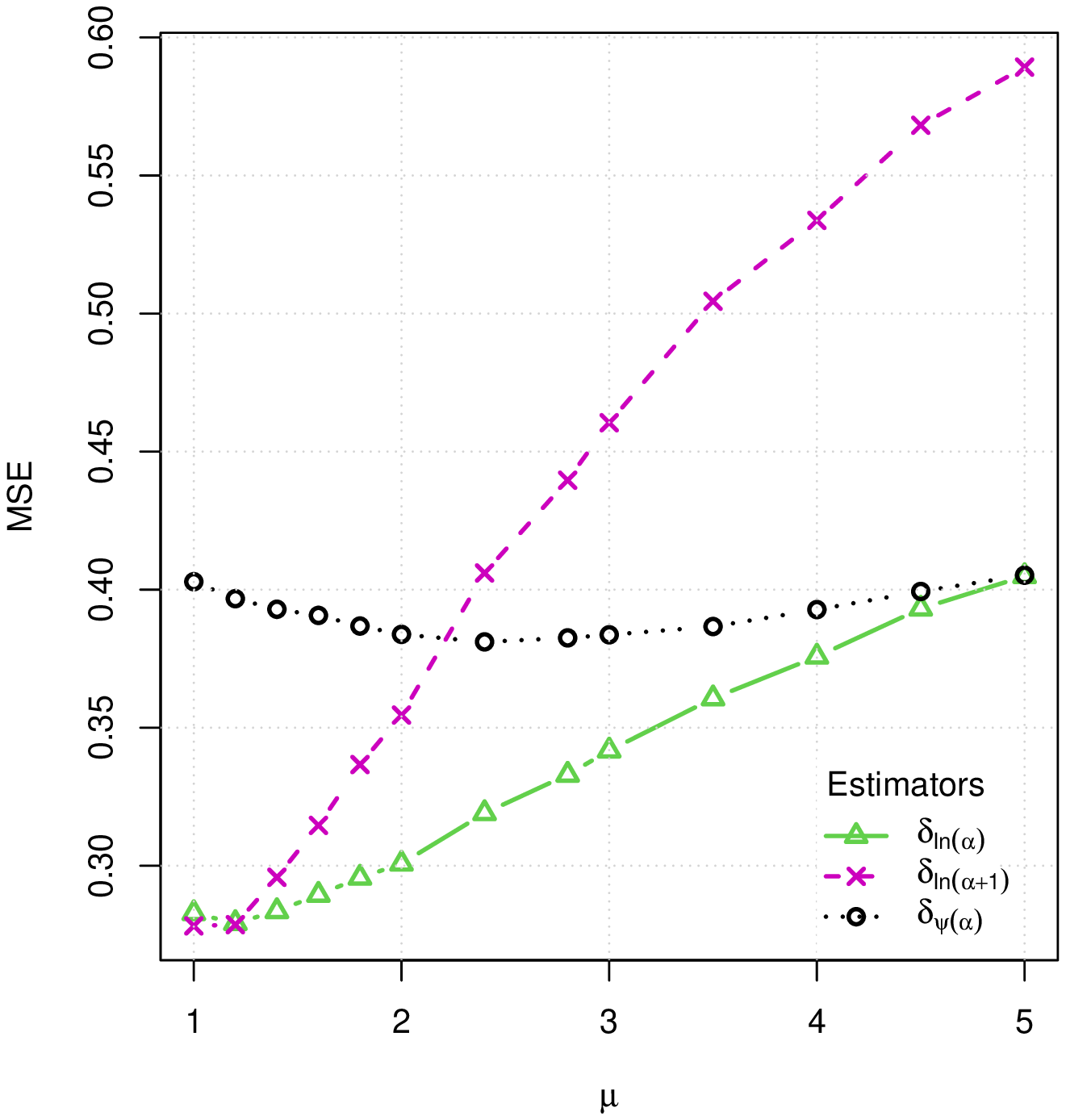}
			\caption{$n=5$, $\alpha=0.5$ }
		\end{subfigure}%
		\begin{subfigure}[b]{0.5\textwidth}
			\centering
			\includegraphics[height=2.3in,width=6.5cm]{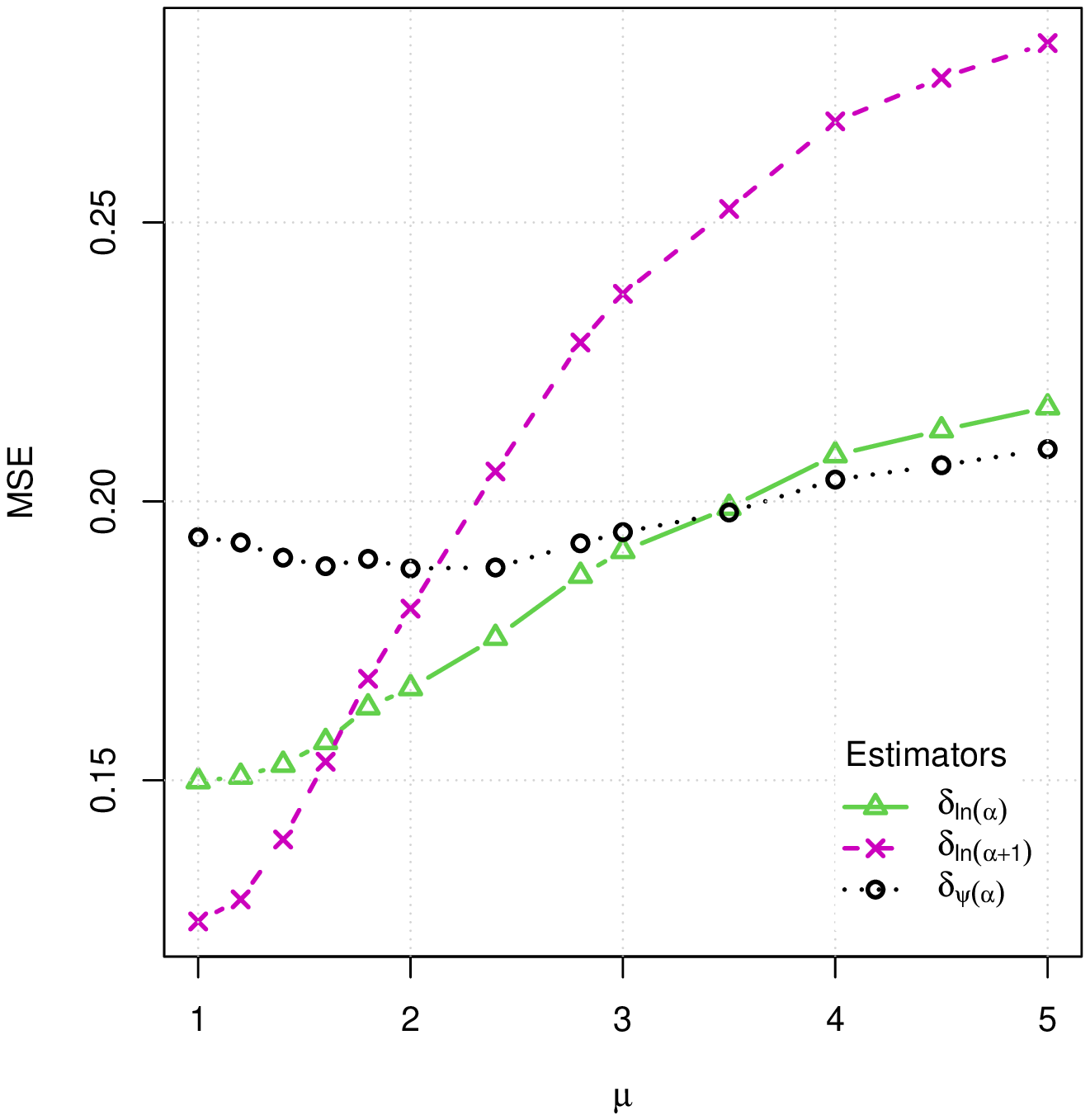}
			\caption{$n=10$, $\alpha=0.5$  }
		\end{subfigure}%
		
		\begin{subfigure}[b]{0.5\textwidth}
			\centering
			\includegraphics[height=2.3in,width=6.5cm]{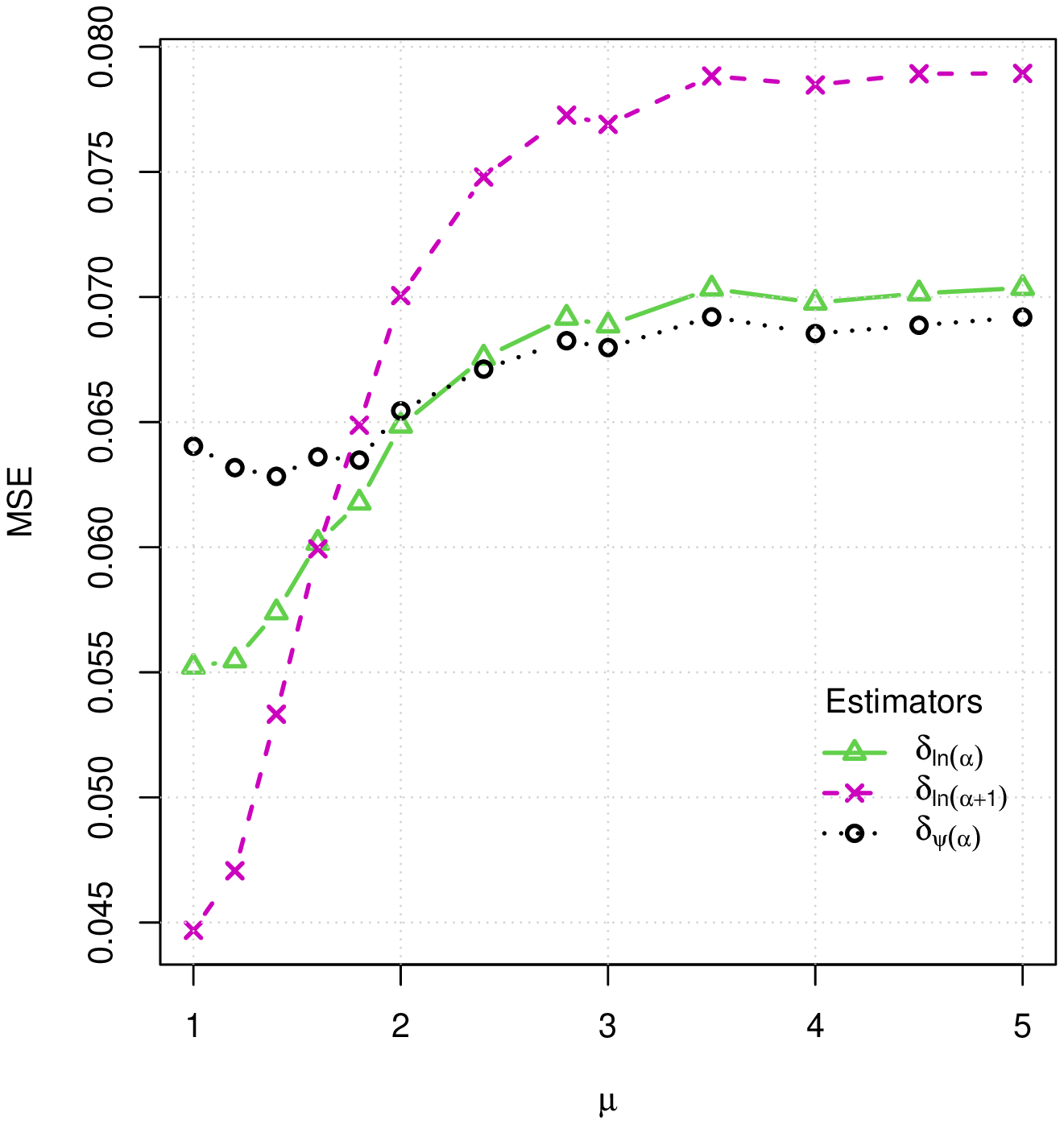}
			\caption{$n=15$, $\alpha=1$  }
		\end{subfigure}%
		\begin{subfigure}[b]{0.5\textwidth}
			\centering
			\includegraphics[height=2.3in,width=6.5cm]{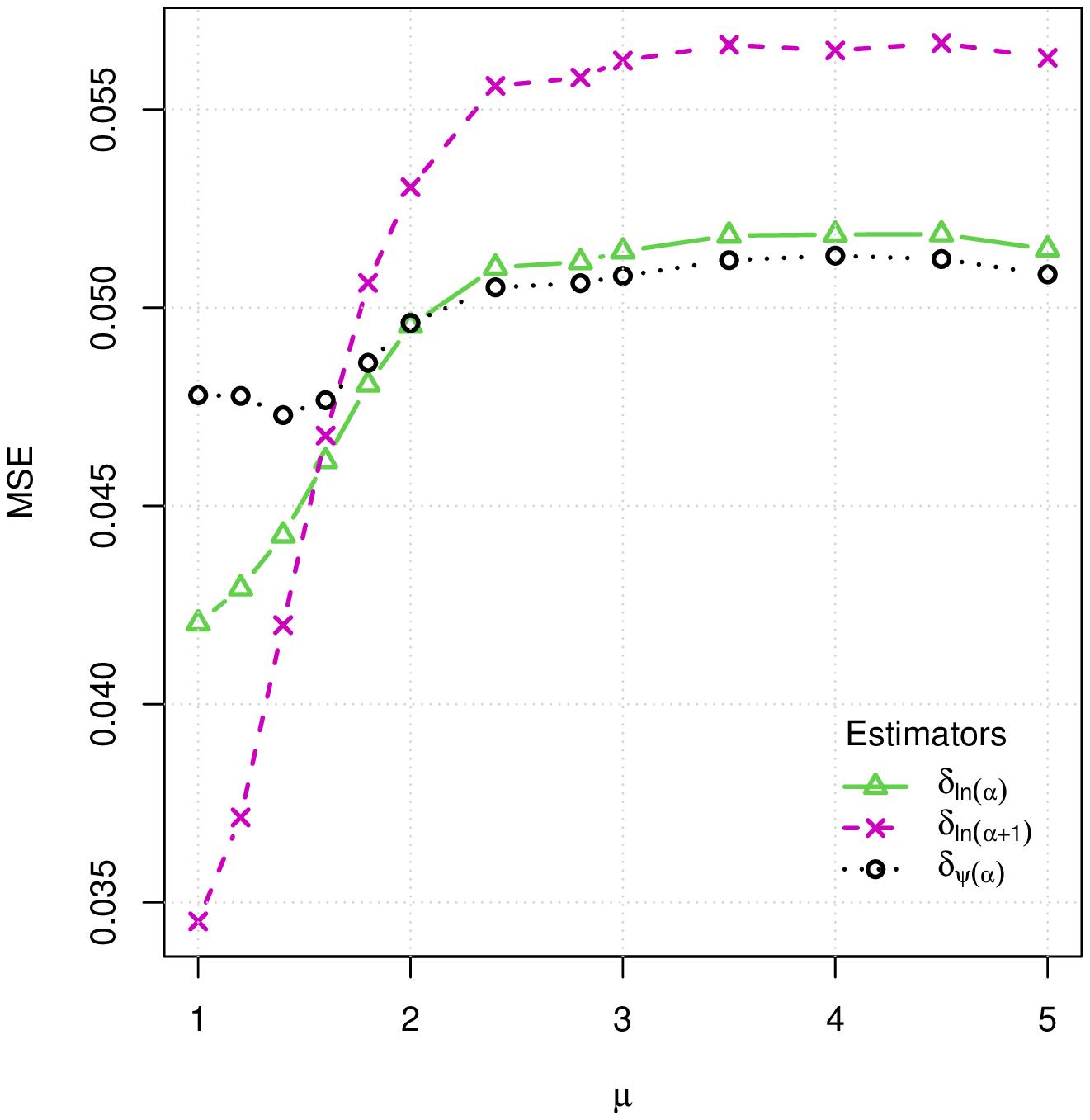}
			\caption{$n=20$, $\alpha=1.5$  }
		\end{subfigure}%
		
		\begin{subfigure}[b]{0.5\textwidth}
			\centering
			\includegraphics[height=2.3in,width=6.5cm]{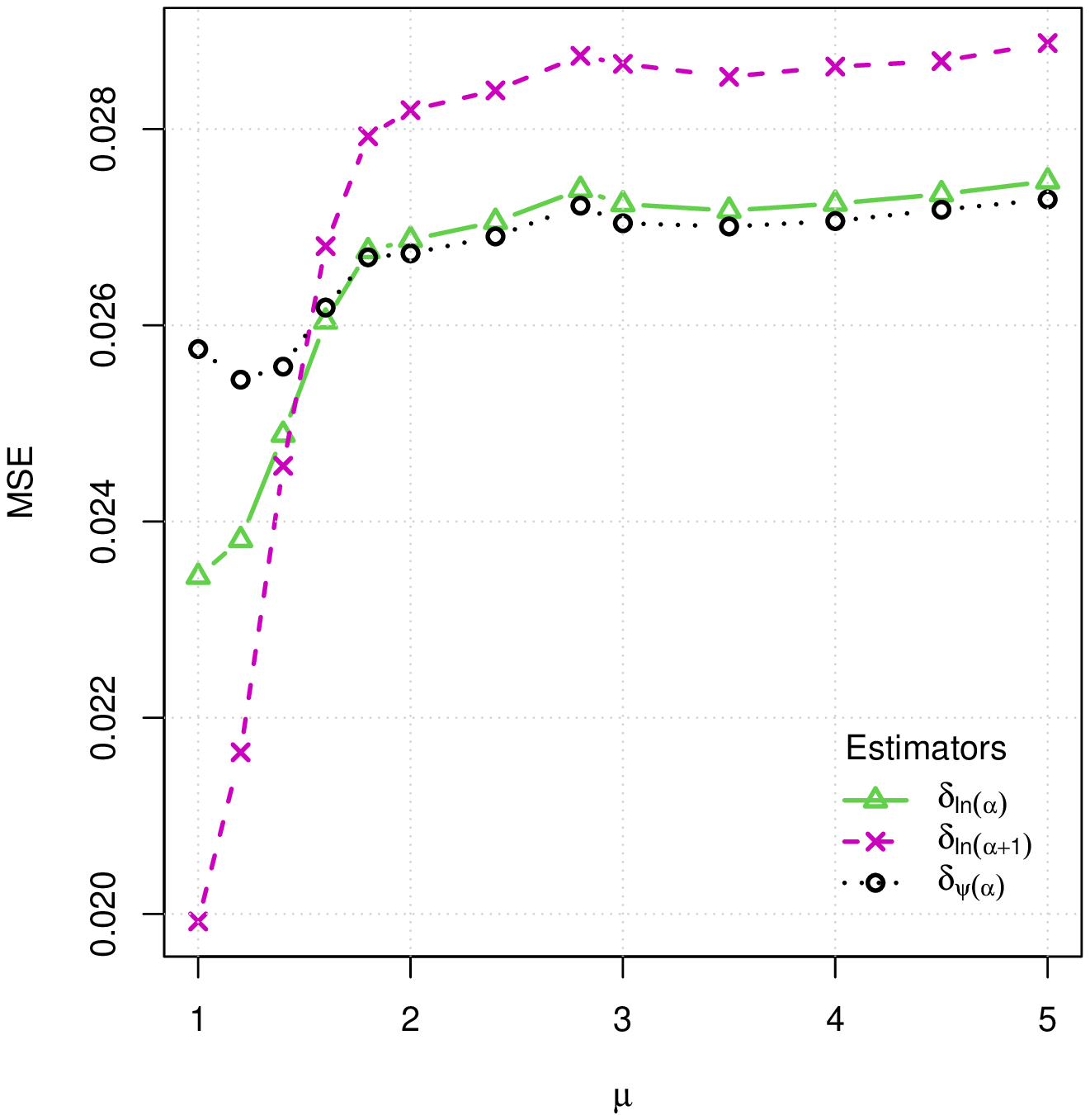}
			\caption{$n=25$, $\alpha=2$  }
		\end{subfigure}%
		\begin{subfigure}[b]{0.5\textwidth}
			\centering
			\includegraphics[height=2.3in,width=6.5cm]{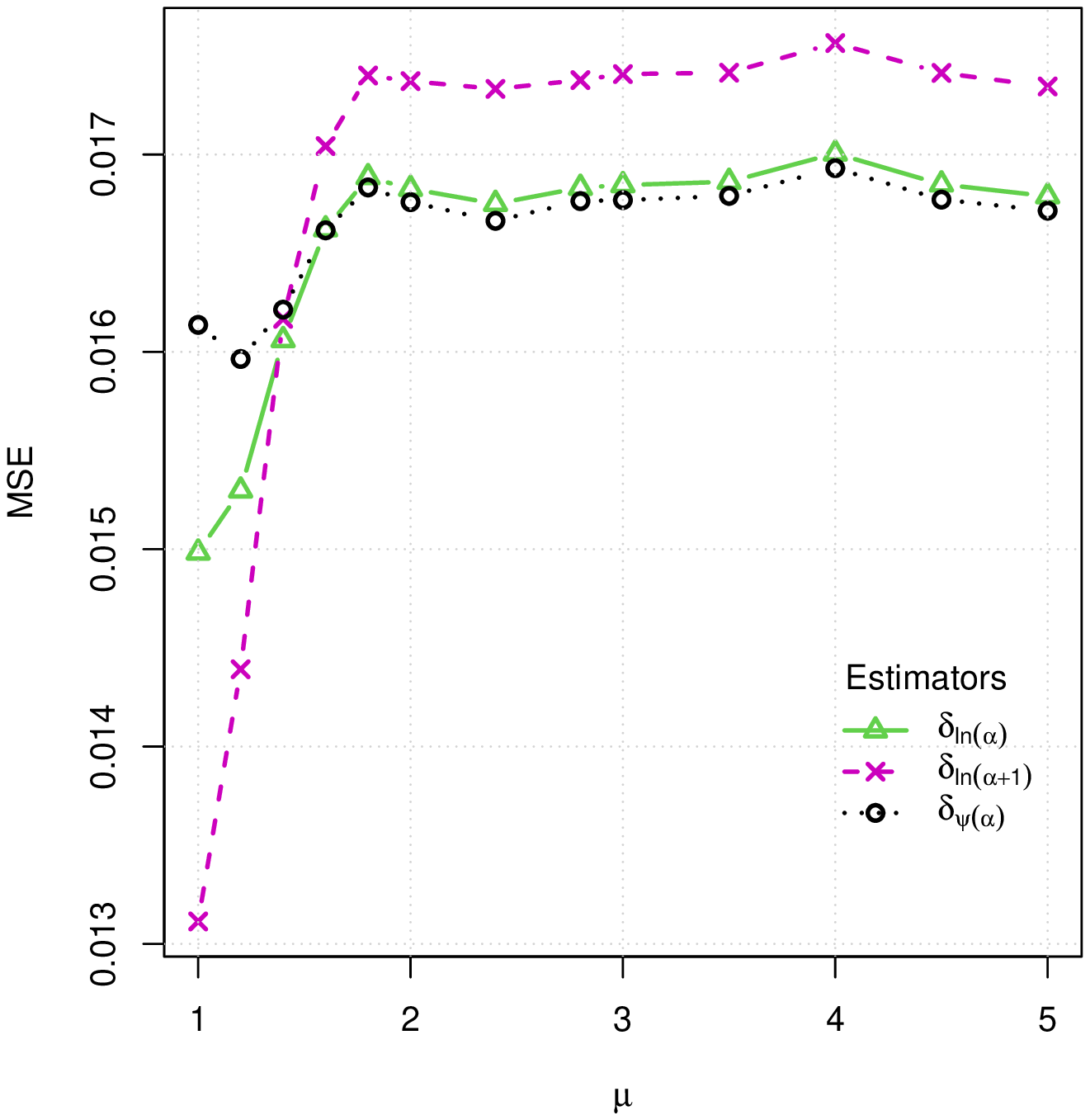}
			\caption{$n=30$, $\alpha=2.5$  }
		\end{subfigure}%
		
	\end{center}
	\caption{\textbf{Mean squared error plots of the three natural estimators ($\delta_{\ln (\alpha)}, \delta_{\ln (\alpha+1)}$ and $\delta_{\psi(\alpha)}$) of $H_S(\underline{\theta})$.}
	}%
	\label{fig:subfigures}
\end{figure}
\FloatBarrier
\FloatBarrier
\begin{figure}[ht!]
	\begin{center}
		\begin{subfigure}[b]{0.5\textwidth}
			\centering
			\includegraphics[height=2.3in,width=6.5cm]{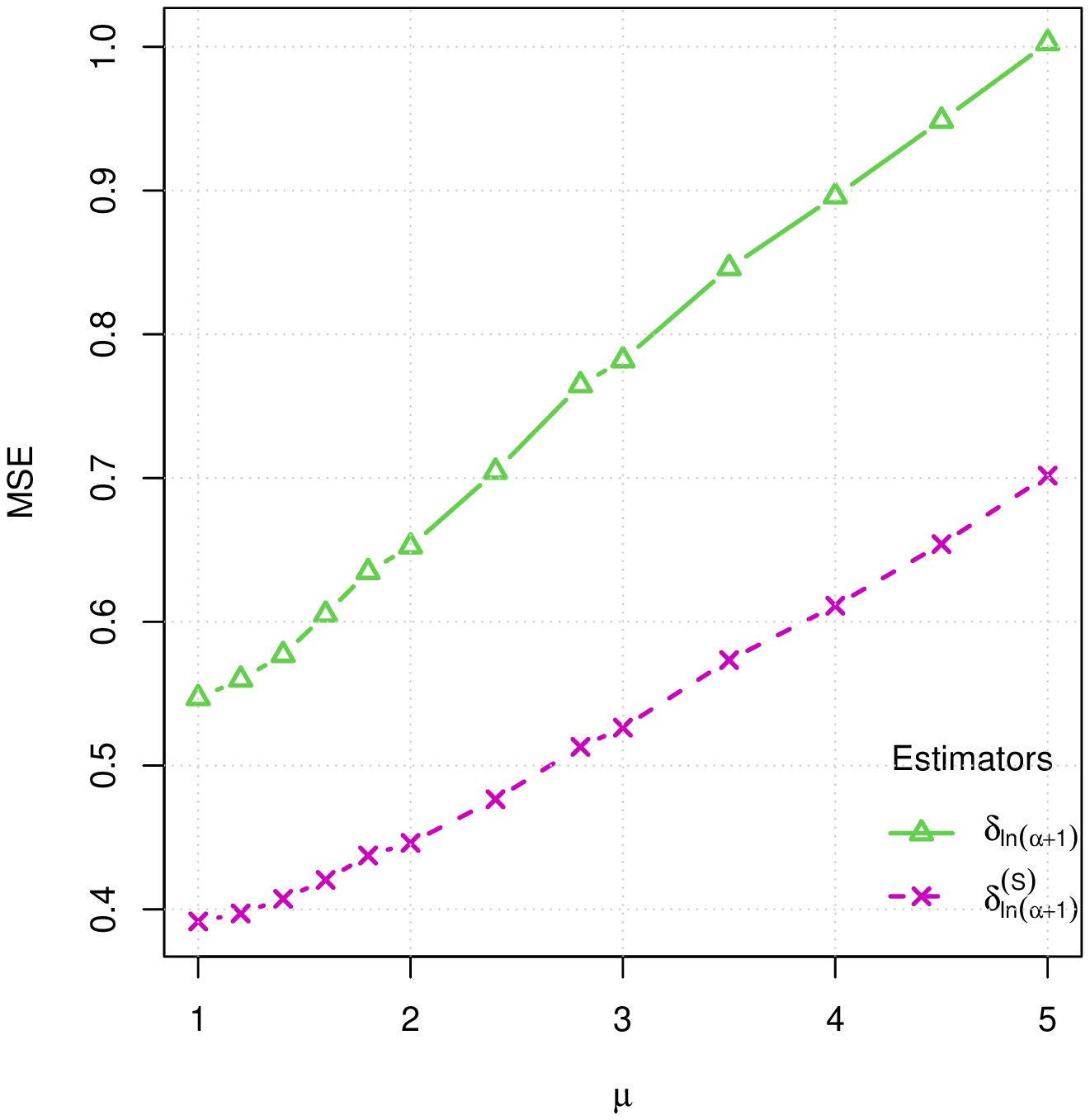}
			\caption{$n=5$, $\alpha=0.5$ }
		\end{subfigure}%
		\begin{subfigure}[b]{0.5\textwidth}
			\centering
			\includegraphics[height=2.3in,width=6.5cm]{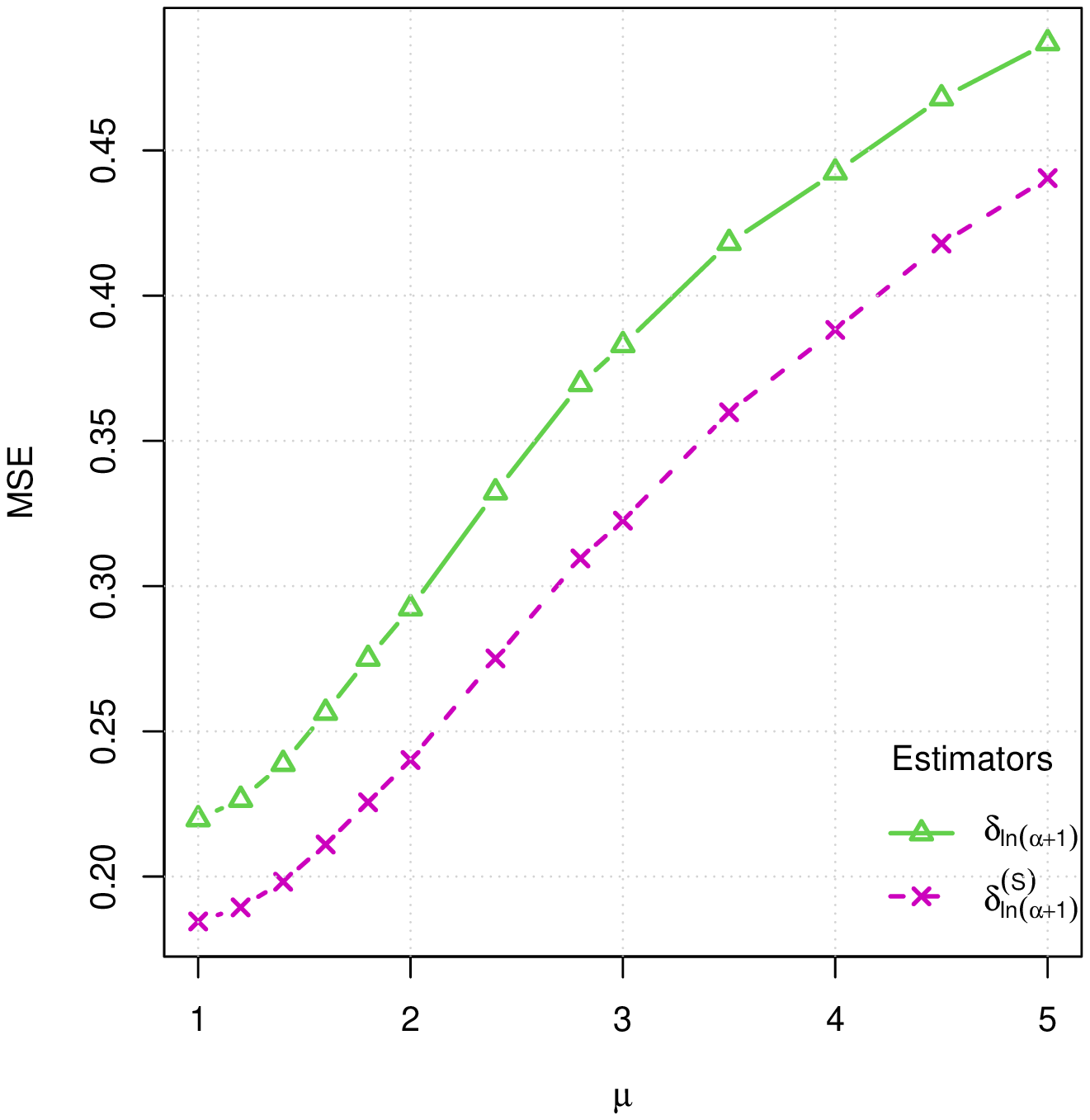}
			\caption{$n=10$, $\alpha=0.5$  }
		\end{subfigure}%
		
		\begin{subfigure}[b]{0.5\textwidth}
			\centering
			\includegraphics[height=2.3in,width=6.5cm]{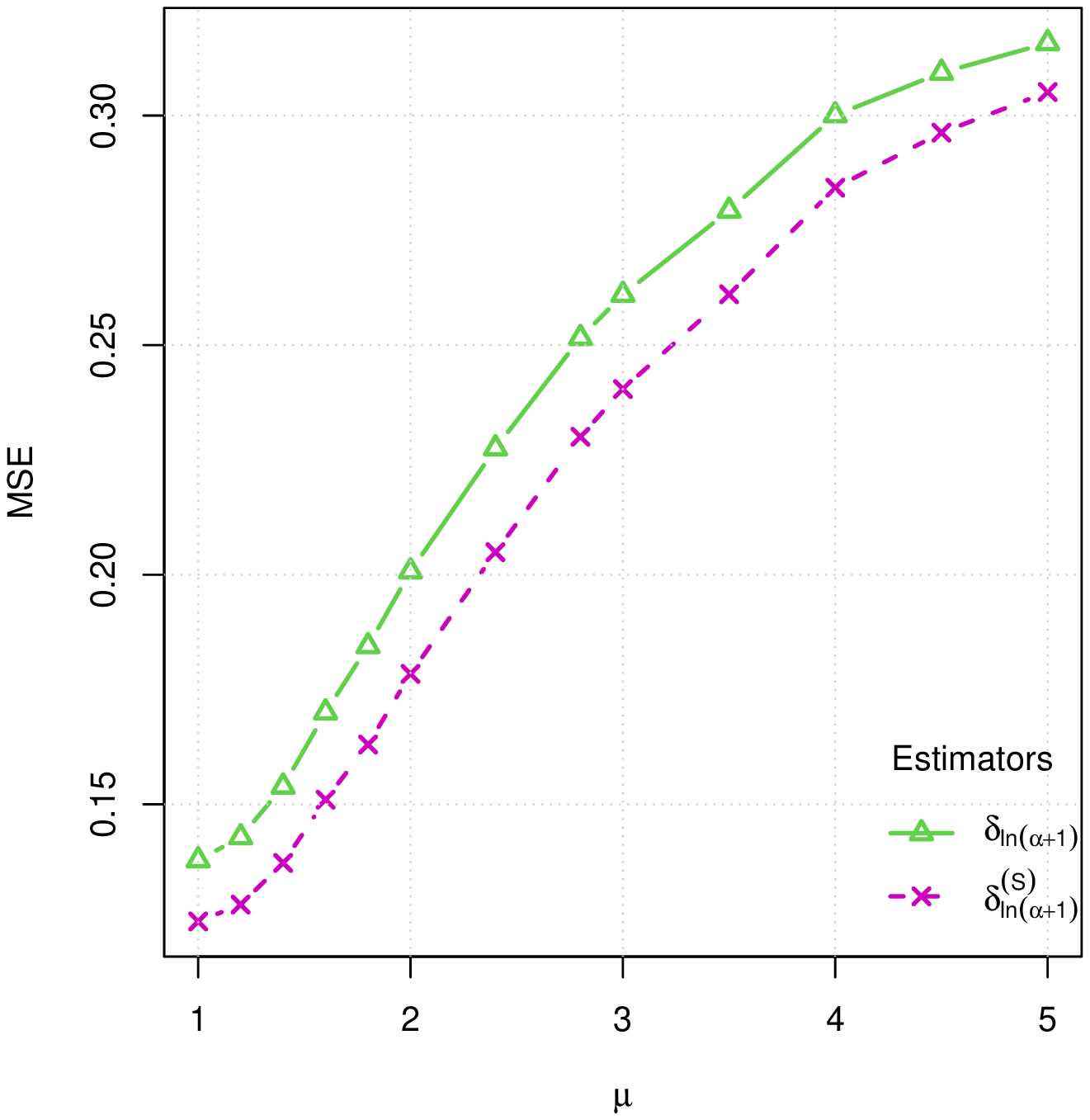}
			\caption{$n=15$, $\alpha=1$  }
		\end{subfigure}%
		\begin{subfigure}[b]{0.5\textwidth}
			\centering
			\includegraphics[height=2.3in,width=6.5cm]{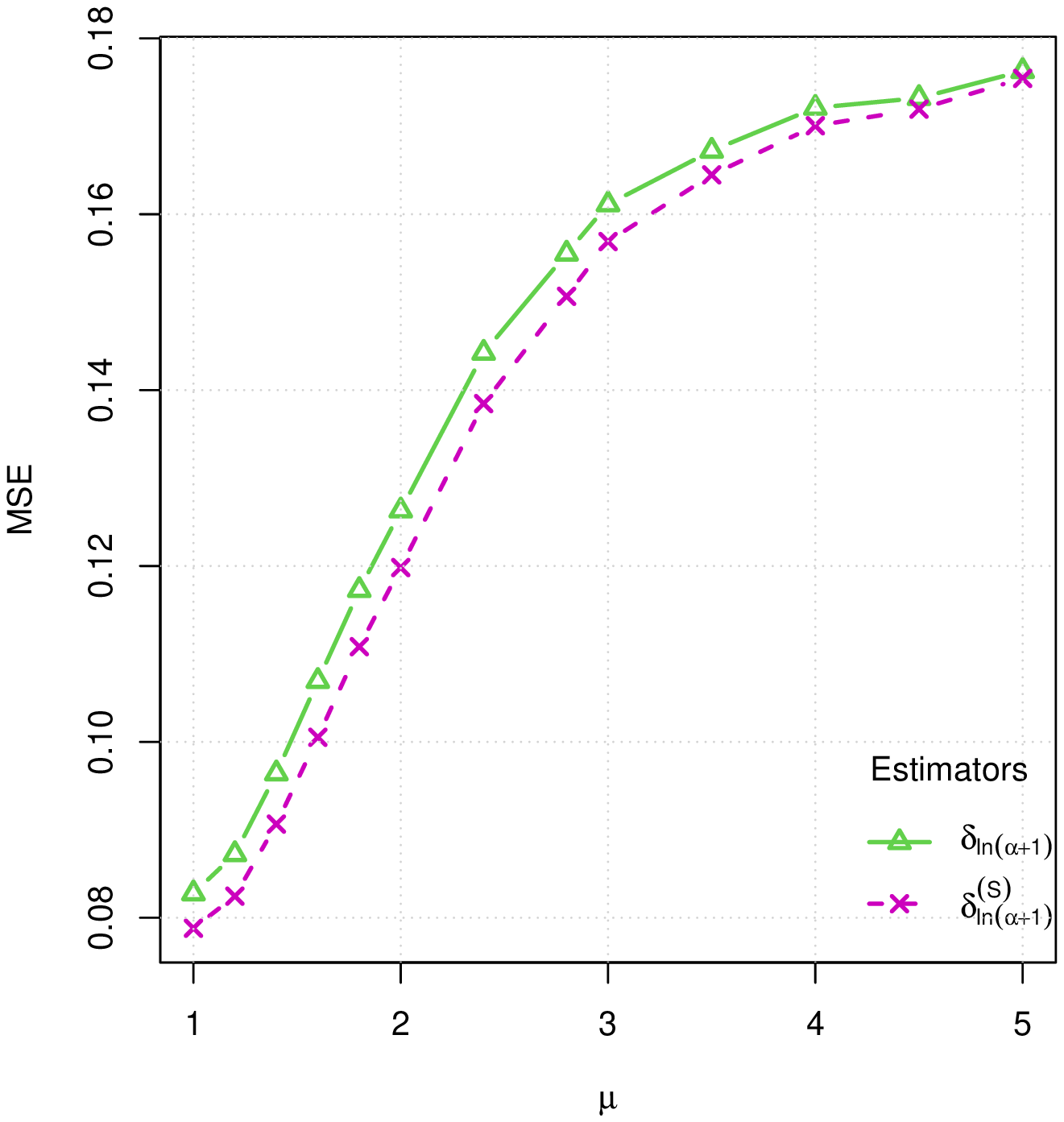}
			\caption{$n=20$, $\alpha=1.5$  }
		\end{subfigure}%

		\begin{subfigure}[b]{0.5\textwidth}
			\centering
			\includegraphics[height=2.3in,width=6.5cm]{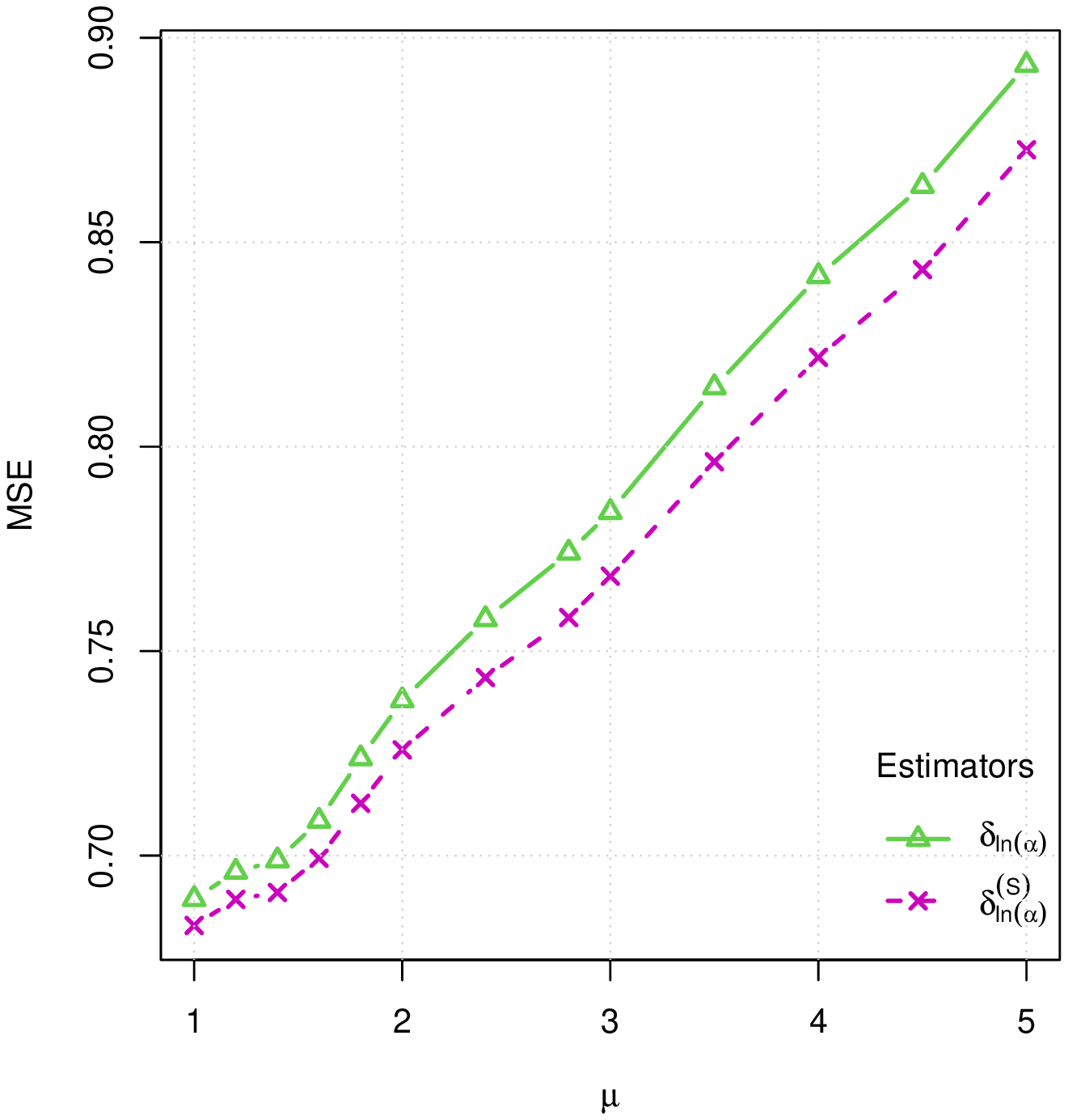}
			\caption{$n=5$, $\alpha=0.5$  }
		\end{subfigure}%
		\begin{subfigure}[b]{0.5\textwidth}
			\centering
			\includegraphics[height=2.3in,width=6.5cm]{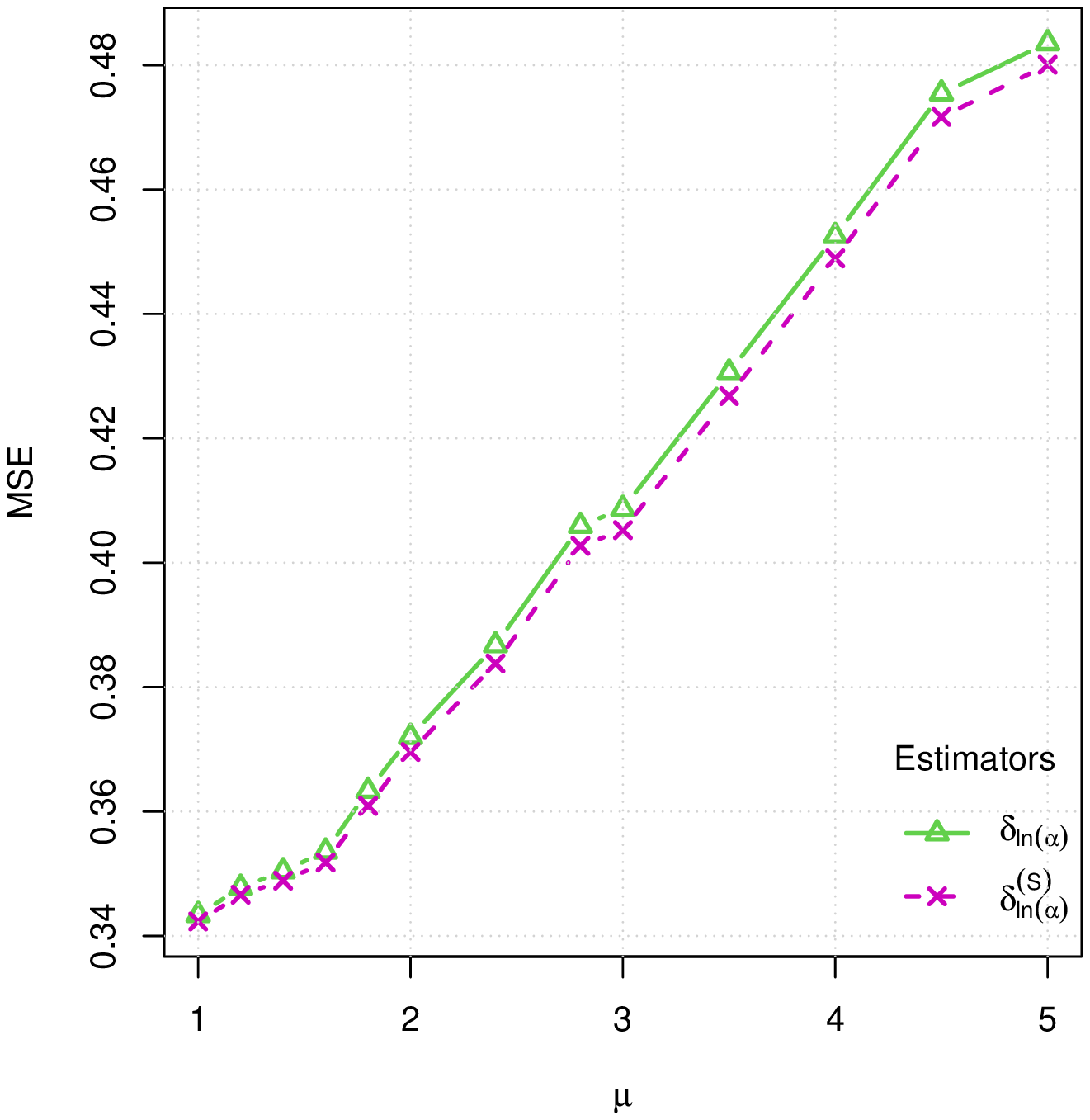}
			\caption{$n=10$, $\alpha=1$  }
		\end{subfigure}%
		
	\end{center}
	\caption{\textbf{Mean Squared error plots of the two naive estimators and their improved versions $\delta_{\ln (\alpha)},\delta_{\ln (\alpha+1)},\delta^{(S)}_{\ln (\alpha)}$ and $\delta^{(S)}_{\ln (\alpha+1)}$   of $H_S(\underline{\theta})$.}
	}%
	\label{fig:subfigures}
\end{figure}
\FloatBarrier
\FloatBarrier
\FloatBarrier
\begin{figure}[ht!]
	\begin{center}
		\begin{subfigure}[b]{0.5\textwidth}
			\centering
			\includegraphics[height=2.3in,width=6.5cm]{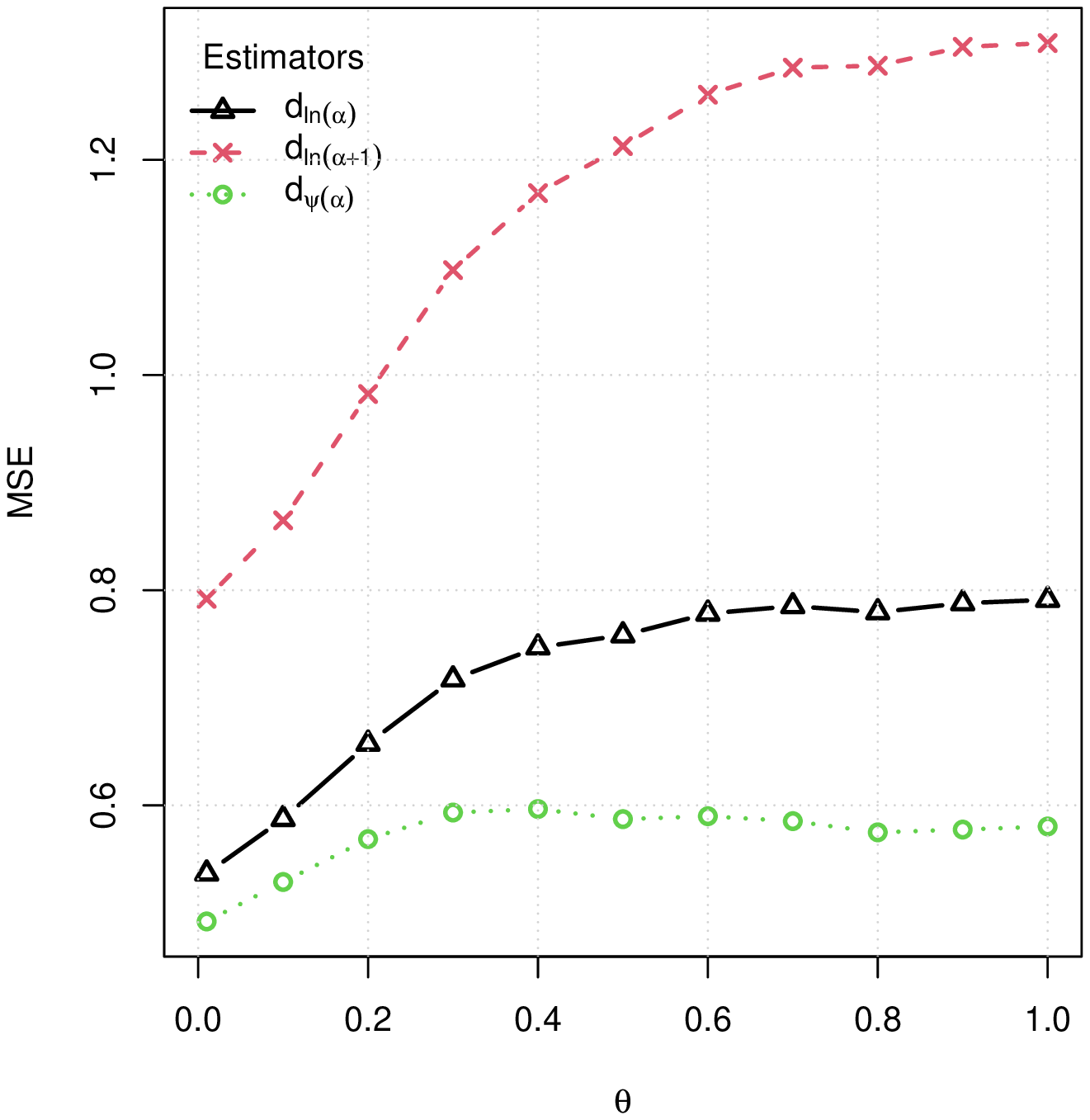}
			\caption{$n=5$, $\alpha=0.5$ }
		\end{subfigure}%
		\begin{subfigure}[b]{0.5\textwidth}
			\centering
			\includegraphics[height=2.3in,width=6.5cm]{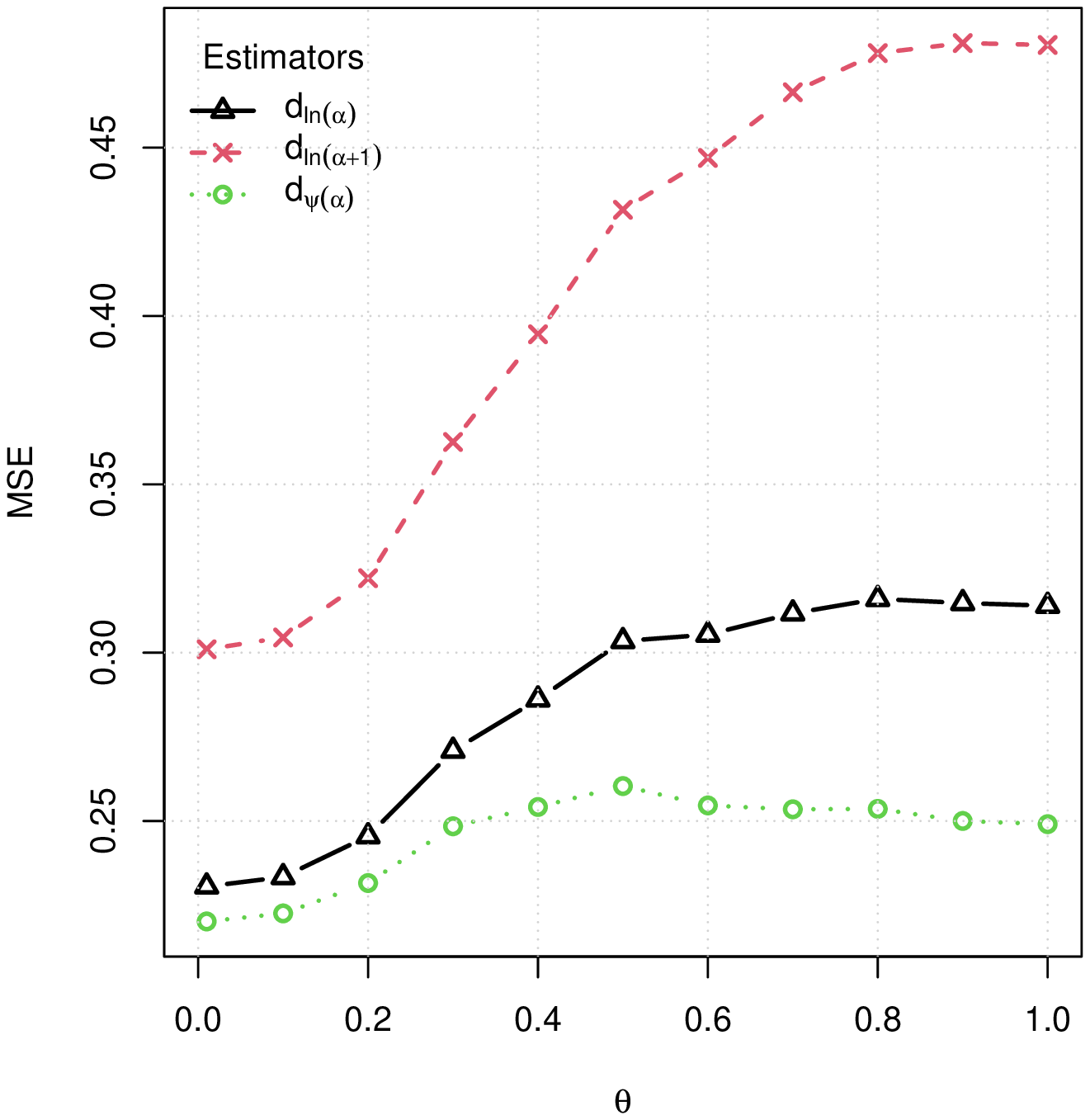}
			\caption{$n=10$, $\alpha=0.5$  }
		\end{subfigure}%
		
		\begin{subfigure}[b]{0.5\textwidth}
			\centering
			\includegraphics[height=2.3in,width=6.5cm]{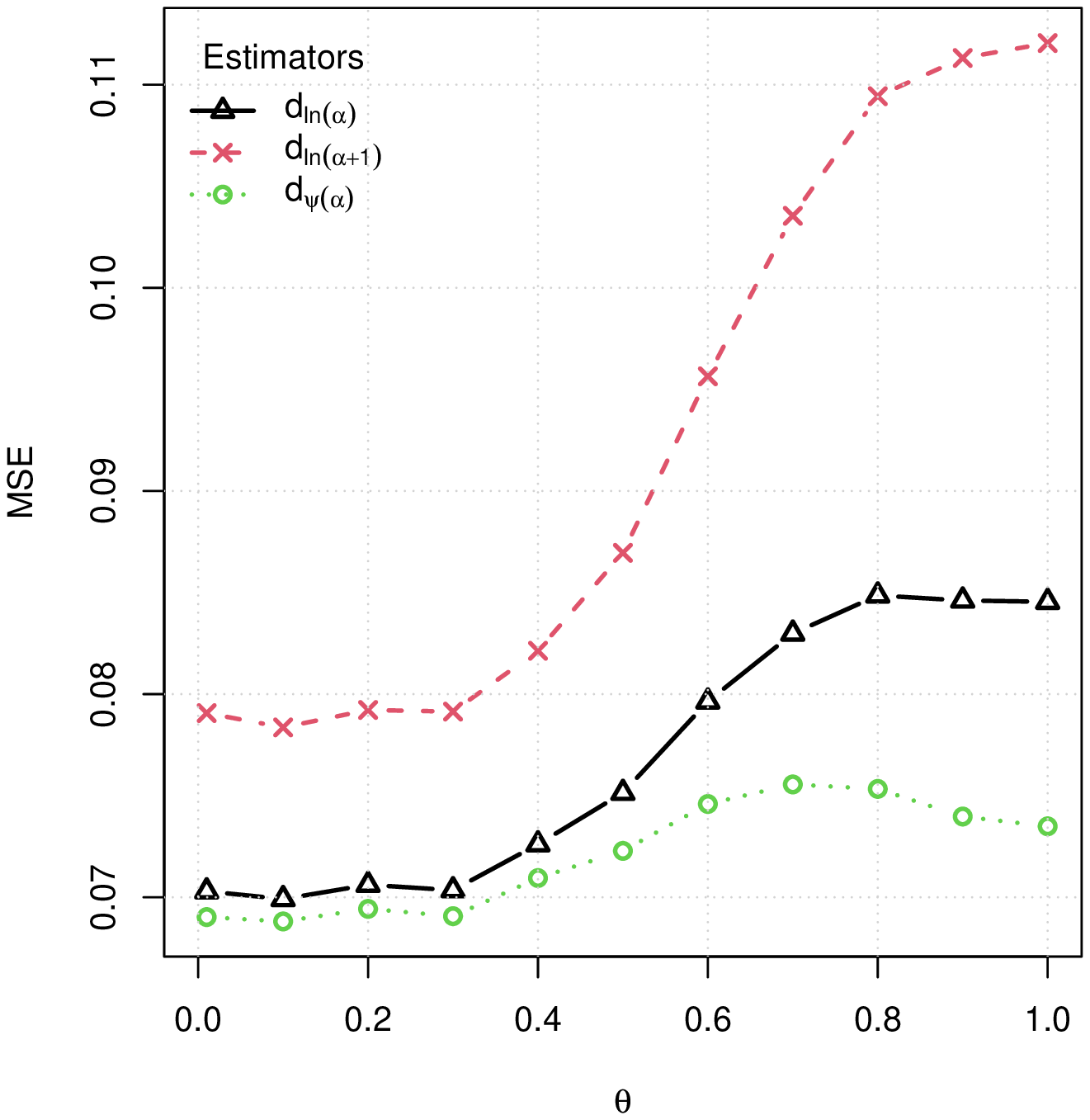}
			\caption{$n=15$, $\alpha=0.5$  }
		\end{subfigure}%
		\begin{subfigure}[b]{0.5\textwidth}
			\centering
			\includegraphics[height=2.3in,width=6.5cm]{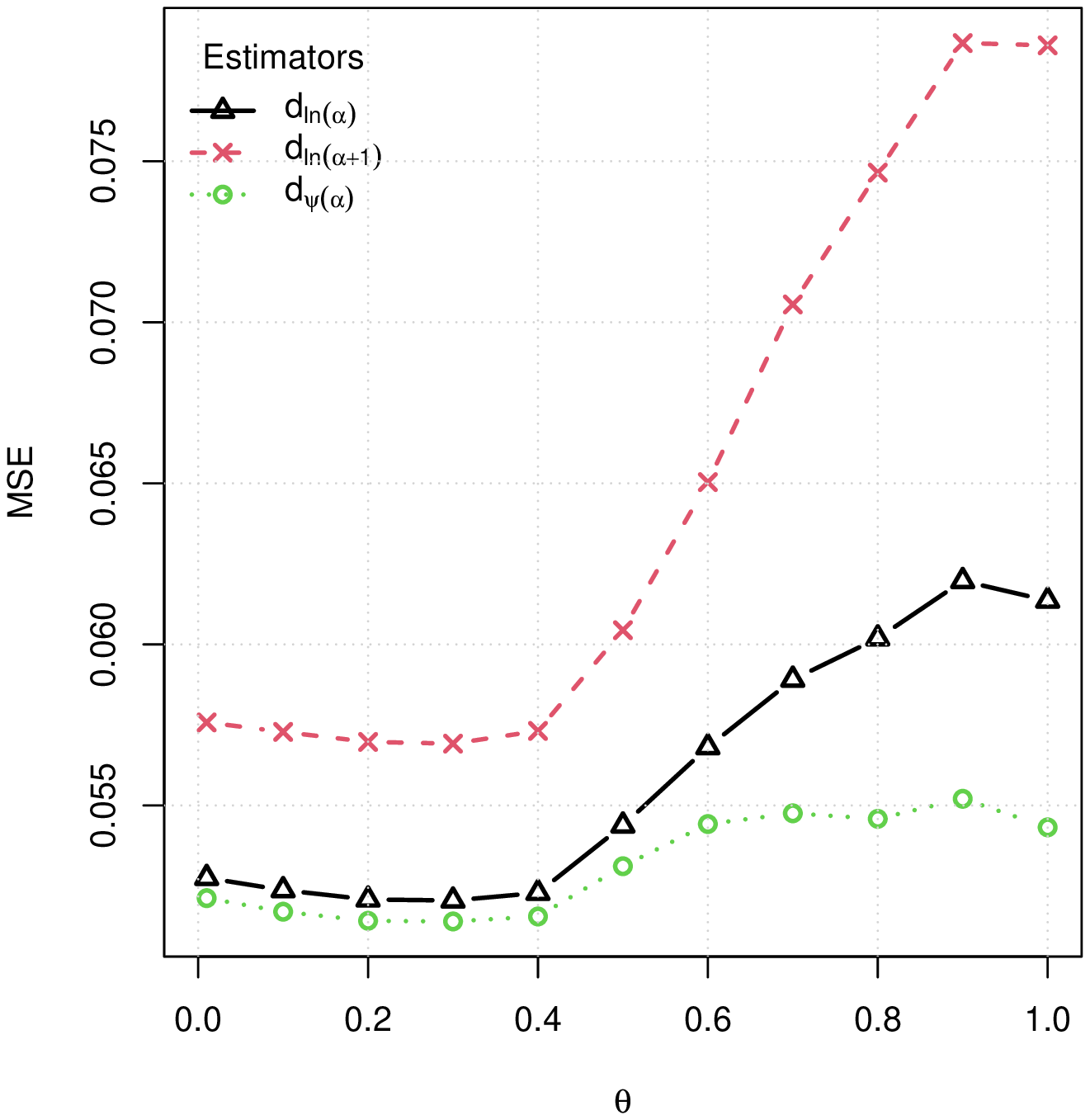}
			\caption{$n=20$, $\alpha=1$  }
		\end{subfigure}%

		\begin{subfigure}[b]{0.5\textwidth}
			\centering
			\includegraphics[height=2.3in,width=6.5cm]{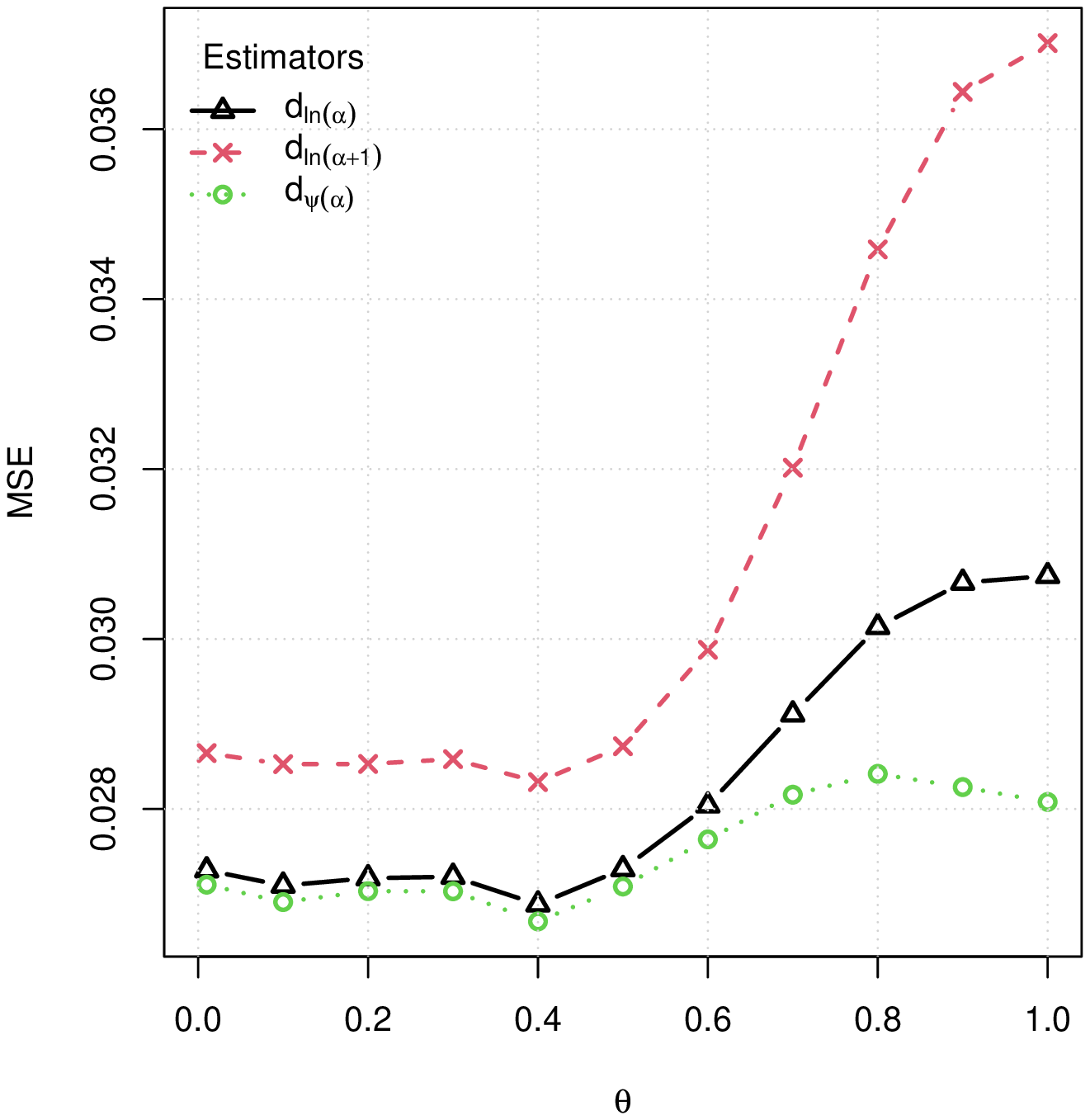}
			\caption{$n=25$, $\alpha=1$  }
		\end{subfigure}%
		\begin{subfigure}[b]{0.5\textwidth}
			\centering
			\includegraphics[height=2.3in,width=6.5cm]{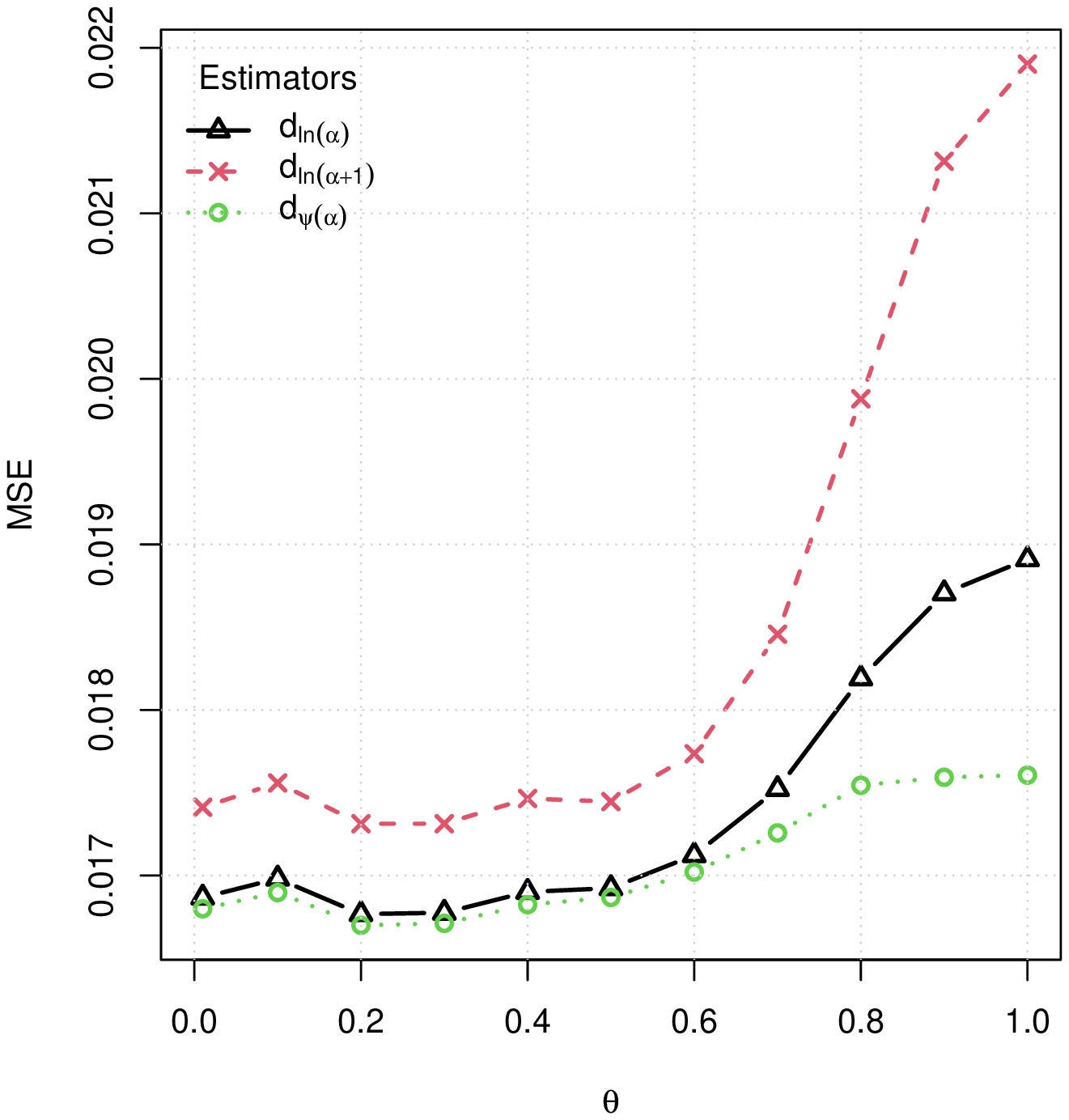}
			\caption{$n=30$, $\alpha=2$  }
		\end{subfigure}%
		
	\end{center}
	\caption{\textbf{Mean squared error plots of the three naive estimators $d_{\ln \alpha}$, $d_{\ln (\alpha+1)}$ and $d_{\psi(\alpha)}$ of $H_M(\underline{\theta})$.}
	}%
	\label{fig:subfigures}
\end{figure}
\FloatBarrier
\begin{figure}[ht!]
	\begin{center}
		\begin{subfigure}[b]{0.5\textwidth}
			\centering
			\includegraphics[height=2.3in,width=6.5cm]{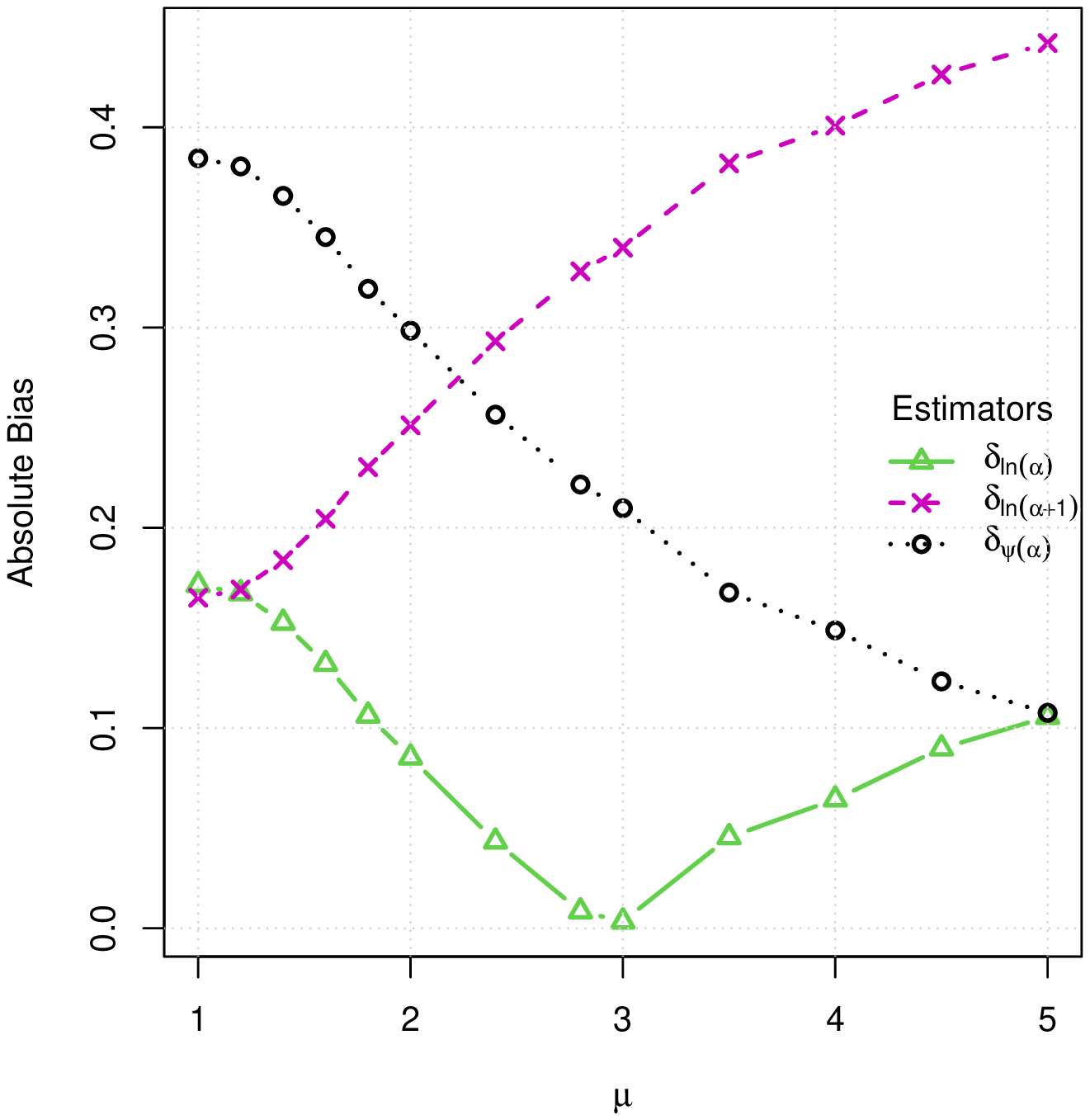}
			\caption{$n=5$, $\alpha=0.5$ }
		\end{subfigure}%
		\begin{subfigure}[b]{0.5\textwidth}
			\centering
			\includegraphics[height=2.3in,width=6.5cm]{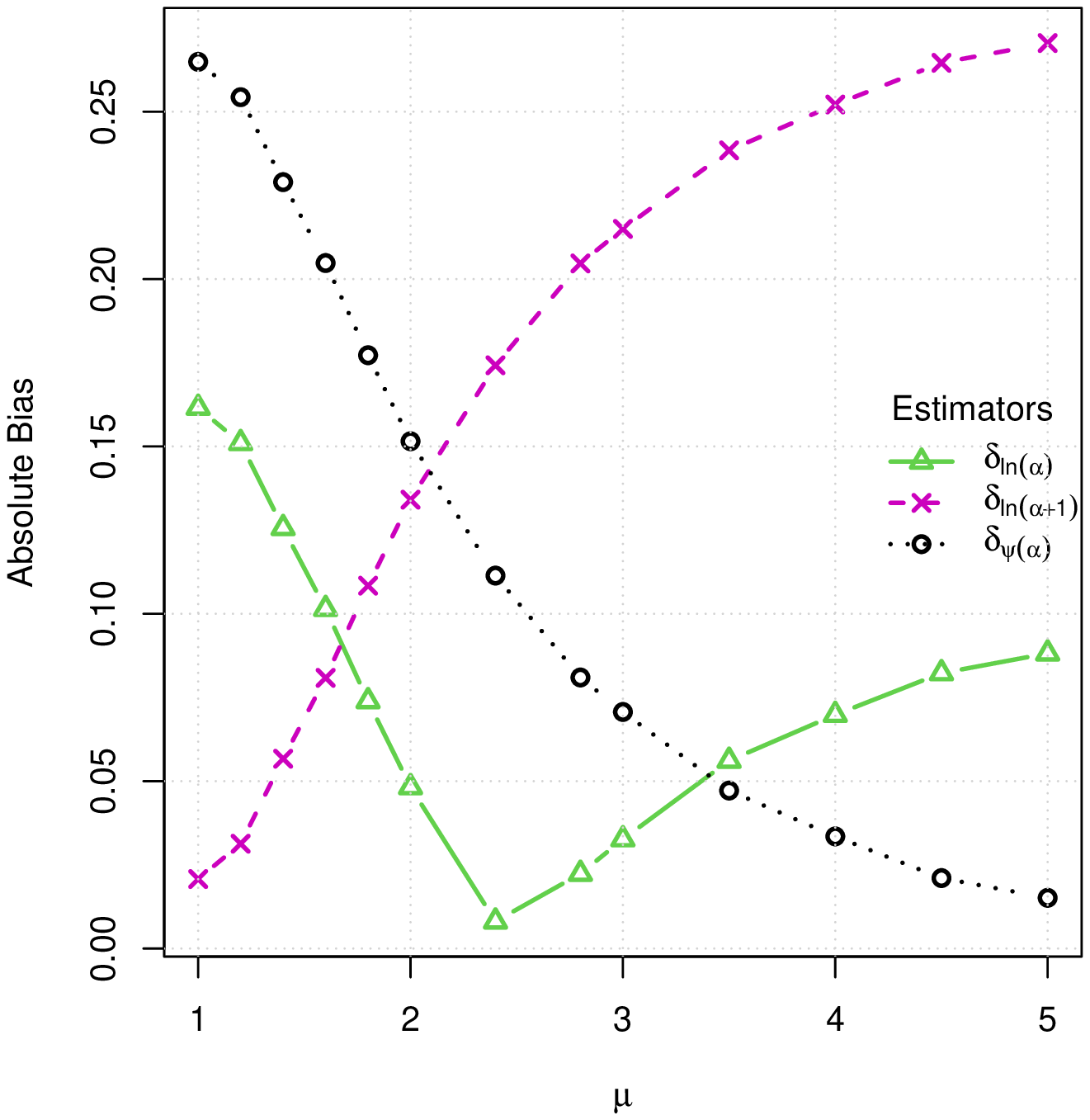}
			\caption{$n=10$, $\alpha=0.5$  }
		\end{subfigure}%
		
		\begin{subfigure}[b]{0.5\textwidth}
			\centering
			\includegraphics[height=2.3in,width=6.5cm]{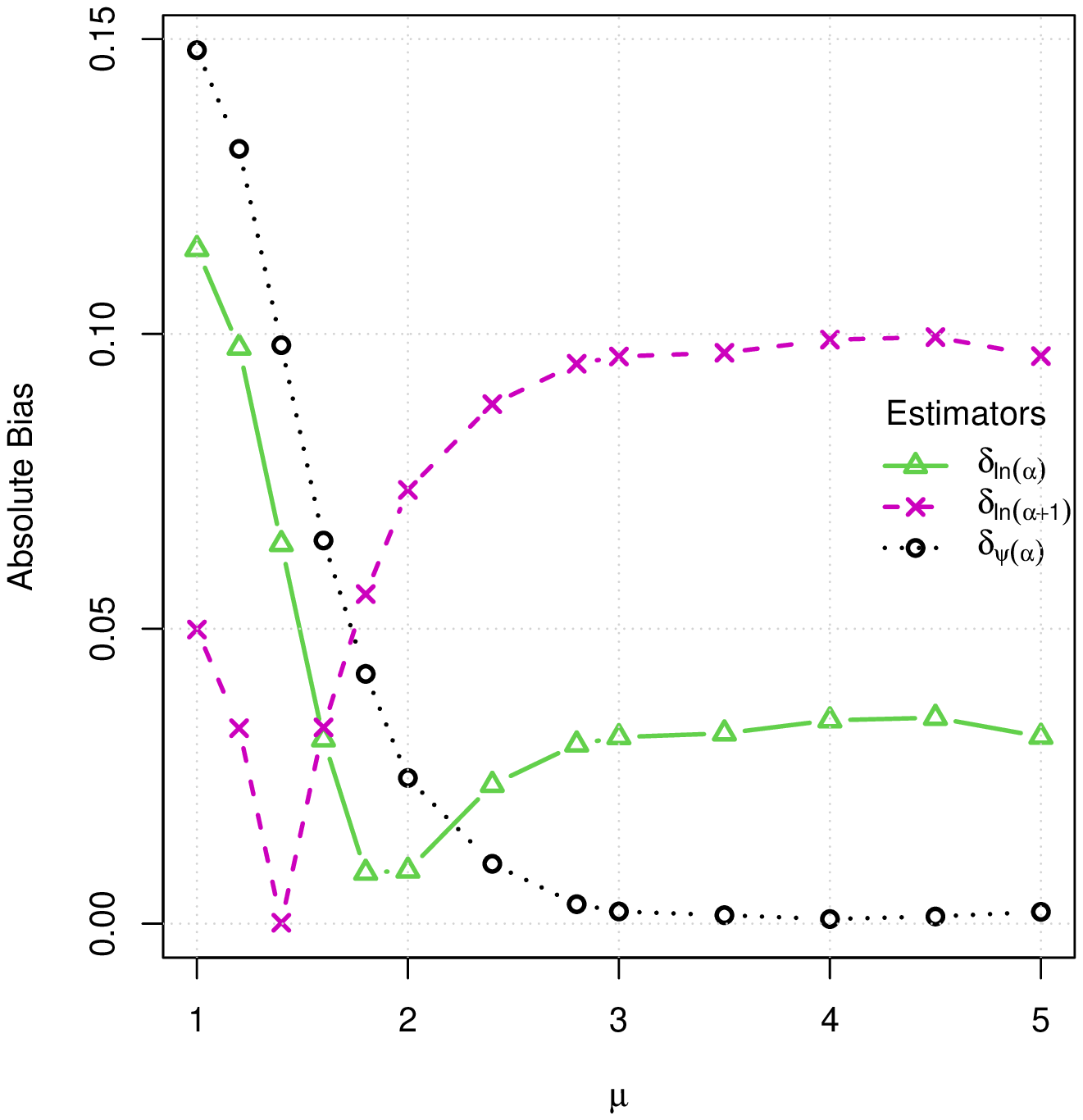}
			\caption{$n=15$, $\alpha=0.5$  }
		\end{subfigure}%
		\begin{subfigure}[b]{0.5\textwidth}
			\centering
			\includegraphics[height=2.3in,width=6.5cm]{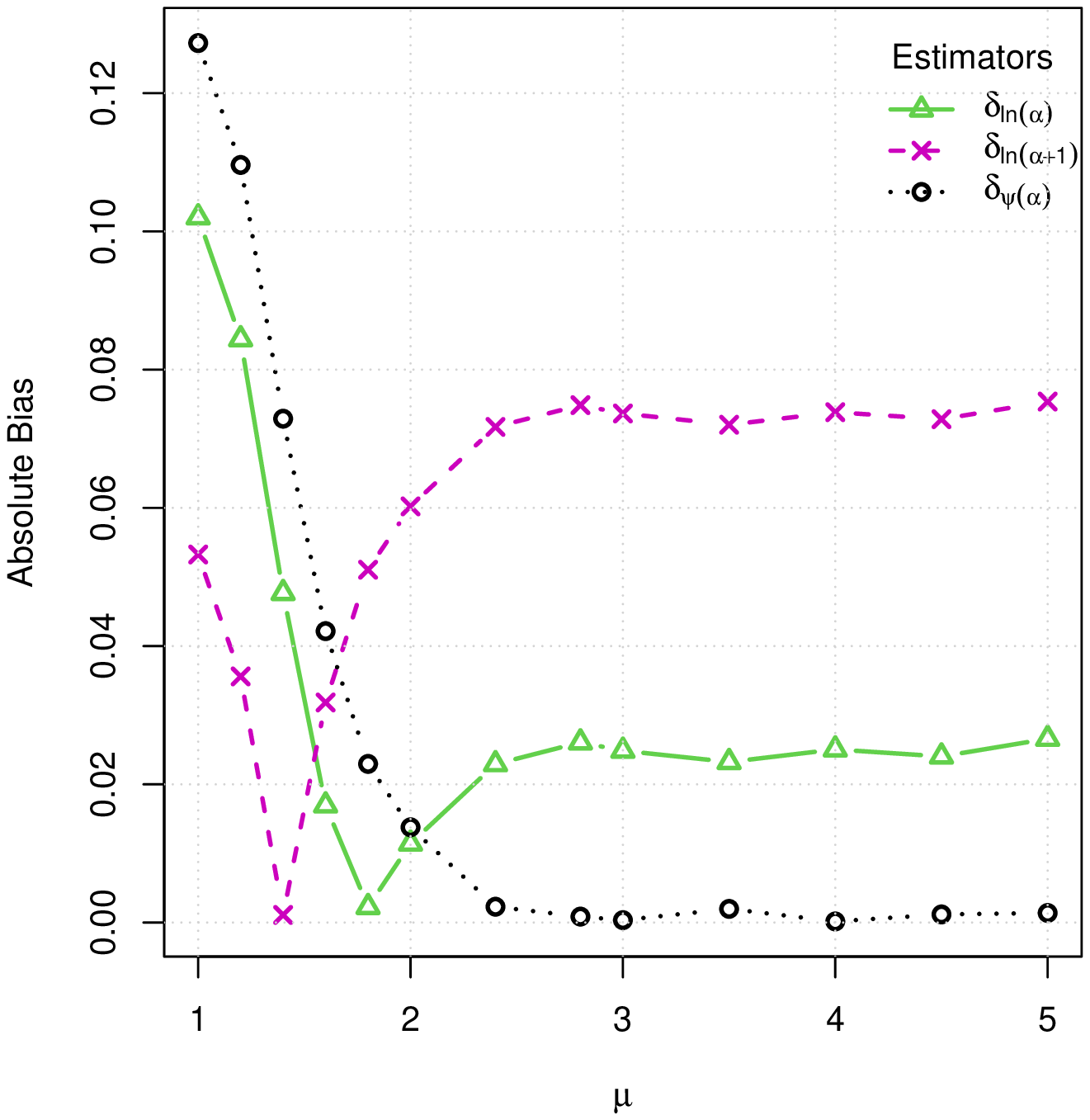}
			\caption{$n=20$, $\alpha=1$  }
		\end{subfigure}%

		\begin{subfigure}[b]{0.5\textwidth}
			\centering
			\includegraphics[height=2.3in,width=6.5cm]{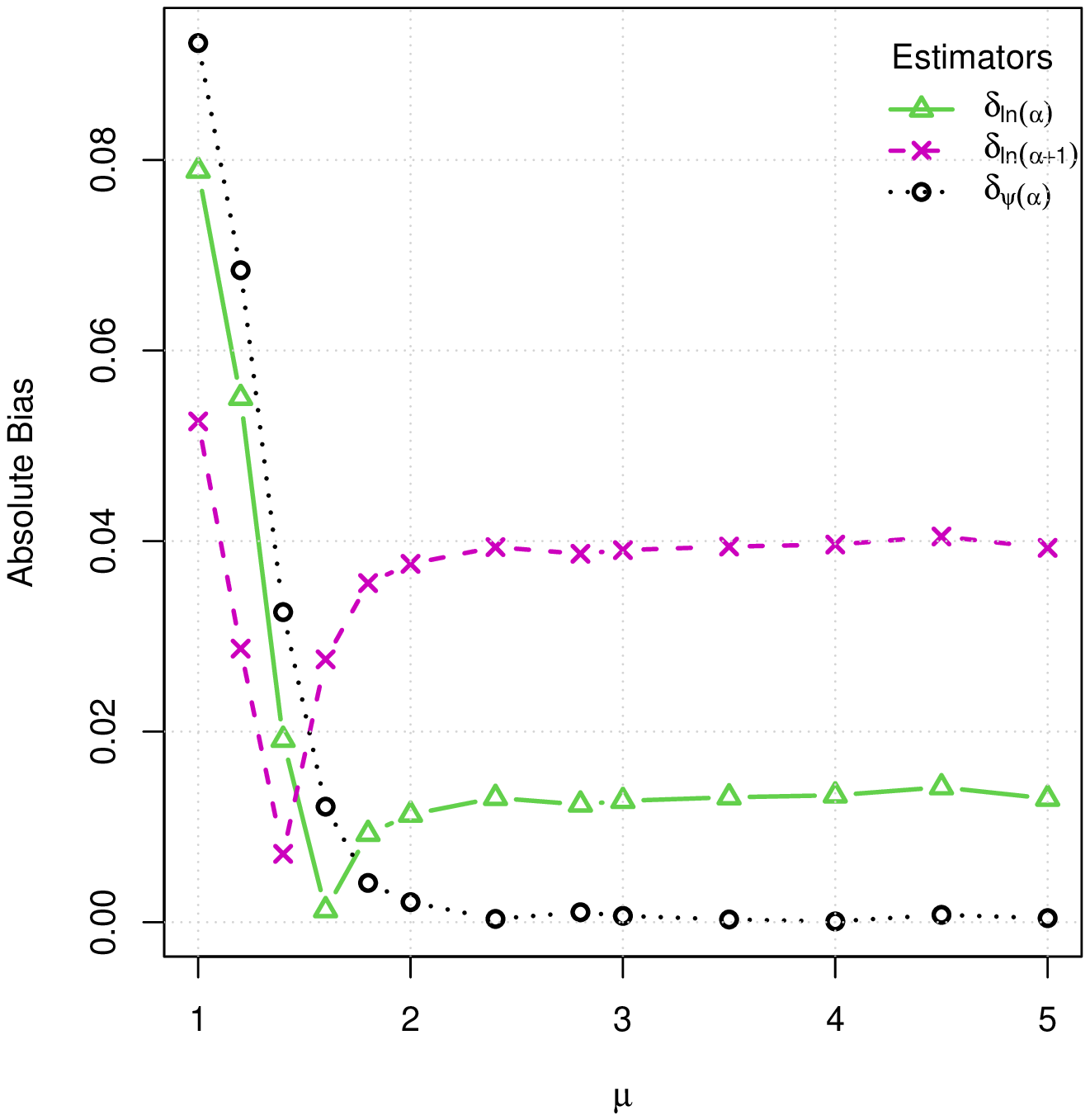}
			\caption{$n=25$, $\alpha=1$  }
		\end{subfigure}%
		\begin{subfigure}[b]{0.5\textwidth}
			\centering
			\includegraphics[height=2.3in,width=6.5cm]{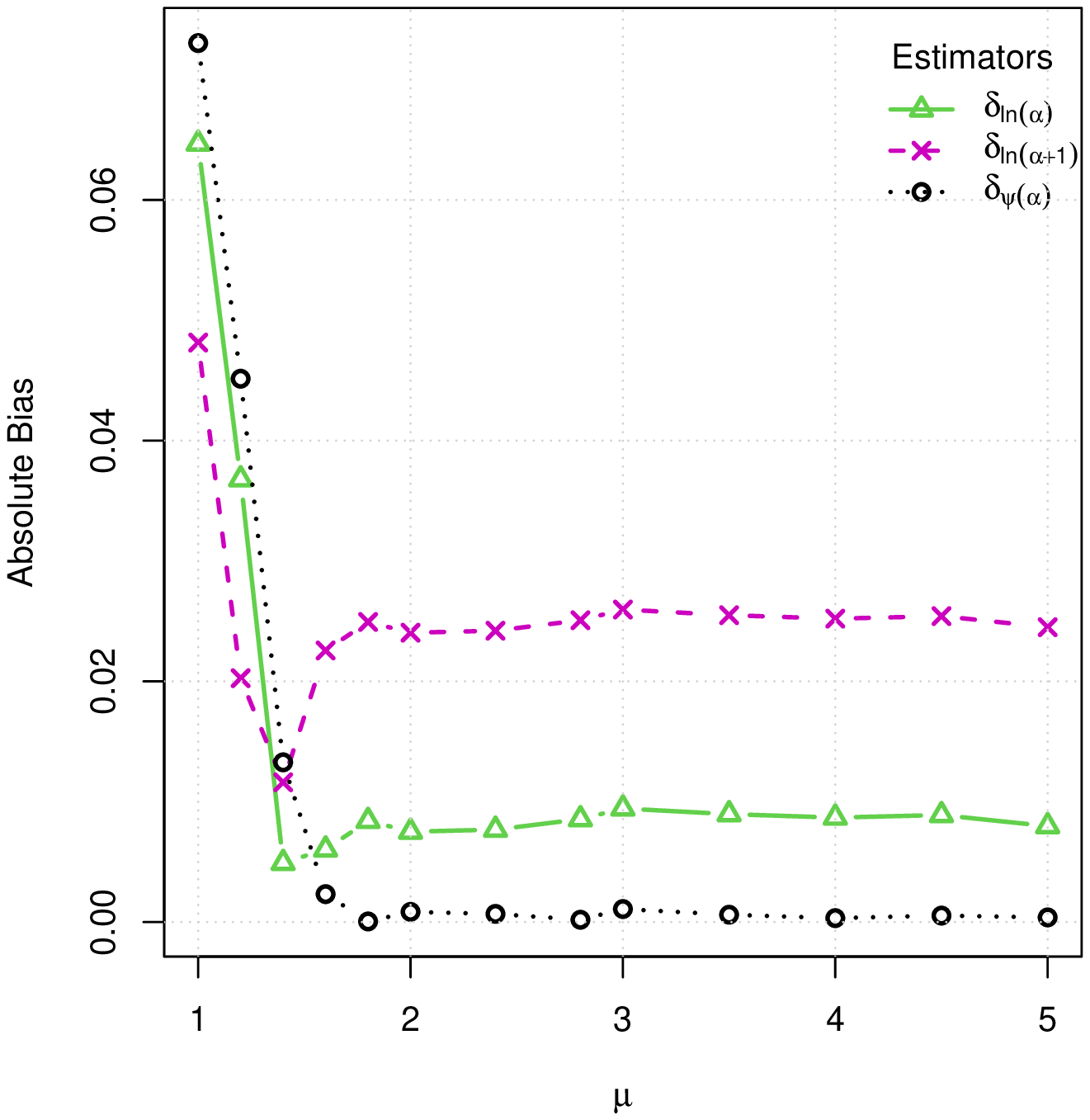}
			\caption{$n=30$, $\alpha=2$  }
		\end{subfigure}%
		
	\end{center}
	\caption{\textbf{Absolute Bias plots of the three naive estimators $\delta_{\ln (\alpha)},\delta_{\ln (\alpha+1)}$ and $\delta_{\psi(\alpha)}$) of $H_S(\underline{\theta})$.}
	}%
	\label{fig:subfigures}
\end{figure}
\FloatBarrier

\FloatBarrier
\begin{figure}[ht!]
	\begin{center}
		\begin{subfigure}[b]{0.5\textwidth}
			\centering
			\includegraphics[height=2.3in,width=6.5cm]{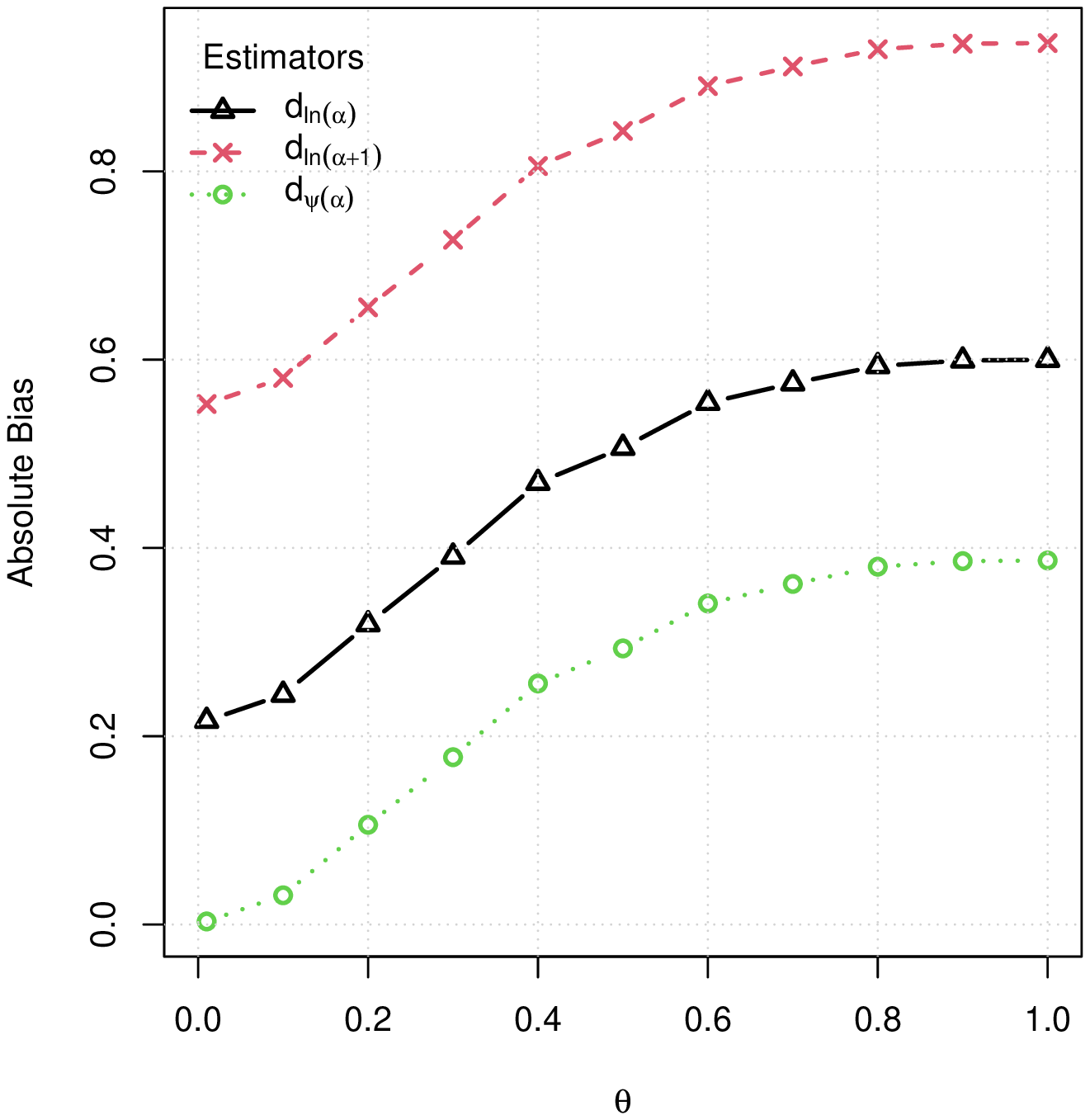}
			\caption{$n=5$, $\alpha=0.5$ }
		\end{subfigure}%
		\begin{subfigure}[b]{0.5\textwidth}
			\centering
			\includegraphics[height=2.3in,width=6.5cm]{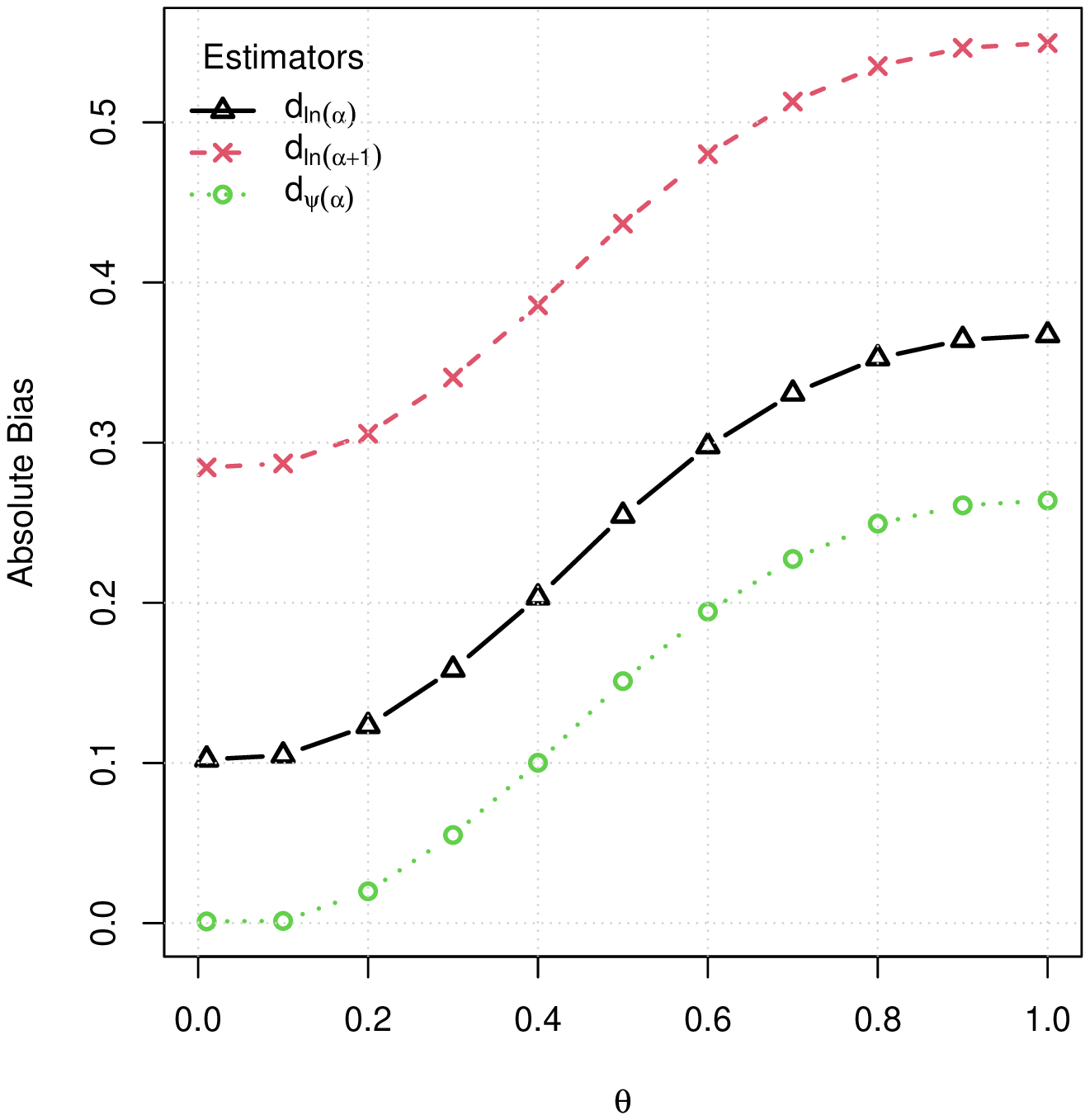}
			\caption{$n=10$, $\alpha=0.5$  }
		\end{subfigure}%
		
		\begin{subfigure}[b]{0.5\textwidth}
			\centering
			\includegraphics[height=2.3in,width=6.5cm]{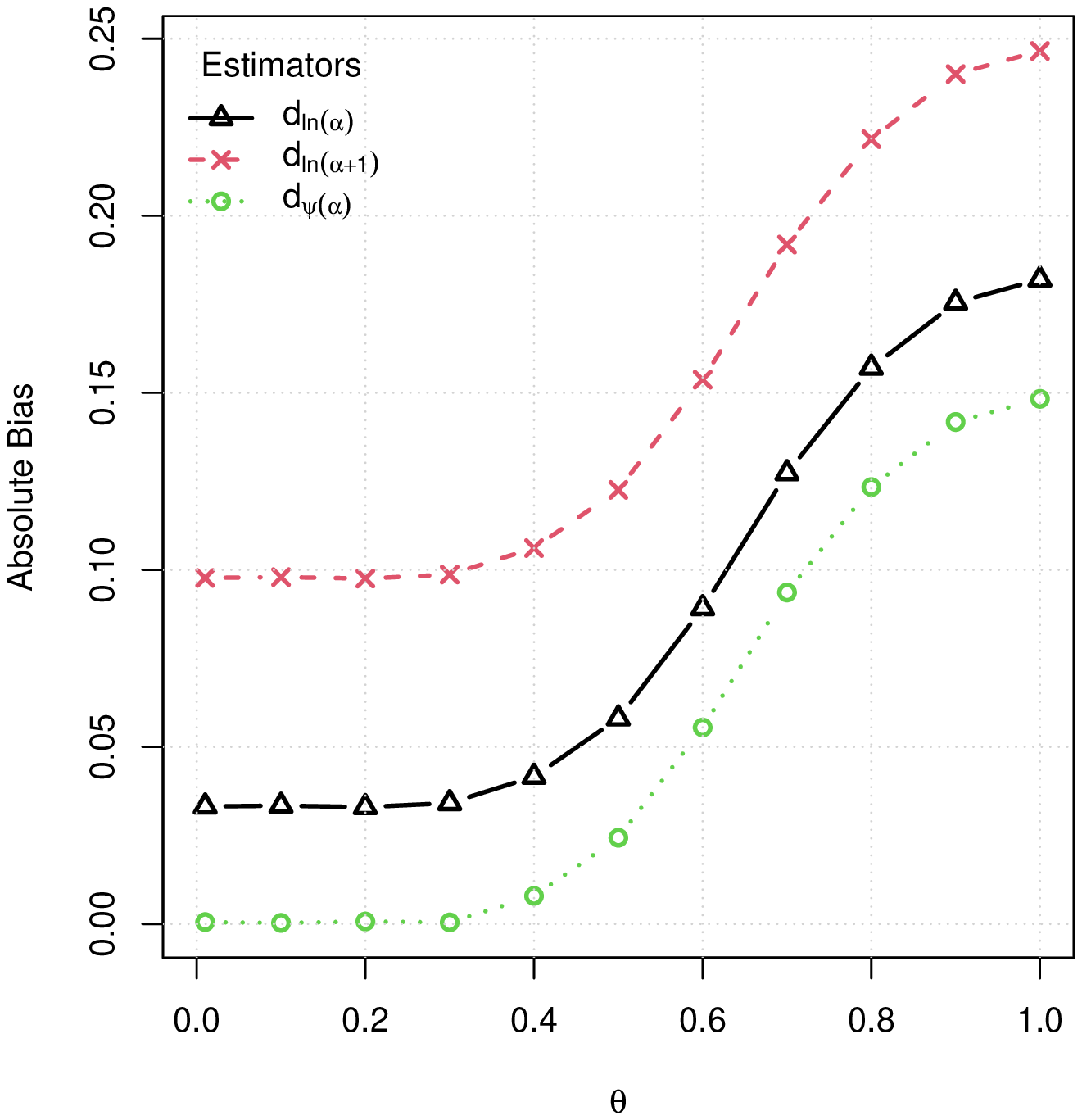}
			\caption{$n=15$, $\alpha=1$  }
		\end{subfigure}%
		\begin{subfigure}[b]{0.5\textwidth}
			\centering
			\includegraphics[height=2.3in,width=6.5cm]{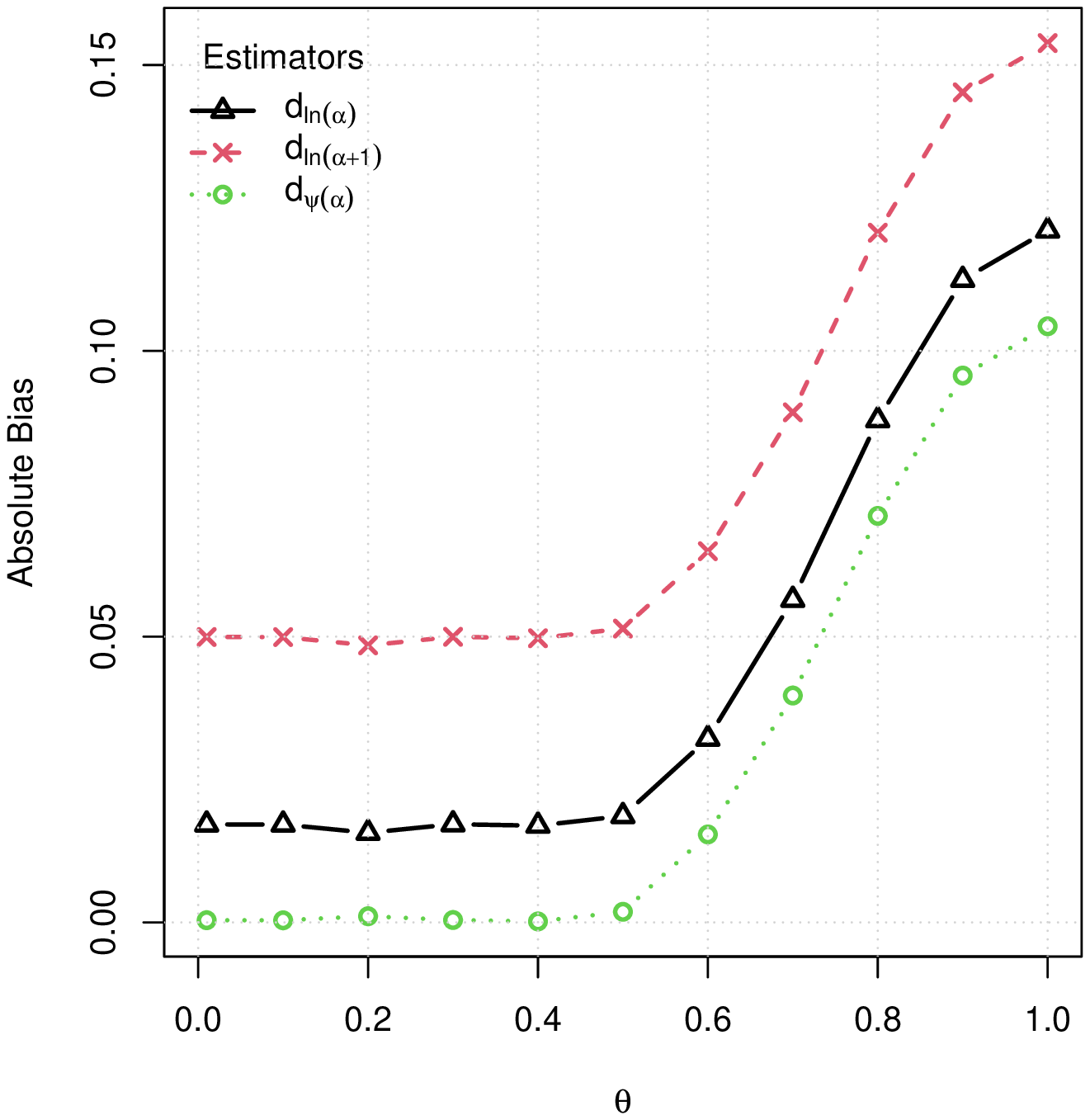}
			\caption{$n=20$, $\alpha=1$  }
		\end{subfigure}%

		\begin{subfigure}[b]{0.5\textwidth}
			\centering
			\includegraphics[height=2.3in]{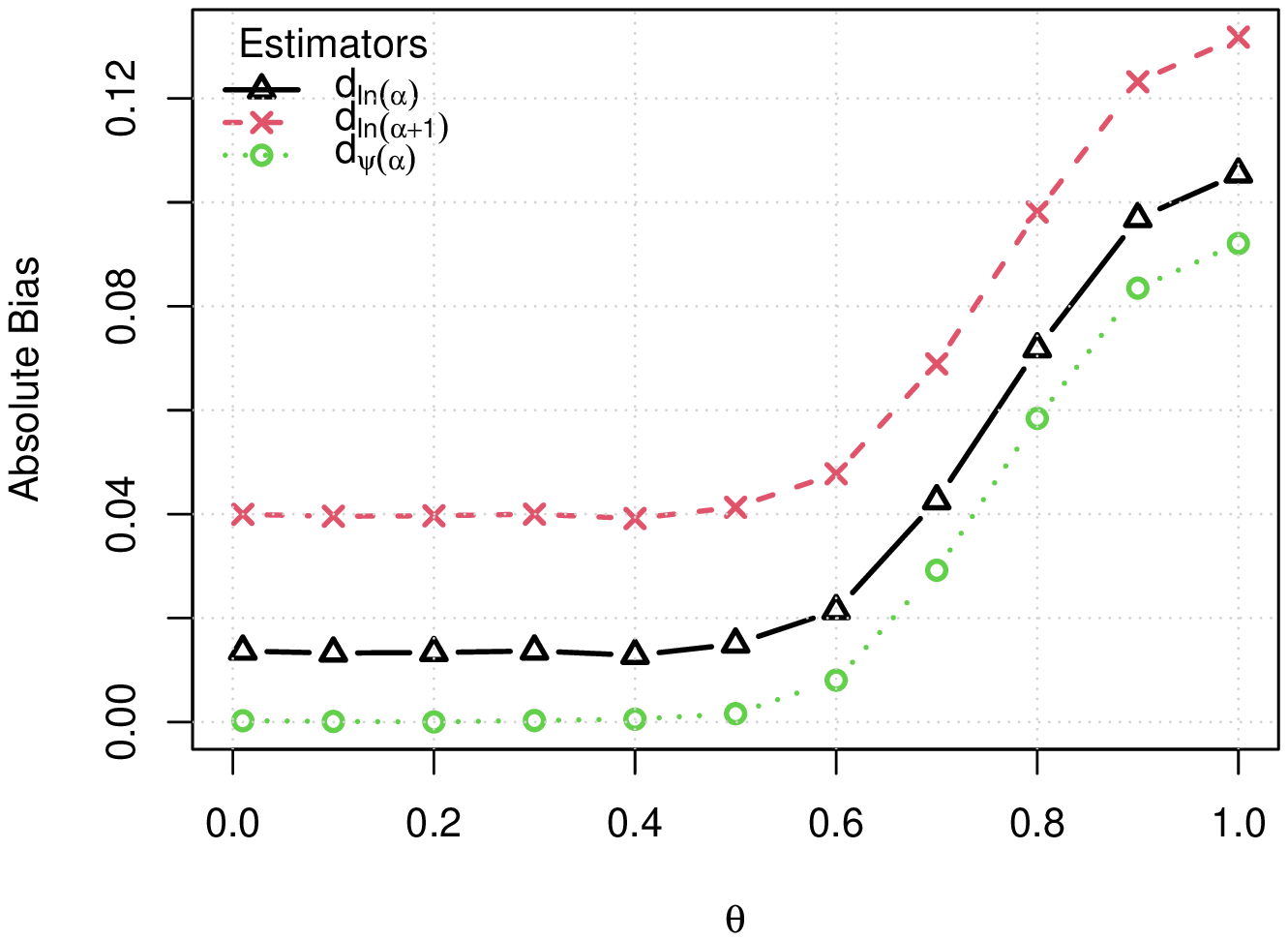}
			\caption{$n=25$, $\alpha=1.5$  }
		\end{subfigure}%
		\begin{subfigure}[b]{0.5\textwidth}
			\centering
			\includegraphics[height=2.3in,width=6.5cm]{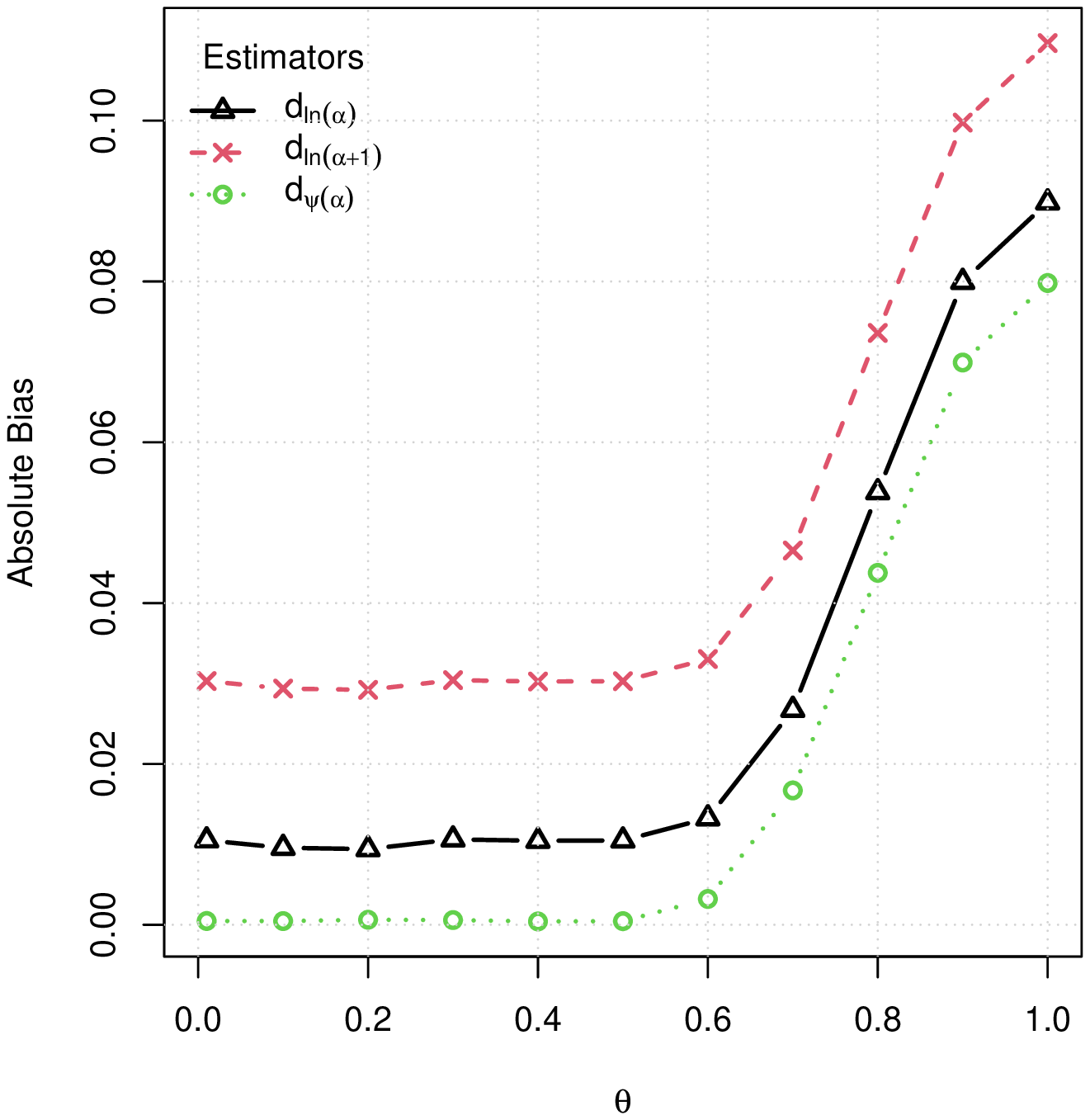}
			\caption{$n=30$, $\alpha=2$  }
		\end{subfigure}%
		
	\end{center}
	\caption{\textbf{Absolute Bias plots of the three naive estimators $d_{\ln \alpha}$, $d_{\ln (\alpha+1)}$ and $d_{\psi(\alpha)}$ of $H_M(\underline{\theta})$.}
	}%
	\label{fig:subfigures}
\end{figure}
\FloatBarrier

\FloatBarrier
\begin{figure}[ht!]
	\begin{center}
		\begin{subfigure}[b]{0.5\textwidth}
			\centering
			\includegraphics[height=2.3in,width=6.5cm]{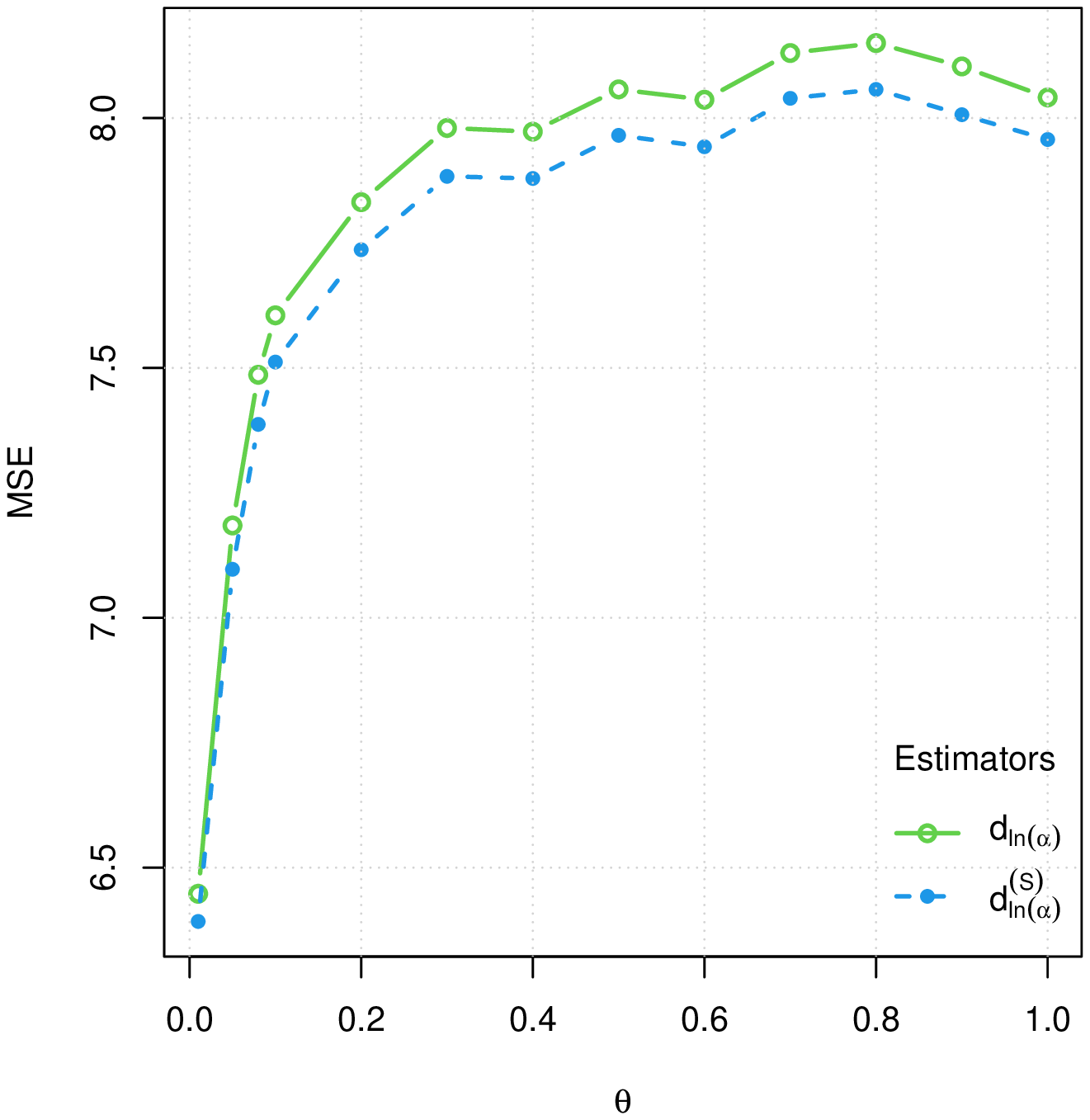}
			\caption{$n=3$, $\alpha=0.2$ }
		\end{subfigure}%
		\begin{subfigure}[b]{0.5\textwidth}
			\centering
			\includegraphics[height=2.3in,width=6.5cm]{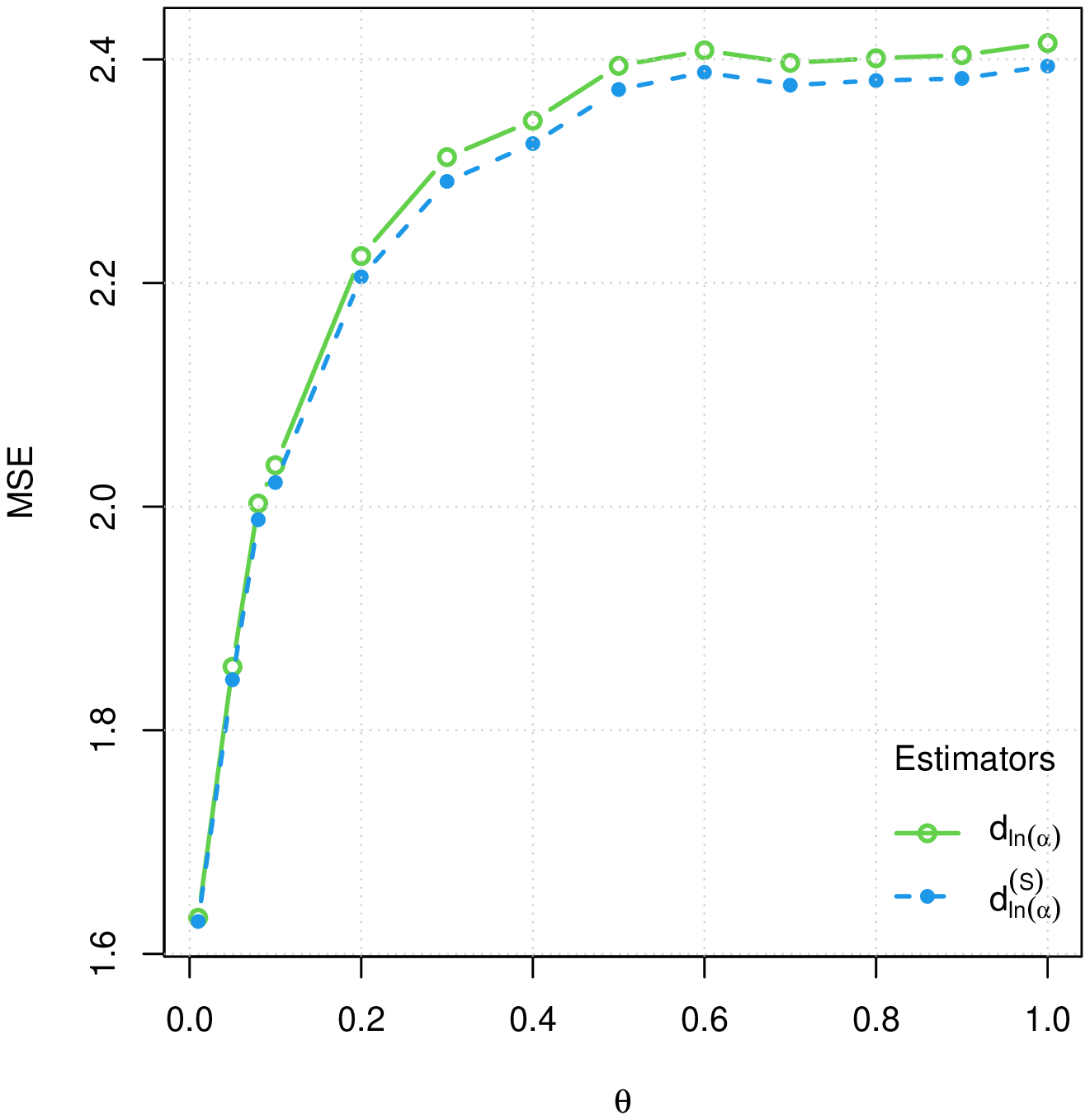}
			\caption{$n=6$, $\alpha=0.5$  }
		\end{subfigure}%
		
		\begin{subfigure}[b]{0.5\textwidth}
			\centering
			\includegraphics[height=2.3in,width=6.5cm]{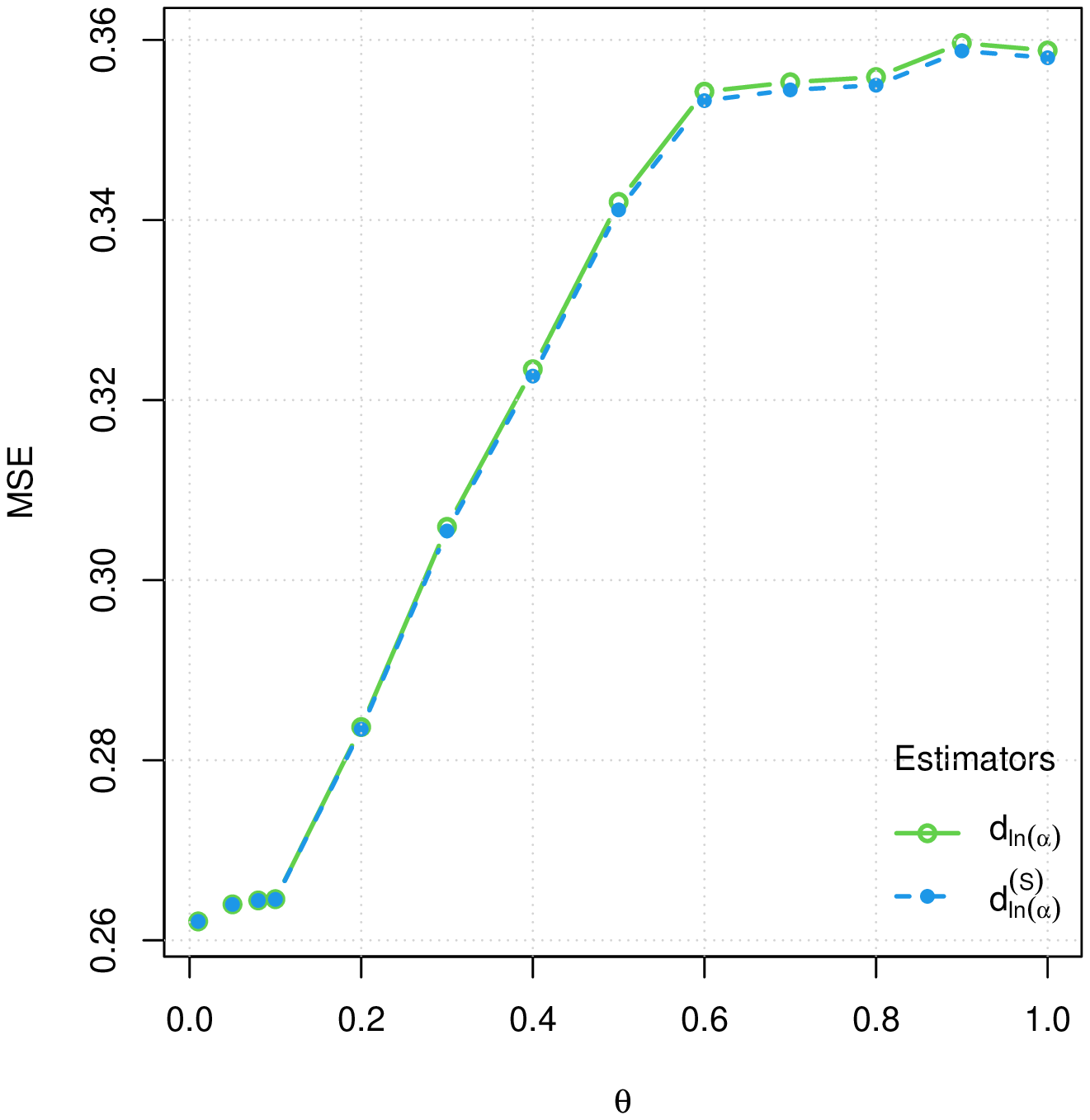}
			\caption{$n=12$, $\alpha=0.5$  }
		\end{subfigure}%
		\begin{subfigure}[b]{0.5\textwidth}
			\centering
			\includegraphics[height=2.3in,width=6.5cm]{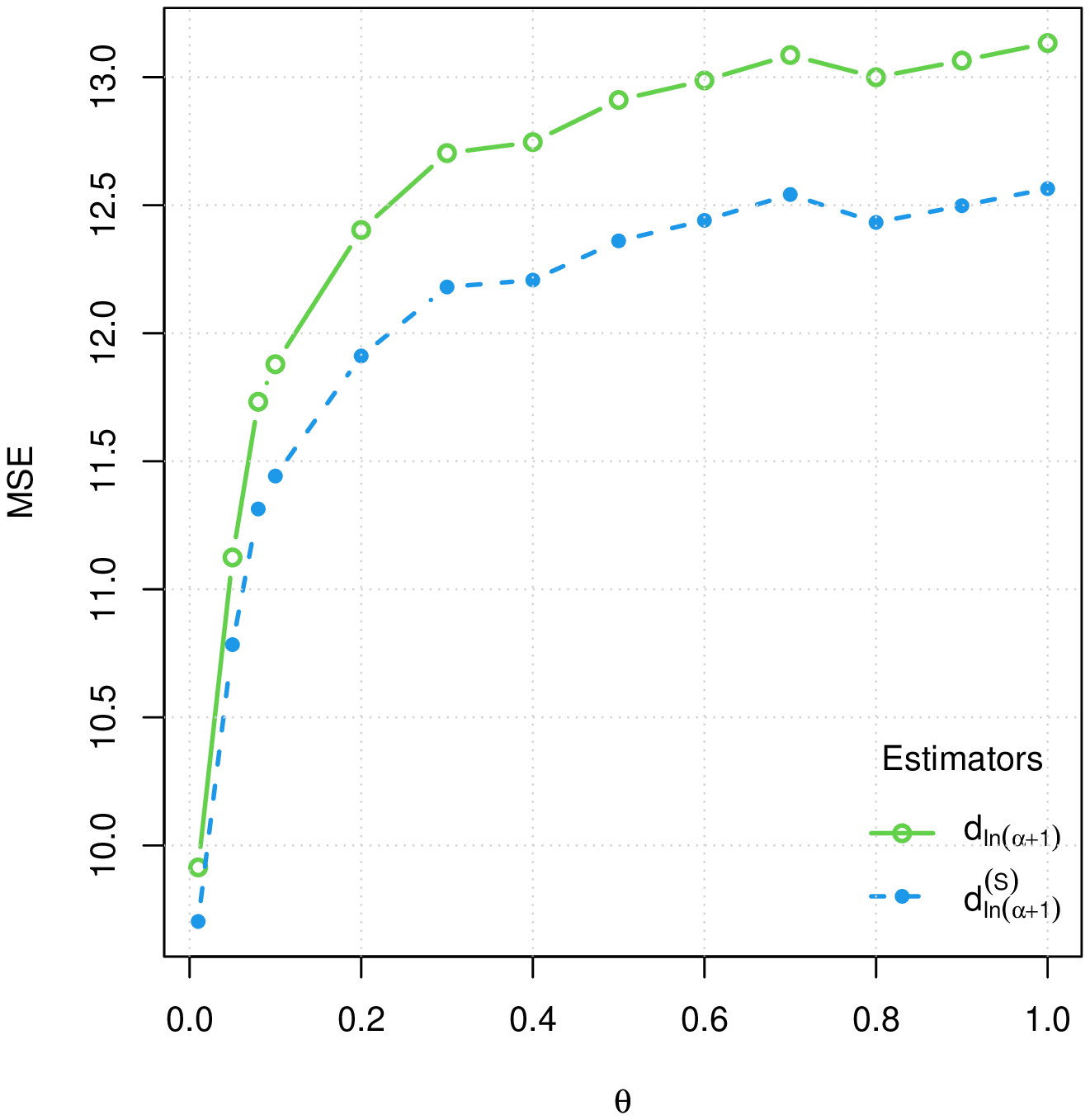}
			\caption{$n=3$, $\alpha=0.2$  }
		\end{subfigure}%

		\begin{subfigure}[b]{0.5\textwidth}
			\centering
			\includegraphics[height=2.3in,width=6.5cm]{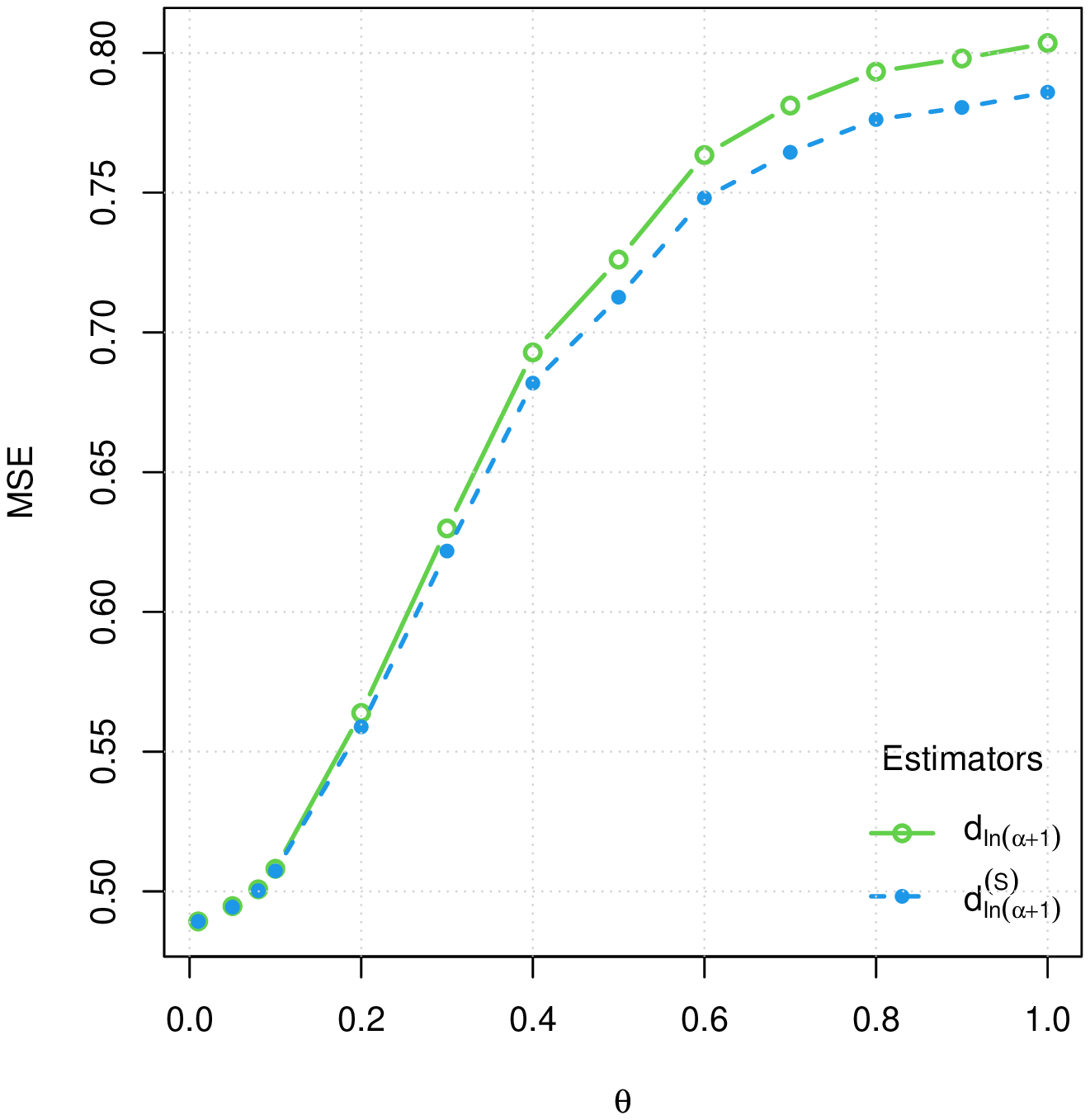}
			\caption{$n=6$, $\alpha=0.5$  }
		\end{subfigure}%
		\begin{subfigure}[b]{0.5\textwidth}
			\centering
			\includegraphics[height=2.3in,width=6.5cm]{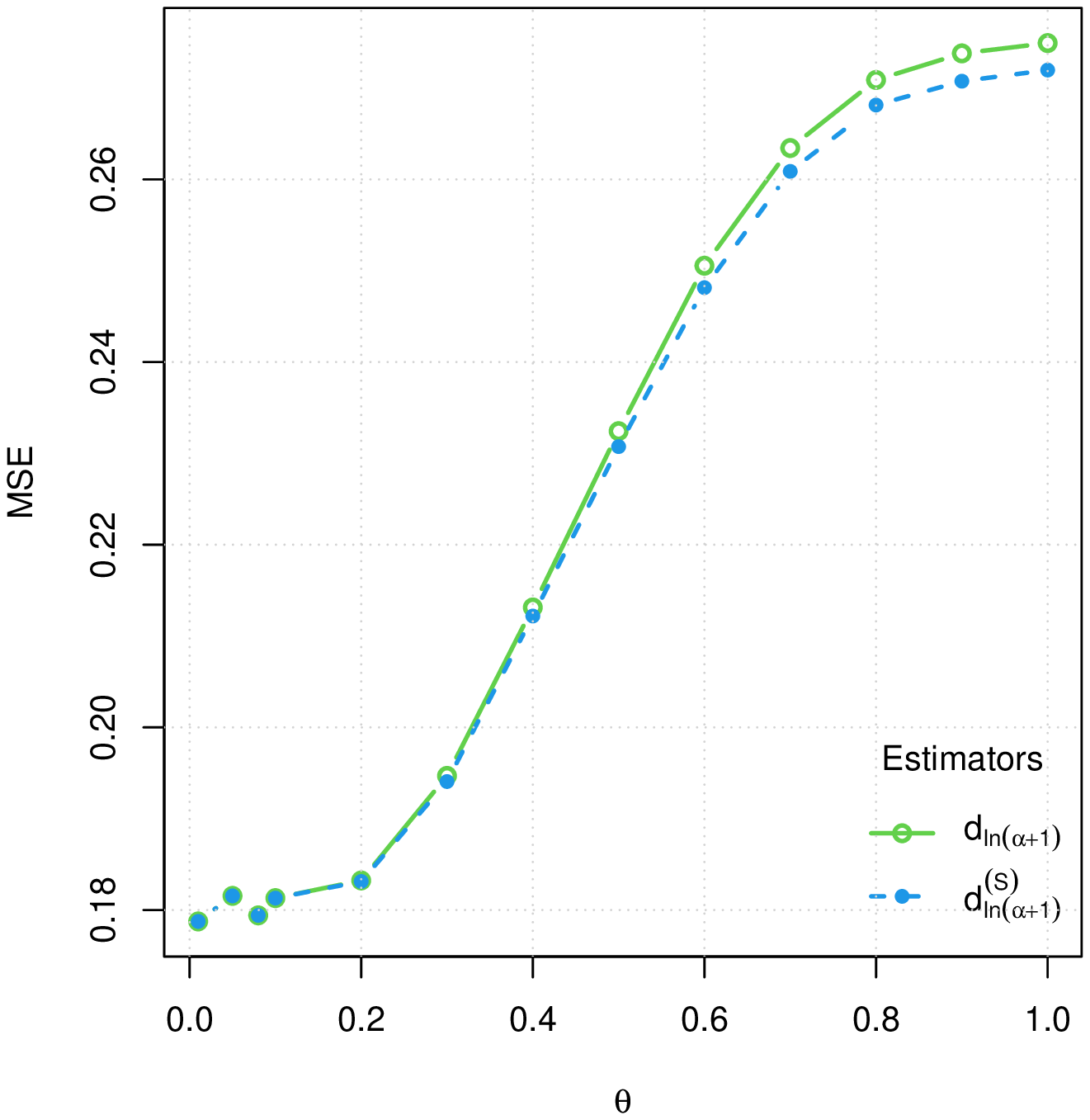}
			\caption{$n=15$, $\alpha=0.5$  }
		\end{subfigure}%
		
	\end{center}
	\caption{\textbf{Mean squared error plots of the two naive estimators and their improved versions ( $d_{\ln (\alpha)}, d_{\ln (\alpha+1)}, d^{(S)}_{\ln (\alpha)}$ and $d^{(S)}_{\ln (\alpha+1)}$) of $H_M(\underline{\theta})$.}
	}%
	\label{fig:subfigures}
\end{figure}
\FloatBarrier

\newpage
\section{Real data example}
Here we consider a real data set provided in \cite{proschan1963theoretical}. The data set reflects the failure times (in hours) of the air conditioning systems of Boeing 720 jet planes ``7913" and ``7914". The data set contains sample observations of equal size $n=24$ and is given as below: \par
\noindent \textbf{Plane 7913}: 97,51,11,4,141,18,142,68,77,80,1,16,106,206,82,54,31,216,46,111,39,63,18,191.\\
\textbf{Plane 7914}: 50,44,102,72,22,39,3,15,197,188,79,88,46,5,5,36,22,139,210,97,30,23,13,14.\\
\par
For testing the goodness of fit of gamma distribution, we analysed the given data in $R$ software using the Kolmogorov-Smirnov goodness of fit test. We observe that at $5\%$ level of significance, one can not reject the hypothesis that the data for Planes 7913 and 7914 are from $Gamma(1, 68)$ and $Gamma(1, 61)$ distributions, respectively. Based on the given data set, we have calculated, $X_1=1869$, $X_2=1539$. Therefore, we select the plane 7913 as the ``worst" and 7914 as the ``best" plane. We obtain the values of $Z_1$, $Z_2$, $T$ and $V$ as $Z_1=1539$, $Z_2=1869$, $T=0.7420444$ and $V=1.347628$. For estimating $H_S(\underline{\theta})$ (the Shannon entropy of the selected plane ``7913") under the squared error loss function, various estimates are provided in the following table:
\begin{center} 
	\textbf{Table 1:} Various estimates of $H_S(\underline{\theta})$ \\
	
	\begin{tabular}{|c|c|c|c|c|}
		
		\hline
		$\mathbf{\delta_{\ln \alpha}}$ & $\mathbf{\delta_{\ln (\alpha+1)}}$ &  $\mathbf{\delta_{\psi(\alpha)}}$ & $\mathbf{\delta^{(S)}_{\ln \alpha}}$ &  $\mathbf{\delta^{(S)}_{\ln (\alpha+1)}}$ \\ 
		\hline
		4.355105 & 4.314283 &4.376083& 4.355&4.314\\ \hline	
	\end{tabular}	
\end{center}
For estimating $H_M(\underline{\theta})$ (the Shannon's entropy of the selected plane ``7914") under the squared error loss function, various estimates are provided in the following table:
\begin{center} 
	\textbf{Table 2:} Various estimates of $H_M(\underline{\theta})$ \\
	
	\begin{tabular}{|c|c|c|c|c|}
		
		\hline
		$\mathbf{d_{\ln \alpha}}$ & $\mathbf{d_{\ln (\alpha+1)}}$ &  $\mathbf{d_{\psi(\alpha)}}$ & $\mathbf{d^{(S)}_{\ln \alpha}}$ &  $\mathbf{d^{(S)}_{\ln (\alpha+1)}}$ \\
		\hline
		 4.160834 &  4.120012 & 4.181812 & 4.27313 & 4.27313\\ \hline	
	\end{tabular}	
\end{center}
It is recommended to use $\delta_{\psi}$ ($d_{\psi}$) for estimating the selected Shannon entropy $H_S(\underline{\theta})$($H_M(\underline{\theta})$) of the failure times of the air conditioning systems of Boeing jet planes.
\section{Final Remarks}
The Shannon entropy is a useful measure of uncertainty that has applications in a variety of fields including wireless communication, weather forecasting, economic modelling, molecular biology, and so on. Estimation of the selected Shannon entropy is an important practical problem in these areas.
In the present article, we have focussed on efficient estimation of the Shannon entropy of the selected gamma population and proposed several estimators. We have derived various admissibility results for a class of naive estimators of the selected entropy and also obtained shrinkage estimators dominating upon various naive estimators. Although, our simulation study suggests that the generalized Bayes estimator with respect to the Jeffreys non informative prior is admissible within the class of scale and permutation equivariant estimators of $H_S(\underline{\theta})$ (or $H_M(\underline{\theta})$) under the mean squared error criterion, we have not been able to prove it. From our analysis, we believe that there does not exist any unbiased estimator of the selected Shannon entropy $H_S(\underline{\theta})$ (or $H_M(\underline{\theta})$), but, unfortunately, we have not been able to prove this result too. It will also be interesting to investigate whether the results obtained in this paper can be adapted to $k ~(\geq 2)$ populations. These are some interesting questions for future research.
\bibliographystyle{apalike}
\bibliography{biblography}

\begin{thebibliography}{}

\bibitem[Abramowitz and Stegun, 1970]{abramowitz1964handbook}
Abramowitz, M. and Stegun, I.~A. (1970).
\newblock {\em Handbook of mathematical functions with formulas, graphs, and
  mathematical tables}, volume~55.
\newblock US Government printing office.

\bibitem[Arshad et~al., 2015]{arshad2015estimation}
Arshad, M., Misra, N., and Vellaisamy, P. (2015).
\newblock Estimation after selection from gamma populations with unequal known
  shape parameters.
\newblock {\em Journal of Statistical Theory and Practice}, 9(2):395--418.

\bibitem[Bahadur and Goodman, 1952]{bahadur1952impartial}
Bahadur, R.~R. and Goodman, L.~A. (1952).
\newblock Impartial decision rules and sufficient statistics.
\newblock {\em The Annals of Mathematical Statistics}, pages 553--562.

\bibitem[Brewster and Zidek, 1974]{brewster1974improving}
Brewster, J.-F. and Zidek, J. (1974).
\newblock Improving on equivariant estimators.
\newblock {\em The Annals of Statistics}, 2(1):21--38.

\bibitem[Cohen and Sackrowitz, 1982]{cohen1982estimating}
Cohen, A. and Sackrowitz, H. (1982).
\newblock Estimating the mean of the selected population.
\newblock {\em Statistical Decision Theory and Related Topics III}, 1:243--270.

\bibitem[Eaton, 1967]{eaton1967some}
Eaton, M.~L. (1967).
\newblock Some optimum properties of ranking procedures.
\newblock {\em The Annals of Mathematical Statistics}, 38(1):124--137.

\bibitem[Hsieh, 1981]{hsieh1981estimating}
Hsieh, H.-K. (1981).
\newblock On estimating the mean of the selected population with unknown
  variance.
\newblock {\em Communications in Statistics-Theory and Methods},
  10(18):1869--1878.

\bibitem[Hwang, 1993]{hwang1993empirical}
Hwang, J.~T. (1993).
\newblock Empirical bayes estimation for the means of the selected populations.
\newblock {\em Sankhy{\=a}: The Indian Journal of Statistics, Series A}, pages
  285--304.

\bibitem[James and Stein, 1961]{jamesstein1}
James, W. and Stein, C. (1961).
\newblock Estimation with quadratic loss.
\newblock In {\em Proc. 4th {B}erkeley {S}ympos. {M}ath. {S}tatist. and
  {P}rob., {V}ol. {I}}, pages 361--379. Univ. California Press, Berkeley,
  Calif.

\bibitem[Misra, 1994]{misra1994estimation}
Misra, N. (1994).
\newblock Estimation of the average worth of the selected subset of gamma
  populations.
\newblock {\em Sankhy{\=a}: The Indian Journal of Statistics, Series B}, pages
  344--355.

\bibitem[Misra and Arshad, 2014]{misra2014selecting}
Misra, N. and Arshad, M. (2014).
\newblock Selecting the best of two gamma populations having unequal shape
  parameters.
\newblock {\em Statistical Methodology}, 18:41--63.

\bibitem[Misra and Dhariyal, 1994]{misra1994non}
Misra, N. and Dhariyal, I.~D. (1994).
\newblock Non-minimaxity of natural decision rules under heteroscedasticity.
\newblock {\em Statistics and Decisions}, (12):79--98.

\bibitem[Proschan, 1963]{proschan1963theoretical}
Proschan, F. (1963).
\newblock Theoretical explanation of observed decreasing failure rate.
\newblock {\em Technometrics}, 5(3):375--383.

\bibitem[Putter and Rubinstein, 1968]{165}
Putter, J. and Rubinstein, D. (1968).
\newblock On estimating the mean of a selected population.
\newblock Technical report, Department of Statistics, The University of
  Wisconsin Madison, Wisconsin.

\bibitem[Qomi et~al., 2015]{qomi2015admissibility}
Qomi, M.~N., Nematollahi, N., and Parsian, A. (2015).
\newblock On admissibility and inadmissibility of estimators after selection
  under reflected gamma loss function.
\newblock {\em Hacettepe Journal of Mathematics and Statistics},
  44(5):1109--1124.

\bibitem[Sackrowitz and Samuel-Cahn, 1986]{sackrowitz1986evaluating}
Sackrowitz, H. and Samuel-Cahn, E. (1986).
\newblock Evaluating the chosen population: a bayes and minimax approach.
\newblock {\em Lecture Notes-Monograph Series}, pages 386--399.

\bibitem[Sarkadi, 1967]{sarkadi1967estimation}
Sarkadi, K. (1967).
\newblock Estimation after selection.
\newblock {\em Studia Scientarium Mathematicarum Hungarica}, 2:341--350.

\bibitem[Shannon, 1948]{shannon1948l}
Shannon, C.~E. (1948).
\newblock A mathematical theory of communication.
\newblock {\em The Bell system technical journal}, 27(3):379--423.

\bibitem[Stein, 1956]{stein1}
Stein, C. (1956).
\newblock Inadmissibility of the usual estimator for the mean of a multivariate
  normal distribution.
\newblock In {\em Proceedings of the {T}hird {B}erkeley {S}ymposium on
  {M}athematical {S}tatistics and {P}robability, 1954--1955, vol. {I}}, pages
  197--206. University of California Press, Berkeley-Los Angeles, Calif.

\bibitem[Stein, 1964]{stein2}
Stein, C. (1964).
\newblock Inadmissibility of the usual estimator for the variance of a normal
  distribution with unknown mean.
\newblock {\em Ann. Inst. Statist. Math.}, 16:155--160.

\bibitem[Stein, 1974]{stein3}
Stein, C. (1974).
\newblock Estimation of the mean of a multivariate normal distribution.
\newblock In {\em Proceedings of the {P}rague {S}ymposium on {A}symptotic
  {S}tatistics ({C}harles {U}niv., {P}rague, 1973), {V}ol. {II}}, pages
  345--381.

\bibitem[Vellaisamy, 1992]{vellannals1992inadmissibility}
Vellaisamy, P. (1992).
\newblock Inadmissibility results for the selected scale parameters.
\newblock {\em The Annals of Statistics}, 20(4):2183--2191.

\bibitem[Vellaisamy, 1993]{vellaisamy1993umvu}
Vellaisamy, P. (1993).
\newblock On umvu estimation following selection.
\newblock {\em Communications in Statistics--Theory and Methods},
  22(4):1031--1043.

\end{thebibliography}

\end{document}